\documentclass[11pt]{article}

\usepackage[utf8]{inputenc}
\usepackage[margin=1in]{geometry}

\usepackage{amsmath}
\usepackage{amsthm}
\usepackage{amssymb}
\usepackage{amsfonts}
\usepackage{mathtools}
\usepackage{bm}
\usepackage{graphicx}
\usepackage{float}
\usepackage{subcaption}
\usepackage{booktabs}
\usepackage{multirow}
\usepackage{algorithm}
\usepackage{algorithmic}
\usepackage{enumitem}

\usepackage{url,hyperref}
\usepackage[dvipsnames]{xcolor}
\usepackage{placeins}
\usepackage{threeparttable}

\usepackage[capitalise]{cleveref}

\usepackage[round]{natbib}
\usepackage[sc]{mathpazo}

\newtheorem{theorem}{Theorem}
\newtheorem{proposition}{Proposition}
\newtheorem{lemma}{Lemma}
\newtheorem{corollary}{Corollary}
\newtheorem{definition}{Definition}
\newtheorem{assumption}{Assumption}
\newtheorem{standingass}{Standing Assumption}
\newtheorem{remark}{Remark}

\newcommand{\tran}{^\top}
\newcommand{\R}{\mathbb{R}}
\DeclareMathOperator{\interior}{int}
\DeclareMathOperator{\Proj}{Proj}

\DeclareMathOperator*{\argmax}{arg\,max}

\newcommand{\normM}[1]{\|#1\|_M}

\title{Perturbed utility Markovian traffic equilibrium:\\ theory and computation}
\author{Rui Yao\thanks{Technion -- Israel Institute of Technology, \texttt{rui.yao@technion.ac.il}} \and Kenan Zhang\thanks{\'Ecole Polytechnique F\'ed\'erale de Lausanne (EPFL), \texttt{kenan.zhang@epfl.ch}}}
\date{}

\begin{document}

\maketitle

\begin{abstract}
Large-scale traffic assignment requires equilibrium models that are both behaviorally plausible and computationally tractable. 
This paper develops a perturbed utility Markovian equilibrium (PUME) framework that preserves the scalability of link-based Markovian traffic equilibrium models and extends their applicability to settings with boundary choice probabilities, undiscounted network loading, and general link interactions. 
As the behavioral basis of PUME, we first develop the perturbed utility Markovian choice model (PUMCM) in which the Bellman optimality operator is defined through a convex surplus function whose gradient directly yields the optimal policy.
The model generalizes existing additive random utility (ARUM) Markovian choice models and admits both interior and boundary choice probabilities. Accordingly, unattractive links can receive zero flow without imposing ex ante choice-set restrictions as in existing ARUM models. 
We establish conditions under which the corresponding Markov decision problem is well posed and yields a proper demand mapping. We then formulate the equilibrium as a variational inequality (VI) problem on the dual cost space and establish its existence and uniqueness. Particularly, the VI formulation of PUME accommodates non-separable and asymmetric cost structures and thus offers a more flexible modeling framework than existing Markovian traffic equilibrium (MTE) models. 
% This serve as the demand foundation for the proposed PUME. We formulate the resulting equilibrium problem as a variational inequality that accommodates nonseparable and asymmetric cost structures.
For computation, we develop a modified policy iteration method for network loading and a safeguarded accelerated meta-algorithm for computing equilibrium. Both algorithms are proven to be globally convergent and have demonstrated satisfactory numerical performances. 
Experiments on benchmark and synthetic networks further show that the proposed framework is highly scalable and robust towards a wide variety of demand-supply settings.
\end{abstract}


%%%%%%%%%%%%%%%%%%%%%%%%%%%%%%%%%%%%%%%%%%%%%%%%%%%%%%%%%%%%%%%%%%%%%%
%%                       SECTION 1: INTRODUCTION
%%%%%%%%%%%%%%%%%%%%%%%%%%%%%%%%%%%%%%%%%%%%%%%%%%%%%%%%%%%%%%%%%%%%%%
\section{Introduction}\label{sec:intro}

Traffic equilibrium models are used to predict network flow patterns and to support system design and management.
The classical deterministic user equilibrium (UE) builds on shortest-path routing behavior under flow-dependent link costs~\citep{sheffi1985urban}, while
stochastic user equilibrium (SUE)~\citep{daganzo1977stochastic} extends this framework to incorporate unobserved heterogeneity through random utility route choice models~\citep{benakiva1985discrete}. 
% In large networks, the computational bottleneck is typically the network loading given the routing decisions and its interaction with the routing updates~\citep{dial1971probabilistic}.
For equilibrium prediction, a route choice model must be behaviorally expressive, statistically consistent, and computationally scalable.
Behavioral expressiveness is needed to capture realistic substitution patterns among overlapping routes; statistical consistency requires the model used for prediction to be compatible with estimation; and computational tractability matters for large-scale applications.
Markovian formulations address these by modeling routing as sequential link choices in a Markov decision process (MDP). 
Accordingly, the full choice set is embedded in the network, as well as the substitution pattern~\citep{mai2015nested, oyama2022markovian}. 
Previous works also demonstrate that the parameters can be consistently estimated using Markovian models~\citep{fosgerau2013link,yaoperturbed,fosgerau_yao_MPURC_2026}. Moreover, the link-based formulation avoids explicit path enumeration while retaining closed-form choice probabilities and addresses the computational challenge~\citep{akamatsu1997decomposition,fosgerau2013link}.
% , avoiding explicit path enumeration while retaining closed-form choice probabilities~\citep{akamatsu1997decomposition,fosgerau2013link}. 
% This means the network itself defines a full choice set.
% Consequently, substitution pattern emerges directly from the network structure~\citep{mai2015nested, oyama2022markovian}, and parameters can be consistently estimated~\citep{fosgerau2013link,yaoperturbed,fosgerau_yao_MPURC_2026}.
In terms of equilibrium modeling,
\cite{baillon2008markovian} proposes a dual formulation of Markovian traffic equilibrium (MTE) on the link-cost space that scales with the network rather than the path set~\citep{baillon2008markovian}. Its extensions based on network generalized extreme value (GEV) models accommodate richer correlation structures~\citep{oyama2022markovian}.

Despite this progress, three issues limit the applicability of MTE in general traffic equilibrium modeling and computation. 
% how broadly Markovian models can serve as a reliable building block for traffic equilibrium modeling and computation.
First, Markovian route choice models used in existing MTE formulations are mostly based on additive random utility models (ARUM), and therefore yield full support of choice probabilities, i.e., every feasible outgoing link is chosen with strictly positive probability~\citep{mcfadden1981econometric,fosgerau2013link, mai2015nested}. 
As a result, even dominated links carry positive flow, leading to unreasonably long detours unless they are excluded ex ante. In many applications, however, the corner solutions (zero-probability actions) are preferable as they generate more plausible behavioral predictions and practical insights~\citep{fosgerau2022perturbed, tan2024endogenously, yao2024perturbed}.
% Full support means every feasible outgoing link is chosen with strictly positive probability~\citep{fosgerau2013link, mai2015nested}, so dominated links can still carry positive flow, leading to implausible long detours, unless they are excluded ex ante.
% Admitting corner solutions (zero-probability actions) requires a Markovian choice framework whose theory and computational guarantees remain valid on the boundary of the simplex, which existing MTE theory does not provide.
Second, existing solution algorithms based on the dual formulation of MTE repeatedly solve the network loading via value iterations. 
% solve a fixed-point problem in the network loading step , while 
% the function evaluation in the fixed-point iterations relies on dynamic programming (DP). 
% solve a dynamic programming (DP) fixed-point problem for network loading. 
Despite its algorithmic simplicity, the corresponding Bellman operator is only guaranteed to be nonexpansive but not contractive in the undiscounted case ($\gamma=1$), thus the classic value iteration may fail to converge and thus break the network loading~\citep{mai2022undiscounted}. 
% This issue causes particular challenges in MTE computation. 
% so value iteration can fail even when a unique fixed point exists~\citep{mai2022undiscounted}.
% When corner solutions are admitted, policies may assign zero probability to some actions; without additional structure, a policy that concentrates flow on cycles may fail to reach the destination, breaking transience and invalidating standard existence and convergence arguments.
Third, existing MTE formulations impose strong assumptions (e.g., symmetric Jacobian of link costs) so that the equivalent convex optimization problem exists~\citep{baillon2008markovian,oyama2022markovian}.
% Third, the convex optimization formulation of MTE relies on a potential structure, which in turn requires an integrable (symmetric-Jacobian) supply mapping~\citep{baillon2008markovian,oyama2022markovian}.
However, many settings of interest in transportation, such as multi-class interactions and non-separable cost structures, are naturally asymmetric and thus cannot be modeled in the classic MTE framework. On the other hand, the variational inequality (VI) formulation provides a more flexible expression of the equilibrium conditions~\citep{dafermos1980traffic,smith1979existence}, though it has not yet been adopted in MTE. 
% Many settings of interest, such as multi-class interactions, pricing/control, and nonseparable cost specifications, are naturally asymmetric, where a variational inequality (VI) formulation is the appropriate equilibrium model~\citep{dafermos1980traffic,smith1979existence}, yet the Markovian traffic equilibrium literature has not adopted VI formulations.

Motivated by these gaps, this paper develops a \textit{perturbed utility Markovian equilibrium} (PUME) framework. 
% that consists of three key components: 
% i) a perturbed utility Markovian choice model (PUMCM) that admits both interior and corner solutions, ii) a VI formulation of equilibrium defined on the general cost space that accommodates both symmetric and asymmetric cost structures, and iii) efficient algorithms with global convergence guarantees. 
As the foundation of PUME, we propose a perturbed utility Markovian choice model (PUMCM) that characterizes the stochastic route choice as an undiscounted MDP regularized by a convex perturbation function. 
Accordingly, the Bellman optimality operator is defined by a surplus (choice-probability generating) function, whose gradient yields (link) choice probabilities at each state (node). We further derive conditions under which the induced choice map is well defined on the full simplex and able to produce corner solutions (zero choice probabilities).
PUMCM essentially generalizes the ARUM distributional assumption in classic MTE~\citep{hofbauer2002global,fosgerau2012theory} and covers choice behaviors beyond the ARUM class. It also extends the static perturbed utility route choice (PURC) model of \citet{fosgerau2022perturbed,yao2024perturbed} to the Markovian setting, while retaining its scalability.
% while retaining the scalability of recursive models.
% On the demand side, we define the perturbed utility Markovian choice model (PUMCM), an undiscounted MDP in which a convex perturbation function replaces the ARUM distributional assumption~\citep{hofbauer2002global,fosgerau2012theory}.
% The Bellman recursion is written in terms of a surplus (choice-probability generating) function $H_s$ at each state, and we impose mild conditions under which the induced choice map is well defined on the full simplex and can yield corner solutions without excluding actions ex ante.
% This demand model contains recursive logit, nested logit, and network GEV models as special cases and allows choice probability patterns outside the ARUM class.
To efficiently solve PUMCM, and accordingly the network loading, we derive conditions that ensure the policy evaluation is well-posed in the undiscounted setting, which is typically more challenging than the discounted counterpart, and propose a modified policy iteration (MPI) with a global convergence guarantee. 
% On the DP side, we develop a Bellman theory that makes policy evaluation well posed in undiscounted cases. We establish that, under mild conditions, any policy consistent with a Bellman equation will be \textit{proper}~\citep{bertsekas2012dynamic}, which in turn yields a globally convergent modified policy iteration (MPI) routine for network loading.
On top of PUMCM, we formulate PUME as a VI problem in the general cost space with a monotone supply mapping. We establish the equilibrium existence and uniqueness under mild conditions, and further develop an efficient and globally convergent meta-algorithm that combines the first-order VI method with safeguarded acceleration.

% as a cost-space VI with a monotone supply mapping, which accommodates asymmetric congestion models as well as symmetric (potential) cases.
% We establish equilibrium existence and uniqueness under mild conditions and develop globally convergent meta-algorithms by combining first-order VI methods with safeguarded acceleration.

Our contributions are summarized as follows.
\begin{enumerate}
    \item \textit{Perturbed utility Markovian choice model.}  
    We develop PUMCM that generalizes existing ARUM-based Markovian choice models and further admits both interior and boundary choice probabilities. Particularly, we characterize how the regularity of surplus function yields corner solutions and propagates through value functions, policies, and induced flows, whose well-posedness is essential for convergence.

    \item \textit{Globally convergent undiscounted PUMCM network loading.}
    Using structural properties of PUMCM, we establish the existence and uniqueness of the optimal value function and prove global convergence of modified policy iteration (MPI) for the whole model class. We further obtain a tunable local linear rate whose factor depends on the MPI evaluation depth.

    \item \textit{Perturbed utility variational Markovian equilibrium and globally convergent solution algorithms.}
    We extend the classic optimization formulation of MTE by constructing PUME as a monotone VI problem that accommodates asymmetric link interactions and cost structures. For the computation of equilibria, we adopt a meta-algorithm with a customized merit-function safeguard, prove its global convergence, and demonstrate its efficiency on several benchmark networks. 
    % We extend the cost-space equilibrium characterization from potential settings to a monotone VI that 
    % For computation, we combine monotone VI methods with acceleration oracles under a merit-based safeguard and prove global convergence.
\end{enumerate}

The remainder of the paper is organized as follows.
Section~\ref{sec:lit} reviews related work and positions PUME relative to existing models.
Section~\ref{sec:PUMCM} defines PUMCM, establishes its main properties, and derives its induced network flows (demand). 
% as a  Markov network model, the surplus conditions, and derives the PUMCM demand function.
Section~\ref{sec:equilibrium} first formulates PUME, establishes existence and uniqueness, and then establishes its equivalent variational Nash equilibrium.
Section~\ref{sec:algorithms} presents the solution algorithms, including the MPI for network loading and the meta-algorithm for equilibrium update, along with their convergence analysis.
% MPI for network loading and the meta-algorithm for equilibrium computation.
Section~\ref{sec:spec} discusses how to specify key components of PUME in practice and provides examples of equilibrium outcomes under different settings. 
% using several surplus functions to show the induced choice behavior, the interaction with the network supply map, and the performance of the proposed algorithms.
Section~\ref{sec:experiments} conducts experiments to test different algorithm configurations and demonstrate the scalability and robustness on benchmark transportation networks, along with a series of sensitivity analyses. 
% on benchmark transportation networks and sensitivity analyses on scalability and cost symmetry. 
% potential and non-potential supply settings.
Section~\ref{sec:conclusion} finally concludes this study with several future directions. 

\section{Related literature}\label{sec:lit}

\subsection{Stochastic traffic equilibrium}\label{sec:lit:path_to_link}
Same as classic traffic equilibrium, stochastic user equilibrium (SUE) can be interpreted as two components coupled via costs: i) a \textit{supply}  mapping that translates link costs into flow supply (e.g., the inverse of a link cost function), and ii) a \textit{demand} (route choice) mapping that generates flow from a given cost vector via a stochastic choice model~\citep{daganzo1977stochastic,sheffi1985urban}.
Specifically, the supply mapping captures the congestion effect. When the cost structure satisfies certain conditions (e.g., symmetric Jacobian), the supply mapping is integrable and leads to the Beckmann-type optimization formulations~\citep{daganzo1977stochastic,vovsha_link-nested_1998,zhou2012c,kitthamkesorn2013path}. In the more general settings (e.g., heterogeneous users, non-separable and/or asymmetric link costs), the supply mapping is non-integrable and thus the variational inequality (VI) formulations are adopted~\citep{dafermos1980traffic, facchinei2003finite}.

On the other hand, the demand mapping in SUE replaces deterministic routing behavior with stochastic route choice to capture unobserved heterogeneous preferences and perception errors~\citep{benakiva1985discrete}. One commonly used demand mapping is the path-based additive random utility model (ARUM), where each route is treated as an alternative. 
However, the path-based formulations constantly face two technical challenges. The first one is the \emph{path correlation due to overlap}. Since many paths share links and subpaths, the random error in a choice model should reflect such correlation across alternatives. As a result, models that rely on the independence of irrelevant alternatives (IIA) assumption, such as the multinomial logit (MNL) model, can yield implausible path substitution patterns~\citep {ben1999extended,fosgerau2013link,kitthamkesorn2013path}, while those with structured error specifications~\citep[e.g.][]{bekhor2001stochastic} and systematic utility corrections~\citep[e.g.,][]{kitthamkesorn2013path,duncan2020path} require apriori assumptions on the correlation. 
The second challenge regards the \emph{choice-set representation}. All path-based models must define and maintain a set of feasible routes, which can grow exponentially with the network size~\citep{bekhor2005investigating}. The two challenges intertwine and together lead to the major computational bottleneck in large-scale equilibrium models~\citep{dial1971probabilistic,vovsha_link-nested_1998}. 

To avoid explicit path enumeration or iterative path-generation heuristics, link-based models are proposed, which decompose route choice into sequential link choices toward the destination. These models naturally fall into the Markovian framework, where each traveler's route corresponds to a sample path of a Markov chain induced by the routing policy~\citep{dial1971probabilistic, bell1995alternatives,akamatsu1996cyclic, akamatsu1997decomposition}. Particularly, the recursive logit model and its extensions model the route choice as a special type of undiscounted Markov decision process (MDP), where the Bellman optimality equation directly yields closed-form link choice probabilities~\citep{fosgerau2013link,mai2015nested,oyama2022markovian}. 
Although overcoming the correlation and scalability issues of path-based models, the link-based Markovian models induce two additional problems. First, the current ARUM-based models have full support for choice probabilities. Consequently, a path with extensive long detours could still observe a positive flow unless ruled out ex ante~\citep{mcfadden1981econometric,fosgerau2013link,oyama2022markovian,oyama2019prism}. 
Secondly, since the underlying MDP is undiscounted, the Bellman optimality operator is generally non-contractive, thus value iterations may not converge even though a unique fixed point exists~\citep{mai2022undiscounted}.

The PUME proposed in this study takes the same supply--demand viewpoint, also known as the dual formulation~\citep{baillon2008markovian}, but makes extensions on both sides. 
On the supply side, we relax the strong separable and symmetric conditions imposed in classic MTE models and define the supply mapping in a general cost space. 
On the demand side, we propose PUMCM that generalizes the ARUM-based models and admits corner solutions. 
We further derive conditions that ensure well-posedness of the Bellman equation that enables the convergence of modified policy iteration~\citep{puterman1978modified} in the undiscounted MDP setting.

\subsection{Perturbed utility theory and its network applications}\label{sec:lit:pu}

Perturbed utility models (PUM) replace the distributional assumption in classic random utility choice models with a convex perturbation function~\citep{hofbauer2002global,fosgerau2012theory, fudenberg2015stochastic, allen2019satisficing}.
% on additive random shocks with a convex perturbation function~\citep{hofbauer2002global,fosgerau2012theory, fudenberg2015stochastic, allen2019satisficing}.
For a static setting where individuals choose among alternatives with systematic utilities $V\in\mathbb{R}^{|\mathcal{A}|}$, PUM defines choice probabilities as the unique solution of a perturbed expected-utility maximization,
\begin{align}\label{eq:PUM_static}
    p_*(V) = \arg\max_{p\in\Delta}\ \langle p, V\rangle-F^*(p),
\end{align}
where $F^*$ is the convex perturbation function~\citep{Rockafellar+1970}.

Problem~\ref{eq:PUM_static} is associated with a conjugate (surplus) function $F(V)=\sup_{p\in\Delta}\{\langle p,V\rangle-F^*(p)\}$ that is convex in utility $V$, and, under standard regularity assumptions, induces a closed-form choice mapping $p_*(V)=\nabla F(V)$\footnote{Hereafter, we use $*$ as subscript in the optimal value and policy to distinguish superscript $*$ in the perturbation function.}, which generalizes the classical Williams-Daly-Zachary theorem for ARUM~\citep{mcfadden1981econometric}.
This convex-dual representation therefore provides a unified framework for discrete choice~\citep{fosgerau2012theory}. Essentially, PUM is a strict generalization of ARUM: while every ARUM admits a PUM representation (e.g., MNL corresponds to the Shannon-entropy perturbation), PUM can characterize behavioral phenomena outside the ARUM class, 
% PUM can produce comparative statics outside the ARUM class, 
such as complementarity effects where increasing one alternative's utility raises the choice probability of another~\citep{hofbauer2002global, fosgerau2026sensitivity}. 
In addition, PUM accommodates perturbations that yield corner solutions, rather than full support restricted in the ARUM class.

PUM has recently been introduced to transportation research. 
\citet{fosgerau2022perturbed} propose a perturbed utility route choice (PURC) model in which origin-destination (OD)-based flows are characterized as the solution to a convex program over the feasible flow polyhedron with a link-additive perturbation.
% Notably, the link-additive structure induces substitution patterns directly from the network topology, and the perturbation function permits zero flow on sufficiently costly links, removing the need for exogenous choice set construction.
Follow-up studies show that certain truncated ARUM-based choice models, which impose zero flow on sufficiently costly links, have equivalent formulations in the framework of PURC and thus provide the corner solutions of PURC with an interpretation of endogenous choice set generation~\citep{tan2024endogenously,tan2026endogenous}.
% Moreover, \cite{tan2024endogenously,tan2026endogenous} show that certain truncated ARUM-based choice models, where sufficiently costly alternatives receive zero flow, also take the form of a PURC, providing an endogenous choice-set interpretation of corner solutions.
\citet{yao2024perturbed} propose a new SUE based on PURC and develop a fast solution algorithm. 
% embed PURC in a traffic equilibrium setting and develop a fast solution algorithm.

PUMCM proposed in this study can be viewed as the dynamic counterpart of PURC in the Markovian setting, or a variant of dynamic discrete choice models~\citep{rust1987optimal, chiong2016duality} with perturbed utility.
% More broadly, PUMCM can be viewed as a dynamic perturbed utility model in the tradition of dynamic discrete choice~\citep{rust1987optimal}, specialized to absorbing-MDP network routing.
% The present paper develops the dynamic counterpart of PURC in this Markovian setting.
% The perturbed utility Markovian choice model (PUMCM) replaces the maximum expected utility operator at each state of the Bellman recursion (e.g., the log-sum-exp operator in the recursive logit special case) with a general surplus function satisfying mild convexity conditions (see details in Section~\ref{sec:PUMCM}).
% Combined with structural MDP assumptions, these conditions yield a well-defined Bellman operator with a unique fixed point and a continuously differentiable optimal policy map.
Similar to the general PUM, PUMCM strictly generalizes existing ARUM-based dynamic models (e.g., recursive logit), and importantly, accommodates corner solutions when particular perturbation functions are applied. 
% PUMCM strictly contains logit, nested logit, and network GEV as special cases, and, crucially, also accommodates sparse perturbations (Tsallis-entropy, sparsemax) that assign zero probability to dominated alternatives, thereby producing corner solutions that are not supported by existing recursive choice models.

% A distinct strand of recent work in machine learning replaces the softmax (logit) mapping with sparse alternatives that map some probability mass exactly to zero.
% The $\alpha$-entmax family~\citep{peters2019sparse, correia2019adaptively} generalizes softmax through a Tsallis entropy perturbation parameterized by $\alpha\ge 1$: softmax (logit) corresponds to $\alpha=1$, while $\alpha>1$ produces sparse choice probabilities.
% The limiting case $\alpha=2$ is sparsemax (quadratic perturbation)~\citep{martins2016softmax}, whose choice map is the Euclidean projection onto the simplex.
% These mappings are well-defined PUM, but to our knowledge they have not been analyzed in network equilibrium settings.
% Their boundary behavior (corner solutions) falls outside the standard Legendre-type regularity~\citep{Rockafellar+1970} assumed in the existing Markovian equilibrium literature, motivating the general conditions for surplus functions developed in Section~\ref{sec:PUMCM}.

% \subsection{Equilibrium algorithms and acceleration}\label{sec:lit:algo}
\subsection{Solution algorithms for traffic equilibrium}\label{sec:lit:algo}
The solution procedure of traffic equilibrium with dual formulations typically alternates between network loading (i.e., executing the demand mapping from costs to flows) and cost update~\citep{sheffi1985urban,patriksson1994traffic}. For Markovian route choice models, the network loading requires solving the optimal value as the solution to a fixed-point problem defined on the Bellman optimality operator. As discussed in Section~\ref{sec:lit:path_to_link}, the Bellman optimality operator is generally non-expansive but not contractive. Therefore, value iteration may fail to converge in certain scenarios~\citep{mai2022undiscounted}. 
% For general MDP problems, modified policy iteration (MPI) \citep{puterman1978modified,puterman1994markov} has been well established in the discounted setting, though its convergence in undiscounted MDPs, especially for the cases with corner solutions, has not been analyzed. 
This study tackles this issue by deriving conditions that ensure well-posed policy evaluation and proposes a modified policy iteration (MPI) method~\citep{puterman1978modified,puterman1994markov} to solve the optimal value with a global convergence guarantee. Notably, MPI has only been discussed in the discounted setting in the literature. To the best of our knowledge, its convergence in the undiscounted setting of PUMCM, especially for the cases with corner solutions, is first analyzed in this study.

As for the cost update, existing algorithms mainly hinge on whether an equivalent optimization exists, or equivalently, whether the supply mapping can be integrated into a potential function. 
When a potential function exists, the gradient-based methods apply~\citep[e.g.,][]{powell1982convergence}, and acceleration schemes have been proposed for specific formulations~\citep{oyama2022markovian,yao2024perturbed}. 
% MSA and related averaging schemes \citep{powell1982convergence} are robust but can be slow, and several works report substantial improvements from acceleration in specific formulations, including momentum-based methods for network-GEV MTE~\citep{oyama2022markovian} and quasi-Newton acceleration for PURC \citep{yao2024perturbed}.
% Dual formulations are often attractive in these settings because they replace OD-specific link flows by link costs which only subject to simple lower-bound constraints~\citep{baillon2008markovian}.
Otherwise, the equilibrium is often expressed as a VI and solved with projection-type methods~\citep{dafermos1980traffic,smith1979existence,facchinei2003finite}, which all belong to the first-order VI methods.
The structure of PUME allows for a more flexible cost structure, thus requiring a VI formulation. To improve the computational efficiency, we develop a meta-algorithm that combines the first-order VI method with safeguarded acceleration. We further prove that the meta-algorithm is globally convergent for PUME and thus provide a theoretical guarantee for the solution quality in addition to its practical performance.
\section{Perturbed utility Markovian choice model}\label{sec:PUMCM}
This section defines the perturbed utility Markovian choice model (PUMCM), which generalizes existing ARUM-based Markovian choice models by replacing the distributional assumption on random utility with a convex perturbation function.

\subsection{Choice problem as an absorbing MDP}\label{sec:PUMCM:setting}
We model a traveler's route choice toward a destination $d$ as a finite-state, undiscounted MDP with a single absorbing termination state $d$. Let
non-terminal states $\mathcal{S}$ represent network nodes, and actions $\mathcal{A}_s$ represent the feasible outgoing links at state $s$.
% A policy $\pi$ specifies, for each $s \in \mathcal{S}$, a distribution $\pi(\cdot|s)$ over $\mathcal{A}_s$.
Besides, $P(\cdot|s,a)$ denotes the state transition probabilities among non-terminal states $\mathcal{S}$ for a state-action pair $(s, a)$, and $\tilde P(\cdot|s,a)$ denotes the transition defined on the extended state space $\mathcal{S}\cup\{d\}$. Accordingly, we have $\tilde P(s'|s,a) = P(s'|s,a)$ for non-terminal state $s'\in\mathcal{S}$ and $\tilde P(d|s,a) = 1 - \sum_{s'\in\mathcal{S}} P(s'|s,a)$ for the destination $d$.
% Write $P(s'|s,a)\coloneqq \tilde P(s'|s,a)$ for its restriction to nonterminal states $s' \in \mathcal{S}$; the residual probability $\tilde P(d|s,a) = 1 - \sum_{s'\in\mathcal{S}} P(s'|s,a)$ is absorbed at the destination $d$.

For each $s\in\mathcal{S}$, we collect the non-terminal state transitions into a compact matrix form $P(\cdot|s,\cdot)\in\R^{|\mathcal{A}_s|\times|\mathcal{S}|}$, whose $(a,s')$-entry is $P(s'|s,a)$.
% we collect these nonterminal transition probabilities into a matrix $P(\cdot|s,\cdot)\in\R^{|\mathcal{A}_s|\times|\mathcal{S}|}$ whose $(a,s')$-entry is $P(s'|s,a)$.
Similarly, we define the vector form of rewards  $u(s,\cdot)\in\R^{|\mathcal{A}_s|}$, where each element $u(s,a) \in \R$ denotes the one-step reward (instantaneous utility) for each state--action pair, and the vector of values $V \in \R^{|\mathcal{S}|}$. Accordingly, the state-action value at each state $Q_s: \R^{|\mathcal{S}|}\to \R^{|\mathcal{A}_s|}$ is given by 
\begin{equation}\label{eq:Q_def}
    Q_s(V) \coloneqq u(s,\cdot) + P(\cdot|s,\cdot)\,V \;\in\; \R^{|\mathcal{A}_s|}, \qquad s \in \mathcal{S}.
\end{equation}

% Let $\Delta_s \coloneqq \{\pi(\cdot|s) \in \R_+^{|\mathcal{A}_s|} : \mathbf{1}\tran \pi(\cdot|s) = 1\}$ denote the probability simplex on $\mathcal{A}_s$. 
Given a policy $\pi$ that maps from each $s \in \mathcal{S}$ to a distribution $\pi(\cdot|s)$ over $\mathcal{A}_s$, i.e., $\pi(\cdot|s)\in \Delta_s \coloneqq \{x \in \R_+^{|\mathcal{A}_s|}: \mathbf{1}\tran x = 1\}$, we further define the policy induced transition matrix $P_\pi \in \R^{|\mathcal{S}| \times |\mathcal{S}|}$ that maps between non-terminal states with entries. Specifically, each entry in $P_\pi$ is given by
\begin{align}\label{eq:policy_transition}
    P_\pi(s,s') \coloneqq \sum_{a\in\mathcal{A}_s} \pi(a|s)\,P(s'|s,a), \qquad s,s'\in\mathcal{S}.
\end{align}

% the policy induced transition matrix on $\mathcal{S}$ is the substochastic matrix $P_\pi \in \R^{|\mathcal{S}| \times |\mathcal{S}|}$ with entries
% \[
% P_\pi(s,s') \coloneqq \sum_{a\in\mathcal{A}_s} \pi(a|s)\,P(s'|s,a), \qquad s,s'\in\mathcal{S}.
% \]
% For a value function $V \in \R^{|\mathcal{S}|}$, the state--action value vector (Q-value) at state $s$ is
% \begin{equation}\label{eq:Q_def}
%     Q_s(V) \coloneqq u(s,\cdot) + P(\cdot|s,\cdot)\,V \;\in\; \R^{|\mathcal{A}_s|}, \qquad s \in \mathcal{S},
% \end{equation}
% whose $a$-th component $u(s,a) + \sum_{s'} P(s'|s,a)\,V(s')$ equals the immediate reward plus the expected continuation value of choosing action $a$.

\subsection{Surplus functions and induced choice map}\label{sec:surplus}

% \kz{(MOVED FROM REVIEW) A distinct strand of recent work in machine learning replaces the softmax (logit) mapping with sparse alternatives that map some probability mass exactly to zero.
% The $\alpha$-entmax family~\citep{peters2019sparse, correia2019adaptively} generalizes softmax through a Tsallis entropy perturbation parameterized by $\alpha\ge 1$: softmax (logit) corresponds to $\alpha=1$, while $\alpha>1$ produces sparse choice probabilities.
% The limiting case $\alpha=2$ is sparsemax (quadratic perturbation)~\citep{martins2016softmax}, whose choice map is the Euclidean projection onto the simplex.
% These mappings are well-defined PUM, but to our knowledge they have not been analyzed in network equilibrium settings.
% Their boundary behavior (corner solutions) falls outside the standard Legendre-type regularity~\citep{Rockafellar+1970} assumed in the existing Markovian equilibrium literature, motivating the general conditions for surplus functions developed in Section~\ref{sec:PUMCM}.}

For each state $s\in\mathcal{S}$, we introduce a \emph{surplus} function $H_s: \R^{|\mathcal{A}_s|}\to \R$ that satisfies the following conditions throughout the paper. 
\begin{standingass}[Base conditions of surplus]\label{asm:surplus}
    For every $s \in \mathcal{S}$,
    \begin{enumerate}[nosep, label=\textup{(A\arabic*)}]
        \item \label{asm:surplus_convex} \textbf{Convexity and smoothness:} $H_s$ is convex and continuously differentiable ($C^1$) on all of $\R^{|\mathcal{A}_s|}$.
        \item \label{asm:surplus_grad} \textbf{Simplex gradient:} $\nabla H_s(Q) \in \Delta_s$ for every $Q \in \R^{|\mathcal{A}_s|}$.
        \item \label{asm:surplus_equiv} \textbf{Translation equivalence:} $H_s(Q + \alpha\mathbf{1}) = H_s(Q) + \alpha$ for every $Q \in \R^{|\mathcal{A}_s|}$, $\alpha \in \R$.
    \end{enumerate}
\end{standingass}
Standing assumption~\ref{asm:surplus} extends the conditions of the choice probability generating function for ARUM proposed in \citet{FOSGERAU20131} to perturbed utility models.
Specifically, Condition~\ref{asm:surplus_convex} relaxes the \textit{alternating sign condition} in Williams-Daly-Zachary Theorem~\citep{mcfadden1981econometric} to accommodate richer behavioral patterns such as complementarity~\citep{hofbauer2002global}. 
Condition~\ref{asm:surplus_grad} is similar to the modeling assumption in ARUM-based Markovian models that ensures the induced choice map is well defined (see Lemma~\ref{lem:choice_map}). 
Differently, Condition~\ref{asm:surplus_grad}  does not require $\nabla H_s$ to lie in the interior $\Delta_s$, which would otherwise rule out corner solutions with zero probability. It thus gives a strict generalization of PUMCM over ARUM-based models. 
Condition~\ref{asm:surplus_grad} also implies that $H_s$ is non-decreasing as $\nabla H_s \geq 0$, and that $H_s$ is Lipschitz continuous with constant $L_\infty=1$ as $||\nabla H_s||_1 = 1$. 
Condition~\ref{asm:surplus_equiv} is commonly assumed for expected maximum operators (e.g., log-sum-exp) and enables the normalization needed for the well-posedness of the Bellman operator (see Proposition~\ref{prop:well-posedness}). This assumption also underpins the single-level regression-based estimation framework proposed in~\citet{yaoperturbed}. 
% Finally, Assumption~\ref{asm:surplus_extra} is optional and only invoked when the smoothness of induced demand flows is desired in applications~\citep{yao2026maas}.

Below is an additional condition that strengthens the analytical properties of PUME, such as the smoothness of the induced optimal link demand (see Proposition~\ref{prop:demand}).
\begin{assumption}[Second-order smoothness]\label{asm:surplus_extra}
    $H_s$ is twice continuously differentiable ($C^2$) on all of $\R^{|\mathcal{A}_s|}$.
\end{assumption}

% \begin{assumption}[Surplus function]\label{ass:surplus}
% For every $s \in \mathcal{S}$:
% \begin{enumerate}[label=\textup{(A\arabic*)}]
%     \item\label{A1}\textbf{Convexity and smoothness:}\;
%       $H_s$ is convex and continuously differentiable ($C^1$) on all of $\R^{|\mathcal{A}_s|}$.
%     \item\label{A2}\textbf{Simplex gradient:}\;
%       $\nabla H_s(Q) \in \Delta_s$ for every $Q \in \R^{|\mathcal{A}_s|}$.
%     \item\label{A3} \textbf{Translation equivalance:}\;
%       $H_s(Q_s + \alpha\mathbf{1}) = H_s(Q) + \alpha$ for every $Q \in \R^{|\mathcal{A}_s|}$, $\alpha \in \R$.
%     \item \label{A4}\textbf{Second-order smoothness}\;(optional):\;
%       $H_s$ is twice continuously differentiable ($C^2$) on all of $\R^{|\mathcal{A}_s|}$.
% \end{enumerate}
% \end{assumption}

We now connect the surplus function to the perturbed utility maximization problem and derive the choice map. 
Define the perturbation function $H_s^*:\Delta_s\to\R\cup\{+\infty\}$ as the convex conjugate of $H_s$ restricted to the simplex $\Delta_s$:
\begin{equation}\label{eq:conjugate}
  H_s^*(\pi) \coloneqq \sup_{Q \in \R^{|\mathcal{A}_s|}} \bigl\{\pi\tran Q - H_s(Q)\bigr\}, \qquad \pi \in \Delta_s.
\end{equation}

Adopting the Fenchel-Young (FY) duality~\citep{Rockafellar+1970}, we prove that the surplus function $H_s$ corresponds to the maximum expected perturbed utility defined by the perturbation function $H_s^*$, and the gradient $\nabla H_s$ corresponds to the expected perturbed utility maximizer. These results are formally presented in the following lemma and proved in Appendix~\ref{app:proof_lem_choice_map}.

\begin{lemma}[FY duality and induced choice map]\label{lem:choice_map}
% Suppose $H_s$ satisfies Conditions~\ref{A1}--\ref{A3}, then 
For any $Q \in \R^{|\mathcal{A}_s|}$,
\begin{enumerate}[nosep, label=\textup{(\roman*)}]
    \item\label{FY:ineq} \textbf{Weak duality:} $\pi\tran Q - H_s^*(\pi) \le H_s(Q), \quad \forall \pi \in \Delta_s$.
    
    \item\label{FY:max} \textbf{Strong duality:} $H_s(Q) = \max_{\pi \in \Delta_s} \bigl\{\pi\tran Q - H_s^*(\pi)\bigr\} = \nabla H_s(Q)\tran Q - H_s^*(\nabla H_s(Q))$ with optimizer $\pi_* = \nabla H_s(Q)$.
\end{enumerate}
\end{lemma}

The surplus function and its induced choice map are closely related to the ARUM-based (Markovian) choice models~\citep{mcfadden1981econometric, chiong2016duality} and recent work in machine learning regarding the sparse alternative of softmax~\citep[e.g.,][]{martins2016softmax,peters2019sparse,correia2019adaptively}. 
In Section~\ref{sec:spec_surplus}, we discuss available surplus functions and their properties in more detail.

\subsection{Bellman operators and well-posedness}\label{sec:operators}
In this section, we define the Bellman operator associated with the undiscounted MDP defined in Section~\ref{sec:PUMCM:setting} regularized by the perturbation $H_s^*$ introduced in Section~\ref{sec:surplus}. 
% and then introduce mild conditions that guarantee its well-posedness. We further demonstrate the automatic properness of the policy \kz{induced by} the Bellman operator. 
Since the model is undiscounted, the policy Bellman operator is not contractive in general~\citep{mai2022undiscounted}, and a Bellman equation may fail to have a finite solution.
However, we show that, under a strict negativity condition on the surplus, the conjugate structure of PUMCM rules out this problematic case for the whole model family.

% undiscounted Bellman operators induced by a perturbation $H_s^*$ and introduces mild conditions that ensure well-posedness.
% It then states the automatic properness property that underpins our Bellman theory.

We first define the expected perturbed stage reward for a given policy $\pi$ as the expected stage reward minus the perturbation, i.e., 
\begin{align}\label{eq:exp_reward}
    U_\pi(s) \coloneqq \pi(\cdot|s)\tran u(s,\cdot) - H_s^*(\pi(\cdot|s)).
\end{align}
% Accordingly, the Bellman operator is given by 
Because $H_s^*$ is not bounded by definition, we call a policy $\pi$ \emph{admissible} if $H_s^*(\pi(\cdot|s))<\infty$ for all $s\in\mathcal{S}$.
The non-admissible policies thus have $U_\pi(s)=-\infty$ for at least one state and cannot induce a finite policy value.
Accordingly, for an admissible policy, we define the Bellman operator on each state $s$ as
% For a given policy $\pi$, the policy stage reward is the expected one-step utility minus the perturbation,
% $U_\pi(s) \coloneqq \pi(\cdot|s)\tran u(s,\cdot) - H_s^*(\pi(\cdot|s))$.
% The policy Bellman operator is
\begin{equation}\label{eq:T_pi}
    T_\pi V(s) \coloneqq U_\pi(s) + P_\pi(s,\cdot) V,
\end{equation}
where $P_\pi$ is the policy-induced transition defined in Eq.~\eqref{eq:policy_transition}, and similarly the Bellman optimality operator as
% The Bellman optimality operator defined on each state $s$ is 
\begin{equation}\label{eq:T_star}
    T_* V(s) \coloneqq  \max_{\pi \in \Delta_s} \Bigl\{U_\pi(s) + P_\pi(s,\cdot) V\Bigr\} = \max_{\pi \in \Delta_s} \Bigl\{\pi\tran Q_s(V) - H_s^*(\pi)\Bigr\}
    \;=\; H_s(Q_s(V)).
\end{equation}
Note the first equality in Eq.~\eqref{eq:T_star} is derived from Eqs.~\eqref{eq:Q_def} and \eqref{eq:policy_transition}, and the second equality follows from Lemma~\ref{lem:choice_map}.
% the second equality follows from Lemma~\ref{lem:fenchel-young}\ref{FY:max}.
% The optimal policy at $V$ is $\pi_V(\cdot|s) \coloneqq \nabla H_s(Q_s(V))$, and satisfies $T_* V = T_{\pi_V} V$ by Lemma~\ref{lem:fenchel-young}\ref{FY:eq}.
% From Eq.~\eqref{eq:T_pi}, we can write the value function induced by policy $\pi$ as a fixed point

Accordingly, evaluating a fixed policy $\pi$ reduces to solving the linear system of its induced value function $V_\pi$:
\begin{align}
    V_\pi = T_\pi V_\pi = U_\pi + P_\pi V_\pi,
\end{align}
whose solution existence depends on whether the destination is reached almost surely under $\pi$. 
% It is worth noting that, as the MDP is undiscounted, $V_\pi$ is well defined only when the destination is reached almost surely. 
Policies that satisfy this condition are considered proper, which is formally defined below. 

% We first recall the standard notion of properness, which ensures this condition.
\begin{definition}[Proper policy,~\citet{bertsekas2012dynamic}]\label{def:proper}
A policy $\pi$ is proper if the termination state $d$ is reached with probability one from every non-terminal state.
Equivalently, $\rho(P_\pi) < 1$, where $\rho(\cdot)$ denotes the spectral radius.
\end{definition}

% A direct consequence of properness is that the matrix inverse $(\mathbb{I} - P_\pi)^{-1}$, also known as the fundamental matrix~\kz{(ref)}, exists and is nonnegative element-wise. Accordingly,  the value function $V_\pi = (\mathbb{I} - P_\pi)^{-1} U_\pi$ is well defined for every policy stage reward $U_\pi$. We state this formally in the next lemma.

If $\pi$ is proper and admissible, the Neumann series $\sum_{k=0}^\infty P_\pi^k$ converges to the nonnegative fundamental matrix $(\mathbb{I} - P_\pi)^{-1} \geq 0$ (i.e., expected state occupancy under $\pi$), and the policy evaluation equation has a unique finite solution $V_\pi = (\mathbb{I} - P_\pi)^{-1} U_\pi$~(see Lemma~\ref{lem:neumann} in Appendix~\ref{app:proof_well-posedness}).
The converse result, however, does not directly apply to a general undiscounted MDP. 
For instance, an infinite loop may result from a non-terminal state with zero (finite) expected stage reward,  
% For instance, a non-terminal state with zero expected stage reward, although finite, can lead to infinite loops, 
and thus the corresponding policy is not proper. The following assumption rules out this possibility and serves as the second standing assumption of this paper.

% The converse direction is less immediate: can an improper policy nevertheless satisfy a finite Bellman equation?
% In a general undiscounted MDP this can happen when a recurrent non-terminal class has zero average reward, and leads to infinite loops.
% The following two assumptions rule out this possibility in PUMCM.

% The following lemma states the direct consequence of policy properness and its proof is provided in Appendix~\ref{app:proof_lem_neumann}. 
% \begin{lemma}[Existence and uniqueness of policy-induced value]\label{lem:neumann}
% If policy $\pi$ is proper such that $\rho(P_\pi) < 1$, then the fundamental matrix $(\mathbb{I} - P_\pi)^{-1}$ \kz{(i.e., expected state occupancy under policy $\pi$)} exists and is nonnegative element-wise.
% Consequently, for every $U \in \R^{|\mathcal{S}|}$, the linear system $V = U + P_\pi V$ has the unique solution $V = (\mathbb{I} - P_\pi)^{-1} U$.
% \end{lemma}

% Extending the result to the Bellman optimality operator is more challenging and requires additional conditions introduced below, which jointly guarantee the automatic properness of the operator stated in Theorem~\ref{thm:auto-proper}. 
% Similar challenge arises for the optimal value function: in the undiscounted setting, a policy can in principle cycle indefinitely, which would make the optimal value function ill posed.
% We therefore impose two conditions that jointly ensure the undiscounted Bellman operators remain well posed and enable the automatic properness results below.

\begin{standingass}[Base conditions of PUMCM]
For any stage reward $u$, 
    \begin{enumerate}[nosep, label=\textup{(B\arabic*)}]\label{asm:proper}
        \item \label{asm:proper_reach} \textbf{Reachability:} There exists at least one proper and admissible policy. 
        \item \label{asm:proper_nonpos} \textbf{Strictly negative stage surplus:} The stage surplus $H_s(u(s,\cdot))<0,\;\forall s\in S$. 
    \end{enumerate}
\end{standingass}
Condition~\ref{asm:proper_reach} imposes the minimal condition that the destination is reachable with at least one policy and their respective policy-induced values are bounded, which ensures the expected perturbed stage reward is well-defined. 
Condition~\ref{asm:proper_nonpos}, on the other hand, states that, under any policy and at any state, the maximal expected perturbed utility obtained at each stage (excluding future value) is strictly negative.
Together, Standing Assumption~\ref{asm:proper} ensures well-posedness of the underlying MDP. Specifically, the former ensures an exit of the network exists, while the latter penalizes time spent in the network and thus avoids infinite looping. 

\begin{remark}[Strict negative stage surplus via translation equivalence]\label{rem:stage-negativity}
In many traffic assignment models, $u(s,a)$ is defined as the negative link costs and thus always strictly negative. Under the standard surplus function design, Condition~\ref{asm:proper_nonpos} naturally holds. 
When the condition is not satisfied, the translation equivalence introduced by Standing Assumption~\ref{asm:surplus}~\ref{asm:surplus_equiv}  provides a simple solution to enforce a strictly negative stage surplus while leaving the stage choice map unchanged. Specifically, it states that $H_s( u(s,\cdot) - \alpha_s\mathbf{1})=H_s(u(s,\cdot))-\alpha_s$ for any constant $\alpha_s\in\R$ at any state $s\in S$. Accordingly, we have $\nabla H_s(u(s,\cdot)-\alpha_s\mathbf{1}) = \nabla H_s(u(s,\cdot))$, i.e., the induced stage choice map remains unchanged under a constant shift in stage reward. Therefore, we can select $\alpha_s\ge H_s(u(s,\cdot))+\epsilon$ for some $\epsilon>0$ to guarantee the condition stated in Standing Assumption~\ref{asm:proper}\ref{asm:proper_nonpos}.
\end{remark}

With both standing assumptions, we establish the well-posedness of the Bellman operators associated with PUMCM.

\begin{proposition}[PUMCM well-posedness]\label{prop:well-posedness}
% Under Conditions~\ref{A1}--\ref{A3} and Assumptions~\ref{asm:proper}--\ref{asm:proper}:
Any PUMCM satisfying the standing assumptions has the following properties:
\begin{enumerate}[nosep, label=\textup{(\roman*)}]
    \item\label{prop:universal-strict} Every policy $\pi$ satisfies $U_\pi(s)<0$ at all non-terminal states $s \in \mathcal{S}$.
    \item\label{prop:auto-proper} An admissible policy $\pi$ is proper if and only if $V = T_\pi V$ has a finite and unique solution. 
    % When proper, the solution $V_\pi$ is unique.
    \item\label{prop:optimal-proper} If $V = T_* V$ has a finite 
    % \kz{(and unique?)} 
    solution $V_*$, then the optimal policy $\pi_* = \nabla H(Q(V_*))$ is proper.
\end{enumerate}
\end{proposition}
\begin{proof}[Proof sketch]
We first show that a direct application of Lemma~\ref{lem:choice_map}\ref{FY:max} yields Property~\ref{prop:universal-strict}.
The sufficiency of Property~\ref{prop:auto-proper} is proved by invoking Lemma~\ref{lem:neumann}, while the necessity is proved by contradiction. 
% For the sufficiency direction of Item~\ref{prop:auto-proper}, properness gives the Neumann series $V_\pi = (\mathbb{I}-P_\pi)^{-1}U_\pi$, which is unique (Lemma~\ref{lem:neumann}).
% For necessity, we show by contradiction that if $\pi$ were improper, there exists a recurrent class such that the Bellman equation contradicts Item~\ref{prop:universal-strict}.
Property~\ref{prop:optimal-proper} naturally follows \ref{prop:auto-proper} by the equivalence between $T_*$ and $T_{\pi_*}$, where $\pi_*$ denotes the policy associated with $V_*$.
% Item~\ref{prop:optimal-proper} follows because $V_* = T_{\pi_*} V_*$ reduces to~\ref{prop:auto-proper}.
The complete proof is provided in Appendix~\ref{app:proof_well-posedness}.
\end{proof}

The convex conjugate embedded in PUMCM enables the well-posedness properties in Proposition~\ref{prop:well-posedness}: the Fenchel-Young inequality forces $U_\pi(s) < 0$ at every non-terminal state under the strict negative surplus condition (Standing Assumption~\ref{asm:proper}\ref{asm:proper_nonpos}), 
preventing an improper policy that leads to recurrent cycles with a finite Bellman equation. 
% preventing the reward cancellation along recurrent cycles that could otherwise sustain an improper policy with a finite Bellman equation.
The well-posedness of Bellman operators then ensures that all policies yield absorbing Markov chains with a well-defined fundamental matrix $(\mathbb{I}-P_\pi)^{-1}$ (Lemma~\ref{lem:neumann}), which in turn underpins the existence and uniqueness of a finite (optimal) value $V$ ($V_*$). It further induces the differentiability of the optimal value $V_*(u)$ with respect to the stage reward (link utility) $u$ (see Proposition~\ref{prop:demand}) and the global convergence of network loading (see Theorem~\ref{thm:mpi_global}).
% Well-posedness then guarantees that all policies generated by the Bellman operators induce absorbing Markov chains with a well-defined fundamental matrix $(\mathbb{I}-P_\pi)^{-1}$, which in turn underpins the existence and uniqueness of the optimal value $V_*$, the differentiability of the induced demand map, and the global convergence of network loading (Sections~\ref{sec:demand} and~\ref{sec:mpi}).
It is also worth noting that the policy properness derived in Proposition~\ref{prop:well-posedness} holds for any policies, including the boundary ones producing corner solutions with zero choice probabilities, thus bypassing the interior-solution requirement and the network topology restrictions imposed in previous studies~\citep[e.g.,][]{mai2022undiscounted,oyama2019prism}.
\subsection{Optimal values and link demand}\label{sec:demand}
With the well-posedness of PUMCM established in Proposition~\ref{prop:well-posedness}, we now characterize the optimal value $V_*(u)$ and link demand $x_*(u)$, i.e., link flow generated by the optimal policy $\pi_*(u)$, given the stage reward (link utility) $u$. 
For the simplicity of notation, we suppress the dependence on $u$ when clear from context.
The existence and uniqueness of $V_*$ rely on the following monotonicity properties of the Bellman operators, proved in Appendix~\ref{app:proof_Tstar_properties}.
% To this end, we first prove the existence and uniqueness of $V_*(u)$, which relies on the monotonicity of the Bellman operators $T_\pi$ and $T_*$ stated in the following lemma. The proof is provided in Appendix~\ref{app:proof_Tstar_properties}. Here, the condition on the reward $u$ is first dropped for the simplicity of notation. 
% action demand induced by PUMCM for a given reward vector.
% We first establish existence and uniqueness of the undiscounted optimal value $V_*$, which in turn pins down the optimal policy $\pi_*=\nabla H(Q(V_*))$.
% We then derive the induced action demand (link flow) $x^*(u)$ and its basic regularity properties.
% Properness is the key structural ingredient as it ensures that policy evaluation and the derivatives used below remain well defined even under corner solutions.

\begin{lemma}[Monotonicity and greedy improvement]\label{lem:Tstar_properties}
% Under Conditions~\ref{A1}-\ref{A3},
The Bellman operators of PUMCM satisfy the following properties:
\begin{enumerate}[nosep, label=\textup{(\roman*)}]
    \item If $V_1 \ge V_2$ element-wise, then $T_\pi V_1 \ge T_\pi V_2$ for any policy $\pi$.
    \item If $V_1 \ge V_2$ element-wise, then $T_* V_1 \ge T_* V_2$.
    \item $T_* V_\pi \ge V_\pi$ for any proper policy $\pi$ such that $V_\pi = T_\pi V_\pi$.
\end{enumerate}
\end{lemma}

These properties, together with well-posedness derived in Proposition~\ref{prop:well-posedness}, yield the existence and uniqueness of the optimal value stated in the following proposition with a detailed proof in Appendix~\ref{app:proof_Vstar}.

\begin{proposition}[Optimal value existence and uniqueness]\label{prop:Vstar}
The Bellman optimality operator of PUMCM $T_*$ admits a unique fixed point $V_* \in \R^{|\mathcal{S}|}$ satisfying $T_* V_* = V_*$.
Moreover, $V_*$ dominates every admissible proper policy value, i.e., $V_* \ge V_\pi$ for every admissible proper policy $\pi$, and the optimal policy $\pi_* = \nabla H(Q(V_*)):=(\nabla H_s(Q_s(V_*)))_{\forall s\in \mathcal{S}}$ is proper and unique.
\end{proposition}

\begin{proof}[Proof sketch]
% \kz{
We prove the existence of $V_*$ using Brouwer's fixed point theorem applied to a bounded invariant set constructed via Lemma~\ref{lem:Tstar_properties}, and utilize the policy properness to prove its uniqueness.
Dominance of $V_*$ follows from the greedy improvement of $T_*$ (Lemma~\ref{lem:Tstar_properties}). The properness of $\pi_*$ is due to Proposition~\ref{prop:well-posedness}\ref{prop:optimal-proper} and its uniqueness is naturally implied from the uniqueness of $V_*$ and the property of the surplus function $H$. 
% The dominance of $V_*$ is directly induced from the proof, and finally, the properness of $\pi_*$ is directly implied from Corollary~\ref{cor:optimal-proper}. The detailed proof is provided in Appendix~\ref{app:proof_Vstar}. }
\end{proof}

With $V_*$ established, we derive the optimal link demand $x_*(u)$ induced by the optimal policy $\pi_*(u)$ and show that it can be expressed as the gradient of a potential $\phi(u)$. Specifically, $\phi(u)$ is constructed using the load $q \in \R_+^{|\mathcal{S}|}$, i.e., demand originated from each non-terminal state to the terminal state, and the optimal value $V_*(u)$. This result is formally stated in the following proposition. For notation simplicity, we use $C^1$ to denote continuous differentiability and $C^2$ for the second-order hereafter. 

\begin{proposition}[Optimal link demand]\label{prop:demand}
    The optimal value function $V_*(u)$ is component-wise convex and $C^1$ in stage reward $u$. 
    % Let $V_*(u)$ denote the optimal value obtained at reward $u$ under the same conditions of Proposition~\ref{prop:Vstar}. Then, the mapping $u \mapsto V_*(u)$ is convex and $C^1$. 
    For any non-negative load $q \in \R_+^{|\mathcal{S}|}$, the scalar potential $\phi(u)\coloneqq q\tran V_*(u)$ is convex and $C^1$. Specifically, its gradient yields the closed-form expression of optimal link demand 
    \begin{align}\label{eq:link_demand}
        x_*(u) = \nabla \phi(u) = [\nabla V_*(u)]\tran q = [(\mathbb{I} - P_{\pi_*})^{-1} \pi_*(u)]\tran q,
    \end{align}
    which is both continuous and monotone non-decreasing in $u$. 
    Moreover, when Assumption~\ref{asm:surplus_extra} holds, $V_*(u)$ is $C^2$ and $x_*(u)$ is $C^1$ with positive semi-definite Jacobian $\nabla x_*(u)=\nabla^2 \phi(u)$.
\end{proposition}

% With $V_*$ established, the induced action demand can be obtained by differentiating through the optimal value function $V_*$.
% This gives a convenient representation of demand as a smooth map in reward space, which later becomes a demand map in cost space once $u$ depends on $c$.

% \begin{corollary}[Monotone and continuous action demand]\label{cor:demand}
% Suppose the conditions of Proposition~\ref{prop:Vstar} hold.
% Then the mapping $u \mapsto V_*(u)$ is convex and continuously differentiable ($C^1$), and for any nonnegative load vector $q \in \R_+^{|\mathcal{S}|}$, the scalar potential $\phi(u)\coloneqq q\tran V_*(u)$ is convex and $C^1$ with gradient being the action demand:
% \begin{equation}\label{eq:link_demand}
%     x^*(u) \coloneqq \nabla \phi(u) = [\nabla V_*(u)]\tran q = [(\mathbb{I} - P_{\pi_*})^{-1} \pi_*]\tran q.
% \end{equation}
% Since $\phi$ is convex and $C^1$, its gradient $x^*$ is continuous and monotone:
% \begin{equation}\label{eq:reward_monotone}
%     (x^*(u) - x^*(u'))\tran (u - u') \ge 0, \qquad \forall\, u, u'.
% \end{equation}
% If additionally Assumption~\ref{ass:smooth} holds, then $V_*$ is twice continuously differentiable ($C^2$), $x^*$ is $C^1$, and the Jacobian $\nabla x^*(u)=\nabla^2 \phi(u)$ is positive semidefinite.
% \end{corollary}

\begin{proof}[Proof sketch]
Given the well-defined fundamental matrix $(\mathbb{I} - P_{\pi_*})^{-1}$ (Lemma~\ref{lem:neumann}), the optimal link flow can be computed using the state-action occupancy and load, i.e., $x_*(u)=[(\mathbb{I} - P_{\pi_*})^{-1}\pi_*(u)]\tran q$.
% Recall that the fundamental matrix $(\mathbb{I} - P_{\pi_*})^{-1}$ gives state occupancy at the optimal policy $\pi_*$.
% Accordingly, $(\mathbb{I} - P_{\pi_*})^{-1}\pi_*(u)$ yields the state-action occupancy and thus the link flow is given by $x_*(u)=[(\mathbb{I} - P_{\pi_*})^{-1}\pi_*(u)]\tran q$.
The key step of proof is thus to show that the state-action occupancy corresponds to the gradient of the optimal value function, i.e., $\nabla V_*(u) = (\mathbb{I} - P_{\pi_*})^{-1}\pi_*(u)$, along with its convex and $C^1$ properties. 
% , and accordingly, the link flow corresponds to the gradient of the proposed scalar potential $\phi(u)$. 
See Appendix~\ref{app:proof_demand} for the complete proof.
\end{proof}
% Differentiate the fixed-point equation $V_*(u)=T_* V_*(u)=H(Q(V_*(u)))$ via the implicit function theorem (IFT).
% The Jacobian with respect to $V$ is $\mathbb{I}-P_{\pi_*}$, which is invertible because $\pi_*$ is proper (Corollary~\ref{cor:optimal-proper}).
% IFT then gives $\nabla V_*(u)=(\mathbb{I}-P_{\pi_*})^{-1}\pi_*$, and hence~\eqref{eq:link_demand}.
% Convexity of $V_*$ (componentwise) follows from $V_*(u)=\sup_{\text{proper }\pi}V_\pi(u)$, a supremum of affine functions; this yields convexity of $\phi=q\tran V_*$ for $q\ge 0$.
% Continuity and monotonicity~\eqref{eq:reward_monotone} are standard properties of the gradient of a convex $C^1$ function.
% Under additional Assumption~\ref{ass:smooth}, the map $H$ is $C^2$, so the IFT gives $V_*$ is $C^2$ with $C^1$ gradient $x^*$ and the Hessian $\nabla^2 \phi\succeq 0$ exists.
% See Appendix~\ref{app:deferred-sec3} for details.

Proposition~\ref{prop:demand} establishes the demand side of PUME. Given the stage reward (link utility) $u$, PUMCM produces well-behaved link demand $x_*(u)$. In the next section, we will move to the supply side by defining the supply mapping, and finally match demand and supply via a variational inequality.

% Corollary~\ref{cor:demand} provides a simple interface for equilibrium analysis: given utilities $u$ satisfying conditions in Proposition~\ref{prop:Vstar}, PUMCM produces an action demand vector $x^*(u)$ that is well-behaved.
% In the next section we move to cost space, compose this demand map with a utility--cost relation $u(c)$, and couple the resulting demand with a (possibly asymmetric) supply mapping through a variational inequality. This is our equilibrium framework.

%%%%%%%%%%%%%%%%%%%%%%%%%%%%%%%%%%%%%%%%%%%%%%%%%%%%%%%%%%%%%%%%%%%%%%
%%           SECTION 3: PUME EQUILIBRIUM
%%%%%%%%%%%%%%%%%%%%%%%%%%%%%%%%%%%%%%%%%%%%%%%%%%%%%%%%%%%%%%%%%%%%%%
\section{Perturbed Utility Markovian Equilibrium}\label{sec:equilibrium}

This section couples the demand-side mapping produced by PUMCM  with a supply-side mapping and establishes the equilibrium condition defined in a more general cost space. We start by transforming the link demand derived in Section~\ref{sec:demand} from the link utility (stage reward) space to the cost space (Section~\ref{sec:cost_demand}), then define the supply mapping and variational inequality (VI) condition of PUME, followed by its existence and uniqueness (Section~\ref{sec:vi_formulation}), and finally construct the corresponding variational Nash equilibrium (VNE) condition in the link-flow space (Section~\ref{sec:vne}), also known as the primal formulation~\citep[e.g.,][]{baillon2008markovian}. 

% We now couple the PUMCM demand map from Section~\ref{sec:demand} with a supply-side mapping through a cost-space equilibrium condition.
% Throughout, the cost vector $c$ is the endogenous equilibrium variable, and PUMCM demand is evaluated at utilities induced by $c$.

\subsection{Aggregate link demand in cost space}\label{sec:cost_demand}

Let $\mathcal{L}$ be a finite set of cost components with $L\coloneqq |\mathcal{L}|$, and $\Omega \subseteq \R_+^L$ be a closed convex cost domain. Hence, a cost vector is expressed as $c=(c_\ell)_{\ell\in\mathcal{L}}$. 
In transportation applications, $\mathcal{L}$ can represent the set of physical network links, and $\Omega$ is thus the set of admissible link-cost vectors. More generally, $\mathcal{L}$ represents any shared resources whose costs are endogenously determined. In what follows, we consider $\mathcal{L}$ corresponds to the link set and thus $c$ shares the same dimension as aggregate link demand.

The PUMCM described in Section~\ref{sec:PUMCM} corresponds to a single user class sharing the destination (i.e., terminal state). In standard traffic assignment models, multiple destinations exist. 
More generally, different user classes may vary in state-action structure, transition kernel, and perturbation functions, while they may share the same cost when traveling on the same link. 
To capture the full heterogeneity, we introduce $\mathcal{K}$ as a finite set of user classes and let $\mathcal{S}_k$ and $\mathcal{A}_k = \cup_{s\in\mathcal{S}_k} \mathcal{A}_{k,s}$ denote the corresponding state and action space, respectively. 
The following assumption defines the cost-utility mapping based on the standard sign convention.
\begin{assumption}[Linear cost-utility mapping]\label{asm:utility}
    For each class $k$, 
    there exists a class-specific incidence matrix $B_k\in \{0,1\}^{|\mathcal{A}_k|\times L}$ such that $u_k(c) = -B_k c$, 
    where $B_{k,(s,a),\ell} = 1$ if utility $u_k(s,a)$ corresponds to the cost $c_\ell$ and 0 otherwise. Besides, the matrix $B_k$ ensures that $u_k(c)$ satisfies Standing Assumption~\ref{asm:proper}\ref{asm:proper_nonpos} for all $c\in \Omega$. 
\end{assumption}

Accordingly, the aggregate link demand $x\in \R_+^L$ at cost $c$ is:
\begin{align}
    x(c) := \sum_{k\in \mathcal{K}} B_k\tran x_{k*}(u_k(c)), 
\end{align}
where $x_{k*}$ is the class-$k$ optimal demand from Proposition~\ref{prop:demand}, evaluated at the cost-induced utility $u_k(c)$. The transpose $B_k^\top$ maps each class-specific demand to the shared link space, so the sum aggregates contributions from all user classes onto the common cost components.

The following proposition establishes the monotone property of the aggregate link flow with respect to link cost. 
\begin{proposition}[Monotone link demand in cost space]\label{prop:monotone_demand}
Under Assumption~\ref{asm:utility}, there exists a convex $C^1$ function $\Phi:\Omega\to\R$ such that
\begin{equation}\label{eq:potential}
    \nabla\Phi(c) = -x(c), \qquad \forall\, c \in \Omega.
\end{equation}
The aggregate link demand map $x:\Omega\to\R_+^L$ is continuous and monotone non-increasing: 
\begin{equation}\label{eq:cost_monotone}
    \langle x(c) - x(c'),\, c - c' \rangle \le 0, \qquad \forall\, c,c'\in\Omega.
\end{equation}
When Assumption~\ref{asm:surplus_extra} holds, $\Phi$ is $C^2$ and $x$ is $C^1$ with $\nabla x(c) = -\nabla^2\Phi(c) \preceq 0$ for all $c\in\Omega$.
% If additionally Assumption~\ref{ass:smooth} holds, then $\Phi$ is $C^2$, $x$ is $C^1$, and $\nabla_c x(c) = -\nabla^2\Phi(c) \preceq 0$ is symmetric negative semidefinite for all $c\in\Omega$.
\end{proposition}
\begin{proof}[Proof sketch]
    The potential function is constructed as $\Phi(c)=\sum_k \phi_k(u_k(c))$ based on the class-specific potential derived in Proposition~\ref{prop:demand}. Its properties and other results are derived accordingly. 
    % We prove the result by construct the potential function $\Phi$ using the class-specific potentials as per Proposition~\ref{prop:demand}, $\Phi(c)=\sum_k \phi_k(u_k(c))$, and derive the monotonicity accordingly. 
    See detailed proof in Appendix~\ref{app:proof_monotone_demand}.
\end{proof}

\begin{remark}[General cost-utility mapping]\label{rem:general_utility}
Proposition~\ref{prop:monotone_demand} is established on the linear cost-utility mapping stated in Assumption~\ref{asm:utility}, which covers standard routing problems. It is worth noting that the result holds as long as the class-specific link demand can be expressed as the (negative) gradient of a convex potential function in the cost space. 
The following are two common cases that satisfy this condition:
\begin{enumerate}[nosep,label=\textup{(\roman*)}]
    \item \textbf{General affine mapping:} $u_k(c) = M_k c + b_k$ for some class-specific matrix $M_k$ and offset $b_k$. Then the aggregate link flow is given by
    \begin{align}
        x(c) = -\sum_{k\in\mathcal{K}} (M_k)\tran x_{k,*}(u_k(c)),
    \end{align}
    and the results in Proposition~\ref{prop:monotone_demand} also hold. Assumption~\ref{asm:utility} is then the special case with $M_k=-B_k$ and $b_k=0$.

    \item \textbf{Component-wise convex mapping:}
    Each component of $u_k(c)$ is convex in $c$. Then the aggregate link flow is given by
    \begin{align}
        x(c) = -\sum_{k\in\mathcal{K}} [\nabla u_k(c)]\tran x_{k,*}(u_k(c)).
    \end{align}
    In this case, Proposition~\ref{prop:monotone_demand} still holds due to the standard composition rule of convex functions~\citep[Sec.~3.2.4]{boyd2004convex}. Particularly, it reduces to case (a) when $u_k(c)$ is affine.
\end{enumerate}
\end{remark}

\subsection{Dual formulation in cost space}\label{sec:vi_formulation}
As the final step to establish PUME, we define a supply-side mapping, also in the cost space, and couple it with the demand-side mapping $x(c)$ obtained in Section~\ref{sec:cost_demand}. 
% To close the model, we couple the aggregate demand $x(c)$ from Section~\ref{sec:cost_demand} with a supply-side mapping. 
Let $z : \Omega \to \R_+^L$ be a continuous supply function, with each element $z_\ell(c)$ denotes the supply of flow on link $\ell$ at cost $c$.
% for each cost vector $c$, $z_\ell(c)$ gives the flow that cost component $\ell$ can sustain.
In standard traffic assignment models, the link performance function $t_\ell$ is often assumed to be a strictly increasing function that maps from link flow $x_\ell$ to link cost $c_\ell$. Hence, the supply function can be defined as its inverse $z_\ell = t_\ell^{-1}$.
% the link performance function $t_\ell(x_\ell)$ maps from link flow to link cost, and the corresponding supply function is its inverse, i.e., $z_\ell = t_\ell^{-1}$.
% In the standard transportation setting, link performance functions $t_\ell(x_\ell)$ map flows to costs; the supply function is their inverse, $z_\ell(c_\ell) = t_\ell^{-1}(c_\ell)$.
% For instance, the classic BPR link performance functions yield separable, strictly increasing supply curves.

When $z$ admits a convex potential, the Markov traffic equilibrium can be cast as the solution to a convex program, also known as the dual formulation~\citep{baillon2008markovian,oyama2022markovian}. 
Differently, PUME relaxes such an assumption and allows $z$ to be non-separable (i.e., each link flow is jointly determined by all link costs). It neither requires the Jacobian $\nabla z(c)$ to be symmetric, thus accommodating asymmetric link interactions. 
% In what follows, we define PUME as a variational inequality (VI) on the excess supply operator $E(c) \coloneqq z(c) - x(c)$, which measures the gap between supply and demand at each cost vector.

% When $z$ happens to be a gradient of some convex potential function $Z$, i.e., $z = \nabla Z$, the equilibrium can be cast as a convex program of a (dual) Beckmann-type formulation \citep[see e.g.,][for the MTE case]{baillon2008markovian,oyama2022markovian}.

% Our formulation does not require such integrability.
% The supply function $z$ may be nonseparable (costs on one component affect flows on others) and its Jacobian need not be symmetric, ruling out a potential-function representation.
% This accommodates asymmetric congestion interactions, e.g., merge/diverge effects across links, that arise in practice but fall outside the optimization framework with potential functions.

% Coupling supply with the demand map of Section~\ref{sec:cost_demand}, we define the excess supply operator
% \begin{equation}\label{eq:excess_supply}
%     E(c) \coloneqq z(c) - x(c), \qquad c \in \Omega.
% \end{equation}

\begin{definition}[PUME in cost space]\label{def:PUME}
A cost vector $c^* \in \Omega$ is a \textit{perturbed utility Markovian equilibrium} (PUME) if it solves the variational inequality (VI)
\begin{equation}\label{eq:VI_PUME}
    \langle E(c^*),\, c - c^* \rangle \ge 0, \qquad \forall\, c \in \Omega,
\end{equation}
where $E(c) \coloneqq z(c) - x(c)$ is the excess supply. 
When $c^* \in \interior(\Omega)$, the equilibrium condition reduces to market clearance $z(c^*) = x(c^*)$.
\end{definition}

% Equivalently,~\eqref{eq:VI_PUME} reads $0 \in E(c^*) + N_\Omega(c^*)$, where $N_\Omega(c^*)$ is the normal cone of $\Omega$ at $c^*$.
% The VI formulation handles boundary constraints on costs through $\Omega$ and, crucially, requires only monotonicity of $E$, not integrability (which requires symmetric Jacobian of $E$), for well-posedness.
% Since $x$ is monotone (Theorem~\ref{prop:monotone_demand}), $E = z - x$ is monotone whenever $z$ is; the demand-side regularity from the PUMCM thus complements any monotone supply for regularity of the excess supply function. We discuss conditions on the supply function to allow existence and uniqueness of a PUME.

% Below we establish the existence and uniqueness of PUME. 
Since $\Omega$ is unbounded, the well-posedness of the VI~\eqref{eq:VI_PUME} relies on the monotone increasing property of the excess supply $E$. 
Proposition~\ref{prop:monotone_demand} already guarantees that $x$ is continuous and monotone non-increasing, while the condition on the supply side is outlined in the following assumption. 
% It remains to impose conditions on the supply $z$. 

% \subsection{Existence and uniqueness}\label{sec:existence}

% This subsection provides sufficient conditions for existence and uniqueness of a PUME.
% Existence relies on a growth condition on the supply map and the monotonicity of demand established in Theorem~\ref{prop:monotone_demand}.
% Uniqueness follows from a standard strict-monotonicity strengthening on the supply side.

\begin{assumption}[Coercive and monotone supply]\label{asm:supply}
The continuous supply function $z: \Omega \to \R_+^L$ satisfies the following properties:
\begin{enumerate}[nosep, label=\textup{(\roman*)}]
    \item \label{asm:supply_coercive}\textbf{Coercivity:} For some $\hat{c}\in\Omega$, 
    \begin{equation}\label{eq:coercive}
        \frac{\langle z(c),\, c - \hat{c} \rangle}{\|c - \hat{c}\|} \to +\infty
        \; \text{as } \|c\| \to \infty, c\in\Omega.
    \end{equation}

    \item \label{asm:supply_monotone}\textbf{Monotonicity:} $z$ is monotone non-decreasing in the sense that:
    \begin{equation}\label{eq:monotone_supply}
        \langle z(c) - z(c'),\; c-c' \rangle \geq 0, \; \forall c, c' \in \Omega.
    \end{equation}
\end{enumerate}
\end{assumption}

Assumption~\ref{asm:supply} holds for standard separable supply functions. 
% Standard separable supply functions satisfy both condtions~\eqref{eq:coercive} and~\eqref{eq:monotone_supply}. 
For instance, the inverse of the BPR family $t(x) = \alpha x^\beta + b$ is coercive when $\beta\ge 1$.
The following theorem shows that the coercivity of $z$ alone implies the coercivity of $E$ and thus guarantees existence.
% , even though the demand $x(c)$ could in principle offset the growth in supply.
% $z_\ell(c_\ell)=\kappa_\ell c_\ell^{\beta_\ell}$ is coercive when $\beta_\ell\ge 1$.
% Nevertheless, coercive $z$ stated in Assumption~\ref{asm:coercive} does not directly yields coercive $E$ as the demand $x$ can grow in an opposite direction. The existence of PUME is formally established in the following theorem.

% When $\Omega$ is compact, coercivity is not needed.
% For an unbounded feasible set, existence hinges on coercivity of the \emph{excess supply} $E(c)=z(c)-x(c)$.
% This is not automatic: even if $z$ is coercive, the demand map $x(c)$ is nonlinear and could, a priori, grow in directions that offset the supply term.
% Theorem~\ref{thm:existence} shows that the structural property inherited from PUMCM, monotonicity of $x$ in cost space (Theorem~\ref{prop:monotone_demand}), rules out such cancellation and yields coercivity of $E$ from coercivity of $z$ alone, which then ensures existence of a PUME.

\begin{theorem}[PUME existence]\label{thm:existence}
Under Assumptions~\ref{asm:utility} and~\ref{asm:supply}\ref{asm:supply_coercive}, PUME exists and corresponds to a solution $c^* \in \Omega$ to the VI problem~\eqref{eq:VI_PUME}. 
% \kz{Additionally, if there exists some $\bar{c} \in \interior(\Omega)$ such that $\langle E(c),\, \bar{c} - c \rangle < 0$ for all $c \in \partial\Omega$ (the boundary of set $\Omega$), then there must be an interior equilibrium $c^* \in \interior(\Omega)$ with $z(c^*) = x(c^*)$.}
% Assumptions~\ref{ass:extended},~\ref{asm:proper},~\ref{asm:proper},~\ref{asm:utility}, and~\ref{asm:coercive}, the VI~\eqref{eq:VI_PUME} admits a solution $c^* \in \Omega$.
% If additionally there exists $\bar{c} \in \interior(\Omega)$ such that 
% \begin{equation}\label{eq:additional_interior_cond}
%     \langle E(c),\, c - \bar{c} \rangle > 0
% \end{equation}
% for all $c \in \partial\Omega$, then $c^* \in \interior(\Omega)$ and $z(c^*) = x(c^*)$.
\end{theorem}

\begin{proof}[Proof sketch]
    % The existence is proved by showing that $E$ is coercive, while the interior solution to VI is proved by contradiction. See Appendix~\ref{app:proof_exist} for the detailed proof. 
    The monotone non-increasing demand (Proposition~\ref{prop:monotone_demand}) prevents $x$ from offsetting the growth of $z$, so coercivity transfers from $z$ to $E$, which is the primary condition to ensure the existence of VI solution. See Appendix~\ref{app:proof_exist} for the detailed proof.
\end{proof}

Theorem~\ref{thm:existence} guarantees there exists at least one equilibrium cost vector, while the following theorem further establishes its uniqueness by strengthening the property of $z$, a condition that also underpins the convergence guarantees in Section~\ref{sec:algorithms}.

% The equilibrium uniqueness follows from the strict monotonicity of supply and is formally established below. This property is also used for the algorithmic convergence guarantees in Section~\ref{sec:algorithms}.

\begin{theorem}[PUME uniqueness]\label{thm:uniqueness}
Under the same conditions of Theorem~\ref{thm:existence} and additionally strictly monotone $z$ on $\Omega$, PUME $c^* \in \Omega$ is unique.
% , if $z$ is strictly monotone on $\Omega$:
% \begin{equation}\label{eq:strict_mono}
%     \langle z(c_1) - z(c_2),\, c_1 - c_2 \rangle > 0, \qquad \forall\, c_1 \neq c_2 \in \Omega,
% \end{equation}
% then the PUME $c^* \in \Omega$ is unique.
\end{theorem}
\begin{proof}[Proof sketch]
The strictly monotone $z$, together with monotone non-increasing $x$ (Proposition~\ref{prop:monotone_demand}), yields strictly monotone $E$ and thus uniqueness. See Appendix~\ref{app:proof_unique} for the complete proof.
\end{proof}

In sum, Theorems~\ref{thm:existence} and \ref{thm:uniqueness} outline conditions under which a unique PUME defined in the cost space exists. Specifically, we do not impose additional assumptions on the demand side, but only require the supply function to be coercive for existence and strictly monotone for uniqueness. These two assumptions are also standard in classic traffic assignment models.

% Taken together, Theorems~\ref{thm:existence}--\ref{thm:uniqueness} separate the equilibrium analysis into a demand-side and a supply-side component.
% On the demand side, we impose no parametric structure beyond the PUMCM conditions that ensure the loading problem is well defined and yields a continuous monotone cost-space demand map (Section~\ref{sec:cost_demand}).
% In particular, the demand model may generate corner solutions in the Markovian choice model (zero probabilities on dominated actions), yet the induced demand map and the VI formulation remain well posed.
% Existence and uniqueness are therefore ``pushed'' to transparent requirements on the supply operator: coercivity for existence and strict monotonicity for uniqueness, which are natural in transportation supply modeling and align with the algorithmic guarantees developed in Section~\ref{sec:algorithms}.

% \subsection{Variational Nash equilibrium}\label{sec:vne}
\subsection{Primal formulation in flow space}\label{sec:vne}

% In Section~\ref{sec:vi_formulation}, we formulate PUME in the cost space, while in this section, we show that an equivalent equilibrium can be constructed in the flow space, which is more commonly used in classic traffic assignment models and known as the primal formulation~\citep{baillon2008markovian}. 

The cost-space PUME of Section~\ref{sec:vi_formulation} is natural for the equilibrium analysis and computation. However, classic traffic assignment models predominantly define equilibrium in the flow space, known as the primal formulation~\citep{baillon2008markovian}.
Hence, we proceed to construct the equivalent primal formulation of PUME by defining the feasible link flow set and reinterpreting the aggregate link demand as a perturbed best response associated with another perturbation function connected to the potential function $\Phi$ introduced before. Accordingly, the primal PUME is defined as the solution to another VI problem defined on the link flow space. 
% This subsection constructs an equivalent primal formulation by defining the feasible set in the flow space, reinterpreting the PUMCM demand as a perturbed best response over flows, and endogeneize costs to capture congestion induced by travelers' routing decisions.

% The PUME is formulated in cost space.
% This subsection shows that the same equilibrium admits a flow-space interpretation as a variational Nash equilibrium.
% To do so, we first identify the feasible aggregate flow set, then show that the PUMCM demand is a perturbed best response for a given cost vector, and finally close the model by making costs endogenous through congestion.

% We first define the feasible set of link demand flow of user class $k\in\mathcal{K}$ as follows:
For each user class $k \in \mathcal{K}$, the feasible link flow set is the polyhedron of nonnegative state-action flows satisfying flow conservation:
% For each segment $k \in \mathcal{K}$, we define the feasible action flow set $\mathcal{X}^k$ as the set of non-negative state-action flows that satisfy flow conservation in the absorbing MDP of segment $k$:
\begin{equation}\label{eq:occupancy}
    \mathcal{X}_k \;\coloneqq\; \Bigl\{{x}_k \in \R_+^{|\mathcal{A}_k|} \;:\; \sum_{a \in \mathcal{A}_{k,s}} {x}_k(s,a) - \sum_{s' \in \mathcal{S}_k}\sum_{a' \in \mathcal{A}_{k,s'}} P_k(s\mid s',a')\,{x}_k(s',a') = q_{k,s},\;\; \forall\, s \in \mathcal{S}_k \Bigr\}.
\end{equation}
The equality condition in \eqref{eq:occupancy} enforces the flow conservation at each non-terminal state $s$, that is, the total outflow $\sum_a {x}_k(s,a)$ equals the total inflow transitioning from other states $\sum_{s',a'} P_k(s\mid s',a'){x}_k(s',a')$ plus the exogenous load $q_{k,s}$.
Each $\mathcal{X}_k$ is a polyhedron, hence closed and convex, and nonempty because any optimal policy induces a feasible flow (Proposition~\ref{prop:demand}).
% It is easy to verify that $\mathcal{X}_k$ is a polyhedron and thus closed and convex. It is also non-empty due to the existence of optimal link demand (Corollary~\ref{cor:demand}).

% The conservation condition requires that, at each nonterminal state $s$, the total outflow $\sum_a {x}^k(s,a)$ equals the inflow from transitions plus the exogenous load $q_s^k$.
% Each $\mathcal{X}^k$ is a polyhedron, hence closed and convex; it is non-empty because any proper policy induces a feasible action flow (Corollary~\ref{cor:demand}).

With the class-specific incidence matrix $B_k$ defined in Assumption~\ref{asm:utility}, the feasible aggregate link flow set is
\begin{equation}\label{eq:feasible_set}
    \mathcal{X} \;\coloneqq\; \Bigl\{x \in \R_+^L \;:\; x = \sum_{k \in \mathcal{K}} B_k\tran\,{x}_k,\quad {x}_k \in \mathcal{X}_k\;\;\forall\, k \in \mathcal{K} \Bigr\},
\end{equation}
Since the class-specific link demand $x_{k*}(u_k(c))$ is induced by a proper optimal policy $\pi_{k*}$ for any $c \in \Omega$, we have $x_{k*}(u_k(c)) \in \mathcal{X}_k$ and thus $x(c) \in \mathcal{X}$. 

Next, we derive the dual interpretation of $x(c)$ as the solution to a perturbed best response problem.
% The cost-space demand map $x(c)$ admits a dual interpretation as the solution to a flow-space perturbed best response problem, which underpins the primal VI formulation.
By Proposition~\ref{prop:monotone_demand}, there exists a convex potential $\Phi:\Omega\to\R$ with $\nabla\Phi(c) = -x(c)$. Accordingly, we can define the flow-space perturbation function as
\begin{align}
    R(x) \;\coloneqq\; \Phi^*(-x) \;=\; \sup_{c\in\Omega} \left\{-x\tran c - \Phi(c) \right\},
\end{align}
where $\Phi^*$ denotes the convex conjugate of $\Phi$.
It is easily shown that $R$ is convex in $x$ given that it is a pointwise supremum of affine functions.
% The function $R$ is convex as a pointwise supremum of affine functions. 
% This flow-space perturbation function therefore captures the aggregate perturbation induced by the PUMCM, expressed over flows rather than costs.
The following lemma proves that the aggregate link demand $x(c)$ derived in Section~\ref{sec:cost_demand} is the perturbed best response defined on $R$ with detailed proof in Appendix~\ref{app:proof_PURC}.

\begin{lemma}[Aggregate link demand as perturbed best response]\label{lem:PURC}
Under the conditions of Proposition~\ref{prop:monotone_demand}, for any $c \in \Omega$, the aggregate link demand  $x(c)$ obtained from Eq.~\eqref{eq:potential} also solves the perturbed best response problem:
\begin{equation}\label{eq:PURC}
    x(c) \;\in\; \arg\min_{x \in \mathcal{X}}\; c\tran x + R(x).
\end{equation}
This is equivalent to the stationarity condition $0 \in c + \partial R(x(c))$ with subgradient $\partial R(x)$ at $x$.
\end{lemma}
% \begin{proof}[Sketch of proof]
%     This result is proved by showing $x(c)$ satisfies the optimality condition of the perturbed best response problem \eqref{eq:PURC}. See Appendix~\ref{app:proof_PURC} for the detailed proof. 
% \end{proof}

% The next step is to express the cost-space demand map as a flow-space best response.
% By Theorem~\ref{prop:monotone_demand}, there exists a convex potential $\Phi:\Omega\to\R$ such that
% $\nabla\Phi(c) = -x(c)$.
% We define the corresponding perturbation in flow space by
% \begin{equation}\label{eq:Fenchel_R}
%     R(x) \;\coloneqq\; \sup_{c' \in \Omega}\bigl\{-{c'}\tran x - \Phi(c')\bigr\},
% \end{equation}
% equivalently $R(x) = \Phi^*(-x)$ where $\Phi^*$ denotes the convex conjugate \citep[][Ch. 12]{Rockafellar+1970} of~$\Phi$.
% The function $R$ is convex as a pointwise supremum of affine functions.
% It is the aggregate perturbation induced by the PUMCM in flow space to induce stochastic route choices.

% \begin{lemma}[Demand as perturbed best response]\label{lem:PURC}
% Under the conditions of Theorem~\ref{prop:monotone_demand}, for each $c \in \Omega$ the aggregate demand $x(c)$ solves
% \begin{equation}\label{eq:PURC}
%     x(c) \;\in\; \arg\min_{x \in \mathcal{X}}\; c\tran x + R(x).
% \end{equation}
% Equivalently, $-c \in \partial R(x(c))$, i.e., $0 \in c + \partial R(x(c))$.
% \end{lemma}

Similar to classic traffic assignment models, the primal formulation of PUME requires a link performance function that maps from link flow to link cost. When the following assumption holds, it is naturally derived as the inverse of the supply function. The primal PUME is then formulated on the inverse supply $z^{-1}$ and the newly defined perturbation $R$.

% Lemma~\ref{lem:PURC} treats costs as exogenous: given $c$, the aggregate link demand $x(c)$ solves a perturbed best response problem.
% To derive the equilibrium, costs must arise endogenously from congestion, which requires defining the congestion (link performance) as a function of aggregate link demand $x$. We introduce the following invertibility assumptions which will allow us to define the congestion function as the inverse of supply function.

% Lemma~\ref{lem:PURC} is still a best-response statement with exogenous costs: for a given cost vector $c$, the aggregate flow solves a perturbed cost-minimization problem over $\mathcal{X}$.
% To obtain equilibrium, costs must be generated endogenously by the aggregate flow itself.
% This requires passing from the supply map $z$, which maps costs to supplied flows, to its inverse cost map $t$, which maps feasible aggregate flows back to congestion costs.
% The following assumption ensures that this inverse is well defined on $\mathcal{X}$.

\begin{assumption}[Invertible supply]\label{asm:supply_inverse}
The supply function $z: \Omega \to \R_+^L$ is invertible with continuous inverse $z^{-1}: z(\Omega) \to \Omega$ and $\mathcal{X} \subseteq z(\Omega)$.
% \begin{enumerate}[nosep,label=\textup{(\alph*)}]
%     \item \textit{Invertibility: $z$ is invertible and $\mathcal{X} \subseteq z(\Omega)$;}
%     \item \textit{Continuity:} $z$ and its inverse $z^{-1}: z(\Omega) \to \Omega$ are continuous;
%     \item \textit{Monotonicity:} $z$ is monotone, i.e.,  $\langle z(c) - z(\tilde{c}),\, c - \tilde{c} \rangle \ge 0$ for all $c, \tilde{c} \in \Omega$.
% \end{enumerate}
\end{assumption}

% Define the induced congestion cost function $t \coloneqq z^{-1}: z(\Omega) \to \Omega$.
% By Assumption~\ref{asm:supply_inverse}, $t$ is well defined on $\mathcal{X}$ and is monotone.
% Substituting the endogenous cost $c=t(x)$ into the best-response condition motivates the following definition.

% Assumption~\ref{asm:supply_inverse} ensures that $z^{-1}$ is well defined and continuous on $\mathcal{X}$.
% Substituting the endogenous cost $c = z^{-1}(x)$ into the perturbed best-response condition of Lemma~\ref{lem:PURC} yields the flow-space equilibrium formulation.

% With the inverse supply function $z^{-1}$ well defined on $\mathcal{X}$ in Assumption~\ref{asm:supply_inverse}, we are now ready to define the primal VI formulation.
\begin{definition}[PUME in flow space]\label{def:VNE}
Suppose the supply function satisfies Assumption~\ref{asm:supply_inverse}. A flow vector $x^* \in \mathcal{X}$ is a PUME if there exists $\rho^* \in \partial R(x^*)$ such that
\begin{equation}\label{eq:VNE}
    \langle z^{-1}(x^*) + \rho^*,\, x - x^* \rangle \ge 0, \qquad \forall\, x \in \mathcal{X}.
\end{equation}
\end{definition}
% The primal VI formulation of PUME is closely related to the VI formulation of static traffic assignment, where $z^{-1}(x^*)$ denotes the link cost and $\rho^*\in \partial R(x^*)$ expresses the marginal perturbation. 

The primal VI problem~\eqref{eq:VNE} parallels the classic SUE, where the link demand solves the best response problem defined on the link cost (e.g., network loading). Differently, the equilibrium flow $x^*$ is the perturbed best response with cost defined by the inverse supply $z^{-1}(x^*)$ subject to perturbation $R(x^*)$. In other words, $\rho^* \in \partial R(x^*)$ can be interpreted as the marginal perturbation arising from stochastic route choices.
% Under our setting, $z^{-1}(x^*)$ is the link cost at flow $x^*$, and $\rho^* \in \partial R(x^*)$ can be intepreted as the marginal perturbation arising from stochastic route choices.
The following theorem establishes that the cost-space and flow-space formulations are equivalent when the equilibrium is interior.

\begin{theorem}[Interior PUME equivalence]\label{thm:VNE}
Under the same conditions of Theorem~\ref{thm:existence}
and Assumption~\ref{asm:supply_inverse},
% Under Conditions~\ref{A1}-\ref{A3} and Assumptions~\ref{asm:proper}-\ref{asm:supply_inverse}, 
the aggregate link demand $x^* \coloneqq x(c^*)$ induced by any interior PUME cost $c^* \in \mathrm{int}(\Omega)$ is a PUME flow.
\end{theorem}
\begin{proof}[Proof sketch]
    % The equivalence is established by the interior solutions to both VI problems. 
    The interior PUME satisfies market clearance $z(c^*) = x(c^*)$ and thus $c^* = z^{-1}(x^*)$. Plugging this into Lemma~\ref{lem:PURC} yields $0 \in z^{-1}(x^*) + \partial R(x^*)$, the optimality condition of~\eqref{eq:VNE}.
    The detailed proof is included in Appendix~\ref{app:proof_VNE}. 
\end{proof}

\begin{remark}[General equivalence]\label{rem:general_equiv}
The equivalence established in Theorem~\ref{thm:VNE} uses the market clearance at an interior equilibrium cost. 
% When $c^* \in \mathrm{int}(\Omega)$, the cost-space VI reduces to $z(c^*) = x(c^*)$, so the two PUME are connected by $c^* = z^{-1}(x^*)$. 
When the equilibrium cost $c^*$ lies on the boundary of $\Omega$, market clearance need not hold exactly. The ordinary inverse supply $z^{-1}$ in~\eqref{eq:VNE} should then be replaced by a constrained inverse supply defined on the feasible cost set. 

% For a given $x\in\mathcal X$, this constrained inverse returns a feasible cost vector $\hat c\in\Omega$ satisfying
% \begin{equation}\label{eq:constrained_inverse}
%     \langle z(\hat{c}) - x,\, c' - \hat{c} \rangle \ge 0, \qquad \forall\, c' \in \Omega.
% \end{equation}
% Thus, $\hat c$ is the feasible cost vector whose supply clears the flow $x$ in the variational sense.

% Definition~\ref{def:VNE} recovers the equilibrium cost from the flow via~$z^{-1}(x^*)$.
% This coincides with~$c^*$ when market clearance holds (Theorem~\ref{thm:VNE}), but at the boundary of~$\Omega$ market clearance may fail: $z(c^*) \neq x^*$, so $z^{-1}(x^*) \neq c^*$.
% To extend Definition~\ref{def:VNE} to the boundary, we relax the equation $z(c) = x^*$ to a VI over~$\Omega$, defining the constrained inverse $\hat{c}(x^*)$ as the solution to
% \begin{equation}\label{eq:constrained_inverse}
%     \langle z(\hat{c}) - x^*,\, c' - \hat{c} \rangle \ge 0, \qquad \forall\, c' \in \Omega.
% \end{equation}
% Replacing~$z^{-1}(x^*)$ by~$\hat{c}(x^*)$ in~\eqref{eq:VNE} gives a flow-space PUME that is equivalent to the cost-space PUME (Definition~\ref{def:PUME}) for any $c^* \in \Omega$.
% When $\hat{c}(x^*) \in \mathrm{int}(\Omega)$, the VI reduces to $z(\hat{c}) = x^*$, recovering $\hat{c}(x^*) = z^{-1}(x^*)$ and Definition~\ref{def:VNE}.
\end{remark}

\section{Solution Algorithms}\label{sec:algorithms}

The solution procedure of PUME consists of an inner loop that solves PUMCM at the current cost $c$ and the induced demand $x(c)$, and an outer loop that updates the cost $c$. In this section, we will first present the modified policy iteration (MPI) algorithm for PUMCM (Section~\ref{sec:mpi}), then outline the monotone VI solvers for the outer loop (Section~\ref{sec:base_solvers}). The global convergence is proved for both the inner- and outer-loop solution methods. 
To further accelerate the solution procedure, 
we introduce a meta-algorithm and prove that the global convergence still holds if a merit function is properly constructed (Section~\ref{sec:meta}). 

% Computing a PUME is a nested problem.
% For a fixed cost vector $c$, one must first solve the PUMCM network loading problem and recover the aggregate demand $x(c)$.
% That demand map is then inserted into the cost-space equilibrium condition to update $c$.
% The theory of Sections~\ref{sec:PUMCM}--\ref{sec:equilibrium} determines the resulting algorithmic structure.
% Automatic properness makes the inner Bellman problem well posed over the full PUMCM class, including models with corner solutions, while the equilibrium analysis reduces the outer layer to a monotone VI in cost space.

% The section proceeds in three steps: the PUMCM network loading solver, outer VI solvers, and a safeguarded acceleration framework with global convergence guarantees.

\subsection{Inner loop: Modified policy iteration (MPI)}\label{sec:mpi}

Although the Bellman optimality operator $T_*$ of PUMCM is monotone (Lemma~\ref{lem:Tstar_properties}), it is not contractive in general. Therefore, the standard convergence result for value iteration~\citep[see, e.g.,][Ch.~3]{bertsekas2012dynamic} does not apply. 
We instead implement the modified policy iteration (MPI) proposed in \cite{puterman1978modified}, which alternates between $m$ steps of partial policy evaluation and a greedy policy update. The pseudo-code of MPI is provided in Algorithm~\ref{alg:mpi}, and its global convergence is established below.

% For undiscounted PUMCM, the Bellman operator $T_*$ is monotone but not contractive in general.
% Accordingly, the standard discounted fixed-point argument for value iteration~\citep[see, e.g.,][Ch.~3]{bertsekas2012dynamic} does not apply.
% We therefore study modified policy iteration (MPI)~\citep{puterman1978modified} in the context of PUMCM. Depending on the evaluation depth $m$, MPI interpolates between one-step value iteration ($m=1$) and exact policy iteration ($m \rightarrow \infty$).
% Specifically, MPI alternates greedy policy improvement with a finite number of policy-evaluation sweeps.
% Given $V_n$ at iteration $n$, MPI computes the greedy policy and then applies $m$ evaluation sweeps:
% \begin{equation}\label{eq:mpi_update}
%     \pi_n \coloneqq \nabla H(Q(V_n)), \qquad V_{n+1} \coloneqq T_{\pi_n}^m V_n.
% \end{equation}
% Equivalently, $V_{n+1}$ is obtained by initializing $W^{(0)} = V_n$ and iterating $W^{(h+1)} = U_{\pi_n} + P_{\pi_n} W^{(h)}$ for $h = 0, \ldots, m-1$.
% Each additional evaluation sweep costs one matrix-vector multiply with $P_{\pi_n}$.
% As shown below, increasing $m$ improves the local linear convergence rate.
% Thus $m$ controls the trade-off between per-iteration computation and local convergence speed. We summarize the MPI procedure in Algorithm~\ref{alg:mpi}.

\begin{algorithm}[t]
\caption{Modified Policy Iteration (MPI) for PUMCM}\label{alg:mpi}
\begin{algorithmic}[1]
\REQUIRE initial $V_0 \in [V_\pi, \mathbf{0}]$ for some proper $\pi$; evaluation depth $m \ge 1$; tolerance $\varepsilon_\text{in} > 0$
\FOR{$n = 0, 1, 2, \ldots$}
    \STATE $\pi_n(\cdot|s) \leftarrow \nabla H_s(Q_s(V_n))$ for all $s \in \mathcal{S}$ \hfill (greedy policy)
    \STATE $W \leftarrow V_n$
    \FOR{$h = 1, \ldots, m$}
        \STATE $W \leftarrow U_{\pi_n} + P_{\pi_n} W$ \hfill (policy evaluation sweep)
    \ENDFOR
    \STATE $V_{n+1} \leftarrow W$
    \IF{$\|T_* V_{n+1} - V_{n+1}\|_\infty < \varepsilon_\text{in}$}
        \STATE \textbf{break}
    \ENDIF
\ENDFOR
\end{algorithmic}
\end{algorithm}

% We now turn to the main global guarantee of MPI for the PUMCM network loading step.

\begin{theorem}[Global convergence of MPI]\label{thm:mpi_global}
The MPI iterations over PUMCM with a fixed $m \ge 1$ and initial value $V_0$ induced by some proper policy satisfy
% fix $m \ge 1$ and initialize $V_0 = V_{\pi_0}$ for some proper policy $\pi_0$.
% The MPI iterates~\eqref{eq:mpi_update} satisfy:
\begin{enumerate}[nosep, label=\textup{(\roman*)}]
    \item $V_0 \le V_1 \le \cdots \le V_n \le V_* \le \mathbf{0}$;
    \item $V_n \to V_*$ element-wise as $n \to \infty$.
\end{enumerate}
\end{theorem}

\begin{proof}[Proof sketch]
The theorem is proved through an intermediate Lemma~\ref{lem:subsolution}, which shows inductively that the MPI iterates improve monotonically and stay bounded. Hence, the iterates converge to some $\bar V$. Then a direct sandwiching argument ($V_n\le T_*V_n\le V_{n+1}$) yields that $\bar V$ is a fixed point of $T_*$, which coincides with the optimal value $V_*$ by its uniqueness (Proposition~\ref{prop:Vstar}). See Appendix~\ref{app:proof_mpi_global} for the detailed proof.
% The theorem is proved through an intermediate Lemma~\ref{lem:subsolution} and \ref{lem:fixedpoints_Tpim}. The former shows inductively that the MPI iteratively improves value such that $V_n\le V_{n+1}$ for all iterates $n$. The latter proves that the operator defined by the $m$-step policy evaluation admits the same unique solution as the Bellman operator, which further coincides with the optimal value due to its uniqueness (Proposition~\ref{prop:Vstar}). See Appendix~\ref{app:proof_mpi_global} for the detailed proof.
\end{proof}
% Since $V_0=V_{\pi_0}$ for a proper policy, greedy improvement gives $V_0 \le T_*V_0$ (Lemma~\ref{lem:Tstar_properties}).
% Lemma~\ref{lem:subsolution} then shows inductively that the MPI iterates satisfy $V_{n+1}\ge V_n$ for all $n$, a nondecreasing sequence.
% Since $V_0 \le V_*$ by Proposition~\ref{prop:Vstar}, induction then shows that $V_n \le V_*$ for all $n$.
% Thus $(V_n)$ is componentwise nondecreasing and bounded above by $V_*$, so it converges componentwise to some $\bar V \le V_*$.

% By continuity of $V \mapsto \pi(V) \coloneq \nabla H(Q(V))$ and $(\pi,V)\mapsto T_\pi^mV$, the limit satisfies $\bar V = T_{\pi(\bar V)}^m \bar V$.
% Lemma~\ref{lem:fixedpoints_Tpim} then yields $\bar V = T_{\pi(\bar V)}\bar V$.
% Since $\pi(\bar V)$ is greedy at $\bar V$, we have $T_{\pi(\bar V)}\bar V = T_*\bar V$ by Lemma~\ref{lem:fenchel-young}, so $\bar V$ is a fixed point of $T_*$.
% Uniqueness therefore implies $\bar V = V_*$ (Proposition~\ref{prop:Vstar}).
% The full proof is given in Appendix~\ref{app:deferred-sec5}.

Theorem~\ref{thm:mpi_global} gives the global convergence guarantee for the inner loop that applies to the entire PUMCM class.
% developed in Section~\ref{sec:PUMCM}, including demand models that generate corner solutions.
The following result further quantifies the local convergence rate based on the policy evaluation step $m$.
% We next quantify the local convergence speed of MPI and relate it directly to the evaluation depth $m$.

% \begin{theorem}[Local linear convergence of MPI]\label{thm:mpi_local}
% Under the conditions of Theorem~\ref{thm:mpi_global}, let $\mathcal{M}_m(V)\coloneqq T_{\pi(V)}^mV$ and $G_*\coloneqq P_{\pi_*}$ with $\pi_*=\nabla H(Q(V_*))$.
% Fix any matrix norm $\|\cdot\|$ with $\|G_*^m\| < 1$ and any $\varepsilon \in (0, 1 - \|G_*^m\|)$.
% Then:
% \begin{enumerate}[nosep,label=\textup{(\roman*)}]
%     \item there exists a neighborhood $\mathcal{U}$ of $V_*$ such that $\|\mathcal{M}_m(V) - V_*\| \le \nu\,\|V - V_*\|$ for all $V \in \mathcal{U}$, where $\nu = \|G_*^m\| + \varepsilon < 1$;
% \item $\limsup_{n \to \infty} \|V_{n+1} - V_*\| / \|V_n - V_*\| \le \|G_*^m\|$.
% \end{enumerate}
% \end{theorem}
\begin{theorem}[Local linear convergence of MPI]\label{thm:mpi_local}
Let $O_m(V) \coloneqq T^m_{\pi(V)}V$, where $\pi(V)$ denotes the greedy policy obtained at value $V$, and consider the transition $P_{\pi_*}$ at the optimal policy $\pi_*$. 
With any matrix norm $\|\cdot\|$ such that $||P^m_{\pi_*}|| < 1$ and $\varepsilon \in (0, 1 - ||P^m_{\pi_*}||)$,
\begin{enumerate}[nosep,label=\textup{(\roman*)}]
    \item there exists a neighborhood of $V_*$, denoted as $\mathcal{N}(V_*)$, and some constant $\nu = ||P^m_{\pi_*}|| + \varepsilon < 1$ such that $||O_m(V) - V_*|| \le \nu||V - V_*||$ for all $V \in \mathcal{N}(V_*)$;
    \item $\limsup_{n \to \infty}||V_{n+1} - V_*|| / ||V_n - V_*|| \le ||P^m_{\pi_*}||$.
\end{enumerate}
\end{theorem}
\begin{proof}[Proof sketch]
The local convergence is established based on Lemma~\ref{lem:mpi_jacobian}, which shows $O_m$ is differentiable at $V_*$ with gradient $\nabla O_m(V_*) = P^m_{\pi_*}$. The local convergence rate is then established within the neighborhood of $V_*$ specified by $\varepsilon$. The full proof is in Appendix~\ref{app:proof_mpi_local}.
\end{proof}

% We show in Lemma~\ref{lem:mpi_jacobian} that $\mathcal{M}_m$ is differentiable at $V_*$, with $\nabla \mathcal{M}_m(V_*) = G_*^m$.
% Since $V_* = \mathcal{M}_m(V_*)$, we have
% \[
%     \mathcal{M}_m(V) - V_* = G_*^m(V - V_*) + [\mathcal{M}_m(V) - V_* - G_*^m(V - V_*)],
% \]
% where the term in the bracket is $o(\|V - V_*\|)$ because $\mathcal{M}_m$ is differentiable at $V_*$.
% Hence, for any $\varepsilon \in (0,1-\|G_*^m\|)$, there exists $r > 0$ such that $\|V-V_*\| < r$ implies $\|\mathcal{M}_m(V) - V_*\| \le (\|G_*^m\| + \varepsilon)\,\|V - V_*\|$, proving item~(i) with $\nu = \|G_*^m\| + \varepsilon$.
% Since $V_n \to V_*$ by Theorem~\ref{thm:mpi_global}, the iterates eventually enter $B_r(V_*)$; applying with $V = V_n$ and then letting $\varepsilon \downarrow 0$ yields item~(ii).
% The full proof is in Appendix~\ref{app:deferred-sec5}.

Theorem~\ref{thm:mpi_local} indicates $P_{\pi_*}$ and evaluation depth $m$ as key factors that govern the local convergence rate. 
Specifically, with a smaller value of $||P_{\pi_*}||$, fewer policy evaluations are needed to ensure fast convergence. This result gives rise to the potential for an endogenous design of $m$, which is left to future research. 
It is also worth noting that the inner loop problem is independent among user classes and thus can be easily parallelized. Further, the MPI for each inner loop can be warm-started using the previous solution to reduce the computation time.
% It also shows how the evaluation depth $m$ enters the rate, through the bound $\|G_*^m\| \le \|G_*\|^m$.
% Thus deeper policy evaluation suppresses the gap geometrically.

% This completes the inner layer.
% At the outer level, MPI serves as a loading oracle: given $c_n$, it returns the demand evaluation $x(c_n)$ needed to compute the excess supply $E(c_n)=z(c_n)-x(c_n)$.
% The next subsection treats these demand evaluations as given and selects outer methods whose convergence theory matches the monotone VI structure established in Section~\ref{sec:equilibrium}.

\subsection{Outer loop: Monotone VI solvers}\label{sec:base_solvers}

When the supply function satisfies Assumption~\ref{asm:supply}, it is easily proved that the excess supply operator $E$ is monotone and continuous, given that the demand function is also monotone and continuous (Proposition~\ref{prop:monotone_demand}, and see Appendix~\ref{app:proof_unique} for a similar proof).
Accordingly, the dual VI problem \eqref{eq:VI_PUME} shall be naturally solved by projection-based methods. In this study, we implement two monotone VI solvers that both belong to the first-order projection-based method~\citep{korpelevich1976extragradient, facchinei2003finite} but differ in the cost of operator evaluation in each iteration.

\begin{itemize}
    \item \textbf{Solodov--Tseng modified projection method (ST) with line search}~\citep{solodov1996modified}: The ST method performs two operator evaluations and two projections per iteration. Given the current solution $c_n$, it first computes a trial point as 
    \begin{align}\label{eq:ST_fst}
        \hat c_n = \Proj_\Omega(c_n - \lambda_n\, E(c_n)),
    \end{align}
    and then constructs the main iterate as 
    \begin{align}
        c_{n+1} = \Proj_\Omega\!\bigl(c_n - \gamma_n[(c_n-\hat c_n) - \lambda_n E(c_n) + \lambda_n E(\hat c_n)]\bigr).
    \end{align}
    The step size $\lambda_n > 0$ is determined by Armijo-type line search, i.e., shrinking from 1 until it satisfies
    \begin{align}
        \langle E(c_n)-E(\hat c_n), c_n-\hat c_n\rangle \le \frac{\delta \|c_n-\hat c_n\|^2}{\lambda_n}
    \end{align}
    for some fixed $\delta\in (0,1)$; with some fixed $\theta \in (0,2)$, the correction scale $\gamma_n$ is set as 
    \begin{align}\label{eq:ST_lst}
        \gamma_n = \theta\,(1-\delta)\frac{\|c_n-\hat c_n\|^2}{\|(c_n-\hat c_n)-\lambda_n E(c_n)+\lambda_n E(\hat c_n)\|^2}.
    \end{align}

    It has been proved that ST with line search globally converges to a VI solution when the operator is continuous and monotone~\citep[Theorem~3.2]{solodov1996modified}. It thus converges to the unique PUME when the condition in Theorem~\ref{thm:uniqueness} holds. 

    \item \textbf{Adaptive golden ratio algorithm (aGRAAL)}~\citep{malitsky2020golden}: The aGRAAL method reduces the per-iteration cost by performing single operator evaluation and projection, also avoiding line search. Given the golden ratio $\varphi = (1+\sqrt{5})/2$, it maintains an intermediate variable $\theta_n := \varphi \lambda_n/\lambda_{n-1}$ and updates the step size $\lambda_n$ as 
    \begin{align}\label{eq:aGRAAL_fst}
        \lambda_n=\min\left\{\bar\lambda,\; 
        (1/\varphi + 1/\varphi^{2}) \lambda_{n-1},\;
        \frac{\varphi \theta_{n-1}}{4\lambda_{n-1}}
        \frac{\normM{c_n-c_{n-1}}^2}{\|{E(c_n)-E(c_{n-1})}\|_{M^{-1}}^2}\right\},
    \end{align}
    where $\bar\lambda$ denotes the maximum step size, and $\|v\|_M^2 = v\tran M v$ for some positive diagonal matrix $M$. 

    The solution is then updated as 
    \begin{align}
        c_{n+1} = \Proj_\Omega(\hat{c}_n - \lambda_n\, M^{-1} E(c_n)),
    \end{align}
    where the trial point $\hat{c}_n$ is computed as 
    \begin{align}\label{eq:aGRAAL_lst}
        \hat{c}_n = \frac{(\varphi - 1) c_n + \hat{c}_{n-1}}{\varphi}. 
    \end{align}
    Different from ST, the convergence of aGRAAL further requires local Lipschitz continuity of the operator~\citep[][Theorem 2]{malitsky2020golden}. For PUME, this condition holds whenever the demand map $x(c)$ is locally Lipschitz, and the additional smoothness Assumption~\ref{asm:surplus_extra} is sufficient.
\end{itemize}

It is worth noting that any VI solver with a global convergence guarantee for monotone continuous operators is valid for the outer loop problem. In Section~\ref{sec:experiments}, we compare the convergence performance of ST and aGRAAL.

\subsection{Meta-algorithm for acceleration}\label{sec:meta}

% While the VI solver presented in Section~\ref{sec:base_solvers} guarantees global convergence the computational efficiency could be far from satisfactory even in medium-size network (see Section~\ref{sec:exp_solver}). Hence, we propose a meta-algorithm to accelerate the outer loop meanwhile maintain the convergence property. 

The VI solvers presented in Section~\ref{sec:base_solvers} provide a globally convergent outer loop under the monotone and continuous structure of $E$.
However, as first-order methods, they usually require a large number of iterations to converge in practice (see Section~\ref{sec:exp_solver}). 
% However, they are first-order methods and may require many iterations in practice (see Section~\ref{sec:exp_solver}). 
This issue is particularly relevant for PUME because each evaluation of $E$ involves running the inner-loop MPI (see Section~\ref{sec:mpi}). 
One natural strategy for accelerating the outer loop is to implement quasi-Newton type methods, such as the Anderson acceleration~\citep{anderson1965iterative,walker2011anderson} and nonlinear generalized minimal residual (NGMRES)~\citep{washio1997krylov,desterck2012nonlinear}, that construct the next iterate using a history of past steps.
% To reduce the number of outer iterations, we consider extrapolation methods such as Anderson acceleration~\citep{anderson1965iterative,walker2011anderson} and nonlinear GMRES~\citep{washio1997krylov,desterck2012nonlinear}. 
Yet, directly applying these methods does not preserve the global convergence guarantee because the monotone VI solvers do not necessarily generate a contractive mapping. 
% Since the projection maps generated by monotone VI solvers are not necessarily contractive, directly applying these acceleration methods does not automatically preserve the global convergence guarantee.
Therefore, we develop a safeguarded meta-algorithm that is largely motivated by the stabilized Anderson acceleration proposed in \citet{zhang2020globally}. 
% We therefore adopt a safeguarded meta-algorithm. Similar to the globally convergent Anderson acceleration framework of \citet{zhang2020globally}, acceleration is used only to generate candidate iterates, while the base solver remains responsible for convergence. 
The algorithm periodically restarts from the base solver and accepts an accelerated candidate only when it satisfies a merit-based condition, known as \emph{safeguard checking}. 
Accordingly, the restart mechanism exploits the global convergence property of the base VI solver, while the safeguard prevents acceleration steps from degrading the convergence performance, which together guarantees the global convergence.

To construct the safeguard, we first define a merit function as 
\begin{equation}\label{eq:merit_def}
    \zeta(c) \coloneqq \langle E(c),\, r(c)\rangle,
\end{equation}
where $r$ is the natural residual evaluated as 
\begin{equation}\label{eq:res}
    r(c) \coloneqq c - \mathrm{Proj}_\Omega\bigl(c - E(c)\bigr).
\end{equation}
As $r(c)$ vanishes if and only if $c$ solves the VI problem~\eqref{eq:VI_PUME}~\citep[see e.g.,][]{facchinei2003finite}, the merit function $\zeta(c) $ essentially provides a scalar certificate of the equilibrium violation.
The continuity and boundedness properties of $\zeta$ are established in Lemma~\ref{lem:merit} (see Appendix~\ref{app:meta_convergence}), which serve as the foundation of the global convergence of the proposed meta-algorithm (Theorem~\ref{thm:meta}). 

% To construct the safeguard, we use a merit function adapted to VI~\eqref{eq:VI_PUME}. For a feasible cost vector $c \in \Omega$, define the natural residual
% \begin{equation}\label{eq:res}
%     r(c) \coloneqq c - \mathrm{Proj}_\Omega\bigl(c - E(c)\bigr),
% \end{equation}
% which vanishes if and only if $c$ solves~\eqref{eq:VI_PUME}~\citep[see e.g.,][]{facchinei2003finite}. Accordingly, we define the merit function as
% \begin{equation}\label{eq:merit_def}
%     \zeta(c) \coloneqq \langle E(c),\, r(c)\rangle.
% \end{equation}
% The residual $r(c)$ measures violation of the equilibrium condition, and $\zeta(c)$ provides a scalar certificate of this violation.
% By the variational characterization of projection, $\zeta(c) \ge \|r(c)\|^2 \ge 0$ for all $c \in \Omega$, with $\zeta(c)=0$ only at the equilibrium $c^*$.
% Moreover, $\zeta$ is continuous and has bounded sublevel sets (Lemma~\ref{lem:merit}) due to the coercive structure of $E$ (Assumption~\ref{asm:supply}~\ref{asm:supply_coercive}), so a merit-based safeguard keeps the accepted iterates in a bounded region that ensures convergence.

% Given the original VI condition \eqref{eq:VI_PUME}, we define the residual for a feasible cost vector $c$ as 
% \begin{equation}\label{eq:res}
%     r(c) \coloneqq c - \mathrm{Proj}_\Omega\bigl(c - E(c)\bigr),
% \end{equation}
% and a merit function as 
% \begin{equation}\label{eq:merit_def}
%     \zeta(c) \coloneqq \langle E(c),\, r(c)\rangle.
% \end{equation}

Algorithm~\ref{alg:meta} outlines the meta-algorithm. At each iteration, the base solver (e.g., ST and aGRAAL), denoted by $B:\Omega \to \Omega$, first generates a candidate $c^B_{n+1} = B(c_n)$. Unless a restart is triggered (line~5), the acceleration oracle, denoted as $ACC:\mathcal{H}_n \to \Omega$, proposes another candidate $c^{ACC}_{n+1}= ACC(\mathcal{H}_n)$ based on a finite history $\mathcal{H}_n$ (e.g., recent iterates, step sizes, residuals, etc.). $c^{ACC}_{n+1}$ is accepted when it achieves sufficient merit decrease, i.e., $\zeta(c_{n+1}^{\mathrm{ACC}}) \le \eta \rho_n$ (line~8), 
where $\eta\in (0,1)$ is a safeguard factor and $\rho_n$ is a reference merit level. If the condition fails, the fallback option $c^B_{n+1}$ is taken, and a restart is triggered (line~11). The reference $\rho_n$ is constructed as a non-increasing sequence based on the previous merit and a presumed decay factor $\tau\in(0,1)$ (line~17). The update shares the same spirit as the non-monotone line search method~\citep{grippo1986nonmonotone}, which does not require monotonically decreasing merits over consecutive iterates but ensures an overall decay pattern.

\begin{algorithm}[t]
\caption{Meta-algorithm for PUME}\label{alg:meta}
\begin{algorithmic}[1]
\REQUIRE Initial cost $c_0 \in \Omega$; base solver $B$; acceleration oracle $ACC$; safeguard factor $0 < \eta < \tau < 1$; restart period $R\in\mathbb{N}$; gap threshold $\varepsilon_\text{out} > 0$.
\STATE $n \leftarrow 0$, $N_{\mathrm{ACC}} \leftarrow 0$, $\rho_n \leftarrow \zeta(c_0)$
\LOOP
    \STATE $c_{n+1}^{B} \leftarrow B(c_n)$
    % ; \quad $g_n \leftarrow c_n - c_{n+1}^{B}$
    \IF{$n = 0$ \textbf{or} $N_{ACC} \ge R$}
        \STATE $c_{n+1} \leftarrow c_{n+1}^{B}$, \quad $N_{ACC} \leftarrow 0$ \hfill (force base solution and restart)
    \ELSE
        \STATE $c_{n+1}^{ACC} \leftarrow ACC(\mathcal{H}_n)$
        \IF{$\zeta(c_{n+1}^{ACC}) \le \eta\rho_n$}
            \STATE $c_{n+1} \leftarrow c_{n+1}^{ACC}$, \quad $N_{ACC} \leftarrow N_{ACC} + 1$ \hfill (accept accelerated candidate)
        \ELSE
            \STATE $c_{n+1} \leftarrow c_{n+1}^{B}$, \quad $N_{ACC} \leftarrow 0$ \hfill (reject accelerated candidate and restart)
        \ENDIF
    \ENDIF
    \IF{relative residual $\frac{\|r(c_{n+1})\|_\infty}{\max(1,\; \|c_{n+1}\|_\infty)} < \varepsilon_\text{out}$}
        \STATE \textbf{break}
    \ENDIF
    \STATE $\rho_{n+1} \leftarrow \min \bigl\{\rho_n,\; \max\{\tau\rho_n,\, \zeta(c_{n+1})\}\bigr\}$ \hfill (reference update)
    \STATE $n \leftarrow n + 1$
\ENDLOOP
\end{algorithmic}
\end{algorithm}

% The key feature of Algorithm~\ref{alg:meta} is the separation between convergence and acceleration. The base solver supplies the convergence guarantee, while the oracle only proposes feasible candidates. The following theorem states the resulting global convergence property.

In this study, we implement Algorithm~\ref{alg:meta} with two acceleration oracles. For both oracles, we use $m_{ACC}$ to denote the memory depth. 
\begin{itemize}
    \item \textbf{Anderson acceleration (AA, Type-I)}~\citep{anderson1965iterative,walker2011anderson,zhang2020globally}: 
    Denote the fixed-point gap at past iteration $j$ as $g_j = c_j - c_{j+1}^{\mathrm{B}}$. At iteration $n$, AA chooses affine weights $\alpha_j$ ($j=\max(0,n-m_{ACC}), \dots, n-1$) as the optimal solution to the following problem:
    \begin{equation}\label{eq:aa_ls}
        \min_{\mathbf{1}\tran \alpha = 1} \Bigl\|\textstyle\sum_{j} \alpha_j\, g_j\Bigr\|^2 + \eta_{\mathrm{reg}} \|\alpha\|^2,
    \end{equation}
    where the second term serves as a regularizer with parameter $\eta_{\mathrm{reg}} >0$. 
    % Let $g_j = c_j - c_{j+1}^{\mathrm{B}}$ denote the fixed-point residual at $c_j$. Over the window of the $m_{\mathrm{ACC}}$ most recent residuals, AA chooses affine weights $\alpha$ for which the combined residual $\sum_j \alpha_j\, g_j$ is as close to zero as possible under regularization $\eta_{\mathrm{reg}} >0$,
    % \begin{equation}\label{eq:aa_ls}
    %     \alpha^* \in \arg\min_{\mathbf{1}\tran \alpha = 1} \Bigl\|\textstyle\sum_{j} \alpha_j\, g_j\Bigr\|^2 + \eta_{\mathrm{reg}} \|\alpha\|^2,
    % \end{equation}
    
    The accelerated solution is then constructed as 
% and applies the same weights to the corresponding base iterates to form the accelerated iterate:
    \begin{equation}\label{eq:aa_candidate}
        c_{n+1}^{ACC} = \Proj_\Omega\!\Bigl(\textstyle\sum_{j} \alpha_j\, c_{j+1}^{B}\Bigr).
    \end{equation}

    % \kz{Given a memory depth $m_{\mathrm{AA}}$, let $g_j = c_j - c_{j+1}^{\mathrm{B}}$ denote the fixed-point residual and form the residual-difference matrix
    % $Y_n = [\Delta g_{n-m_n}, \ldots, \Delta g_{n-1}]$ from the $m_n \le m_{\mathrm{AA}}$ most recent differences $\Delta g_j = g_{j+1} - g_j$.
    % AA solves the unconstrained least-squares problem with regularization parameter $\eta_{\mathrm{reg}} > 0$:
    % \begin{equation}\label{eq:aa_ls}
    %     \gamma_n^* \in \arg\min_{\gamma \in \R^{m_n}} \|g_n - Y_n \gamma\|^2 + \eta_{\mathrm{reg}} \|\gamma\|^2
    % \end{equation}
    % and derives mixing weights $\alpha \in \mathbb{R}^{m_n+1}$ satisfying $\mathbf{1}\tran \alpha = 1$ via
    % $\alpha_0 = \gamma_1^*$, $\alpha_j = \gamma_{j+1}^* - \gamma_j^*$ for $j = 1, \ldots, m_n-1$, and $\alpha_{m_n} = 1 - \gamma_{m_n}^*$~\citep{zhang2020globally}.
    % The candidate is then the projected affine combination of the corresponding base iterates:
    % \begin{equation}\label{eq:aa_candidate}
    %     c_{n+1}^{\mathrm{ACC}} = \Proj_\Omega\!\Bigl(\textstyle\sum_{j=0}^{m_n} \alpha_j\, c_{n-m_n+j+1}^{\mathrm{B}}\Bigr).
    % \end{equation}}

    \item \textbf{Nonlinear GMRES (NGMRES)}~\citep{washio1997krylov,desterck2012nonlinear}: 
    NGMRES follows a similar structure of AA but differs in two respects. First, it tracks the natural residuals $r_j = c_j - \Proj_\Omega(c_j - E(c_j))$ over the history, instead of the fixed-point gap. 
    Second, the natural residual of the new base solution $c^B_{n+1}$, denoted as $r^B_{n+1}$, is used to solve the weights $\alpha$. The corresponding optimization problem becomes 
    \begin{equation}\label{eq:ngmres_ls}
        \min \Bigl\| r^B_{n+1} + \textstyle\sum_{j} \alpha_j\,(r_j - r^B_{n+1}) \Bigr\|^2 + \eta_{\mathrm{reg}} \|\alpha\|^2,
    \end{equation}
    and the accelerated solution is constructed as
    \begin{equation}\label{eq:ngmres_candidate}
        c_{n+1}^{ACC} = \Proj_\Omega\!\Bigl(c_{n+1}^{B} + \beta_n \textstyle\sum_{j} \alpha_j\,(c_j - c_{n+1}^{B})\Bigr),
    \end{equation}
    where $\beta_n \in (0, 1]$ is a damping parameter that controls the extrapolation step. 

\end{itemize}

The global convergence conditions of the proposed meta-algorithm are summarized in the following theorem. 

\begin{theorem}[Meta-algorithm global convergence]\label{thm:meta}
Suppose the same conditions of Theorem~\ref{thm:uniqueness} hold and, from any initial point $c_0\in \Omega$, the base solver $B$ produces a sequence of solutions satisfying the following conditions:
% Assume the conditions of Theorems~\ref{thm:existence}--\ref{thm:uniqueness} hold, and let $\zeta$ be the merit function~\eqref{eq:merit_def}.
% Suppose the base-solver map $\mathcal{B}_n: \Omega \rightarrow \Omega, \forall n$ satisfies, for every starting point $\hat{c} \in \Omega$:
\begin{enumerate}[nosep,label=\textup{(B\arabic*)}]
    \item \textbf{Uniform iterate bound:} 
    There exists a metric norm $\normM{\cdot}$ and a constant $C \geq 1$ such that $\normM{c_n^B - c^*} \leq C \normM{c_0 - c^*}$ for all $n$;
    % $\zeta(c_n^B) \le \sigma(\zeta(c_0))$ for all $n$, for
    % a non-decreasing map $\sigma:\R_+ \to \R_+$ with $\sigma(\alpha)\to 0$ as $\alpha\to 0^+$.
    % the base-solver iterates $(c_n^{\mathrm{B}})$ initialized at $\hat{c}$ satisfy $\zeta(c_n^{\mathrm{B}}) \le \sigma(\zeta(\hat{c}))$ for all $n$, where $\sigma: \R_+ \to \R_+$ is nondecreasing with $\sigma(0) = 0$ and $\sigma(\alpha) \to 0$ as $\alpha \to 0^+$;
    \item \textbf{Merit convergence:} $\zeta(c_n^B) \to 0$ as $n\to \infty$.
\end{enumerate}
Then, for any $0 < \eta < \tau < 1$ and $R \in \mathbb{N}$, Algorithm~\ref{alg:meta} guarantees the following results:
% Let $\mathcal{A}_n:\mathcal{H}_n \to \Omega$ be any feasible acceleration oracle.
% Then Algorithm~\ref{alg:meta} with $0 < \eta < \tau < 1$ and $R \in \mathbb{N}$ satisfies:
\begin{enumerate}[nosep,label=\textup{(\roman*)}]
    \item \textbf{Merit convergence:} $\zeta(c_n) \to 0$ as $n\to\infty$.
    \item \textbf{Solution boundedness}: The sequence of iterates $\{c_n\}$ is bounded.
    \item \textbf{Equilibrium convergence:} $c_n \to c^*$ as $n\to\infty$, where $c^*$ is the unique PUME.
\end{enumerate}
\end{theorem}

\begin{proof}[Proof sketch.]
The proof of merit convergence considers two scenarios of acceleration steps. 
If only finitely many accelerations are accepted, the algorithm eventually performs only base steps, so Condition~(B2) directly implies merit convergence $\zeta(c_n)\to 0$.
When infinite accelerations occur, the safeguard checking, reference update, and restarting mechanism together ensure merit convergence.
Specifically, an accepted candidate satisfies $\zeta(c^{ACC}_{n+1})\leq \min\{\rho_n, \max\{\tau\rho_n, \eta\rho_n\}\} \leq \tau\rho_n$, so each acceptance shrinks the reference by the factor $\tau < 1$. 
Hence the reference, along with the merits at the accepted steps, decays to zero geometrically.
%  each accepted accelerated iterate satisfies $\zeta(c^{ACC}_{n+1})\leq \min\{\rho_n, \max\{\tau\rho_n, \eta\rho_n\}\} \leq \tau\rho_n$, so the reference is decreased by a fixed factor $\tau < 1$.
% Hence, the reference, along with its accepted merits, tend to zero. 
On the other hand, between consecutive acceptances, the base iterates are restarted from the most recent accepted iterate, which lies progressively closer to $c^*$.
Condition (B1) then keeps the base iterates within a neighborhood of $c^*$, which also shrinks with the accepted merits. Thus, the merits of adopted base iterates vanish as well.
Altogether, the global merit converges $\zeta(c_n)\to 0$.
% Accordingly, we have $\zeta(c_{n+1}) \leq \tau^{\tilde{n}}\rho_0$. The index $\tilde{n}$ increases more slowly than $n$ because the merit could remain the same when a base solution with larger merit is adopted. Nevertheless, Condition (B1) leads to boundedness that ensures $\zeta(c_{n+1})\to 0$. 
The remaining results follow from the merit convergence, and 
the detailed proof is provided in Appendix~\ref{app:meta_convergence}.
\end{proof}

% There are two cases.
% If only finitely many accelerated steps are accepted, then after some iteration the method coincides with the pure base iteration, and \textup{(B2)} gives $\zeta(c_n) \to 0$.
% If infinitely many are accepted, each acceptance reduces the reference by the factor $\tau < 1$ (since $\eta < \tau$ forces $\zeta(c_{n+1}^{\mathrm{ACC}}) < \tau\,\rho_n$), yielding $\rho_{n_j} \le \tau^j\,\rho_0 \to 0$.
% The merit at each accepted iterate satisfies $\zeta(c_{n_j+1}) \le \eta\,\rho_{n_j} \le \eta\,\tau^j\,\rho_0 \to 0$.
% Between accepted steps, the base method generates the iterates and \textup{(B1)} controls their merit relative to the last accepted value.
% Lemma~\ref{lem:merit}~\textup{(M3)} then gives boundedness.
% Finally, any limit point $\bar{c}$ satisfies $\zeta(\bar{c}) = 0$ by continuity (Lemma~\ref{lem:merit}~\textup{(M2)}), hence $\bar{c} = c^*$ by Lemma~\ref{lem:merit}~\textup{(M1)}.
% Since $c^*$ is the unique limit point of a bounded sequence, $c_n \to c^*$.
% The complete proof is in Appendix~\ref{app:deferred-sec5}.
% \end{proof}

In short, Theorem~\ref{thm:meta} indicates that the convergence of the proposed meta-algorithm does not impose any restriction on the acceleration oracle, Instead, it only requires the base solver and its induced merit values to satisfy two mild conditions (B1 and B2). 
In particular, Condition~(B1) requires the corresponding base iterates to remain uniformly controlled by the initial error, while Condition~(B2) simply requests the convergence of the base solver. 
These conditions are standard stability property and typically follow from the convergence proof of the base methods~\citep{bauschke2017convex, solodov1996modified,malitsky2020golden}.
In Appendix~\ref{app:B1_proof}, we further verify that the proposed merit function~\eqref{eq:merit_def}, along with ST and aGRAAL presented in Section~\ref{sec:base_solvers}, indeed satisfies them and thus ensures global convergence of Algorithm~\ref{alg:meta}. 

\section{Specification of PUME components}\label{sec:spec}
This section discusses key components in PUME with a particular focus on their implications for the equilibrium solutions and properties. It starts with the family of surplus functions (Section~\ref{sec:spec_surplus}), followed by a brief discussion of the necessity of strictly negative stage surplus (Section~\ref{sec:spec_neg_surplus}). 
We then demonstrate two particular features of PUME that distinguish it from existing MTE models: i) the capability of producing corner solutions of demand flow (Section~\ref{sec:spec_demand}), and ii) the generalization to non-potential supply function (Section~\ref{sec:spec_supply}).

\subsection{A family of surplus functions}\label{sec:spec_surplus}
Recall that in Section~\ref{sec:surplus}, we establish the correspondence between the surplus function $H_s$ and the perturbation function $H^*_s$ through FY duality (Lemma~\ref{lem:choice_map})
% , along with the induced choice map $\nabla H_s(Q)$. 
In this section, we focus on a family of induced choice maps, named $\alpha$-entmax family~\citep{peters2019sparse,correia2019adaptively}, and show how it covers common surplus functions used in the literature. 
In Appendix~\ref{app:surplus-families}, we present the general conditions that guarantee a well-behaved surplus (see Proposition~\ref{prop:conjugate-construction-body}).

The $\alpha$-entmax family is defined as the solution to the following perturbed utility maximization problem:
\begin{align}\label{eq:general_perturbed_max}
    \max_{\pi \in \Delta_s} \pi\tran Q - H_s^*(\pi;\alpha),
\end{align}
where $H_s^*(\pi;\alpha)$ is the negative Tsallis entropy defined as
\begin{align}\label{eq:tsallis}
    H_s^*(\pi;\alpha) = \begin{cases}
        \mu_s\sum_{a\in\mathcal{A}_s}\pi_a\log \pi_a, & \alpha = 1,\\
        \frac{\mu_s}{\alpha(\alpha-1)}\left(\sum_{a\in\mathcal{A}_s}\pi_a^\alpha - 1\right),&  \alpha >1,
    \end{cases}
\end{align}
When $\alpha = 1$, Eq.~\eqref{eq:tsallis} recovers the Shannon entropy used in recursive logit models~\citep{fosgerau2013link,mai2015nested}:
\begin{align}
    H_s^*(\pi)= \mu_s\sum_{a\in\mathcal{A}_s}\pi_a\log \pi_a,
\end{align}
which corresponds to the log-sum-exp surplus and logit/softmax choice map:
\begin{align}
    H_s(Q) &= \mu_s \log \sum_{a \in \mathcal{A}_s} \exp\left(\frac{Q_a}{\mu_s}\right),\label{eq:log-sum-exp}\\
    \nabla H_s(Q) &= \frac{\exp(Q_a/\mu_s)}{\sum_{a'\in\mathcal{A}_s}\exp(Q_{a'}/\mu_s)}.\label{eq:softmax}
\end{align}
When $\alpha =2$, Eq.~\eqref{eq:tsallis} leads to the quadratic perturbation
\begin{align}\label{eq:sparsemax_proj}
    H_s^*(\pi)= \frac{\mu_s}{2}\|\pi\|_2^2,
\end{align}
and the sparsemax surplus and choice map~\citep{martins2016softmax}:
\begin{align}
    H_s(Q) &= \max_{\pi \in \Delta_s} \pi\tran Q - \frac{\mu_s}{2}\|\pi\|_2^2,\\
    \nabla H_s(Q) &= \text{Proj}_{\Delta_s} (Q/\mu_s).\label{eq:sparsemax}
\end{align} 
% The corresponding choice map $\nabla H_s$ is also known as sparsemax as it tends to produce sparser solutions~\citep{martins2016softmax}.

The $\alpha$-entmax family with $\alpha \in (1,2)$ can produce corner solutions while remaining a continuous choice map of $Q$. However, the corresponding surplus function and choice map no longer have closed-form solutions in general. One exception is the case of $\alpha=1.5$. 
We discuss the computation of the choice map in this special case, as well as the general cases, in Appendix~\ref{app:surplus-families}.

Table~\ref{tab:surplus} summarizes the above cases and their consequential properties of the surplus function and choice map, with detailed discussion provided in Appendix~\ref{app:surplus-families}.

\begin{table}[htb]
    \centering
    \small
    \caption{Properties of surplus function and induced choice map.}
    \begin{tabular}{llcccc}
        \toprule
         & Choice map & Base surplus conditions~ & $C^2$ surplus & Corner solution & Lipschitz \\
         &  & \ref{asm:surplus_convex}--\ref{asm:surplus_equiv} & (Assumption \ref{asm:surplus_extra}) &  & choice map \\
         \midrule
       $\alpha=1$  & logit/softmax & \checkmark & \checkmark & $\times$ & \checkmark \\
       $\alpha\in(1,2)$  & $\alpha$-entmax & \checkmark & \checkmark & \checkmark & \checkmark \\
       $\alpha=2$  & sparsemax & \checkmark & $\times$ & \checkmark & \checkmark \\
       \bottomrule
    \end{tabular}
    \label{tab:surplus}
\end{table}

The choice map as solution to Problem~\ref{eq:general_perturbed_max} at a Q-value vector $Q=[0,-\Delta,-2\Delta]\tran$ with varying $\Delta$ and $\alpha\in\{1,1.2, 1.5,2\}$ is illustrated in Figure~\ref{fig:choice_map}. 
% The properties of different surplus functions are further illustrated in Figure~\ref{fig:choice_map}, where Problem~\ref{eq:general_perturbed_max} is solved with $\alpha\in\{1,1.2, 1.5,2\}$ at a Q-value vector $Q=[0,-\Delta,-2\Delta]\tran$ with varying $\Delta$. 
It is easy to observe that the logit model $(\alpha=1)$ preserves full support at every $\Delta$. Even if the utility of Action 1 strongly dominates ($\Delta = 6$), Action 3, the worst option, still receives a strictly positive choice probability ($\pi_3>0$). 
In contrast, when $\alpha >1$, corner solutions occur at a finite utility gap, i.e., the choice probability reduces to zero for inferior actions at relatively large values of $\Delta$. 
Specifically, a larger value of $\alpha$ leads to sparser but less smooth solutions. For $\alpha\in(1,2)$, the choice map remains continuously differentiable, while it becomes piecewise affine when $\alpha =2$.

\begin{figure}[htb]
    \centering
    \includegraphics[width=1.0\textwidth]{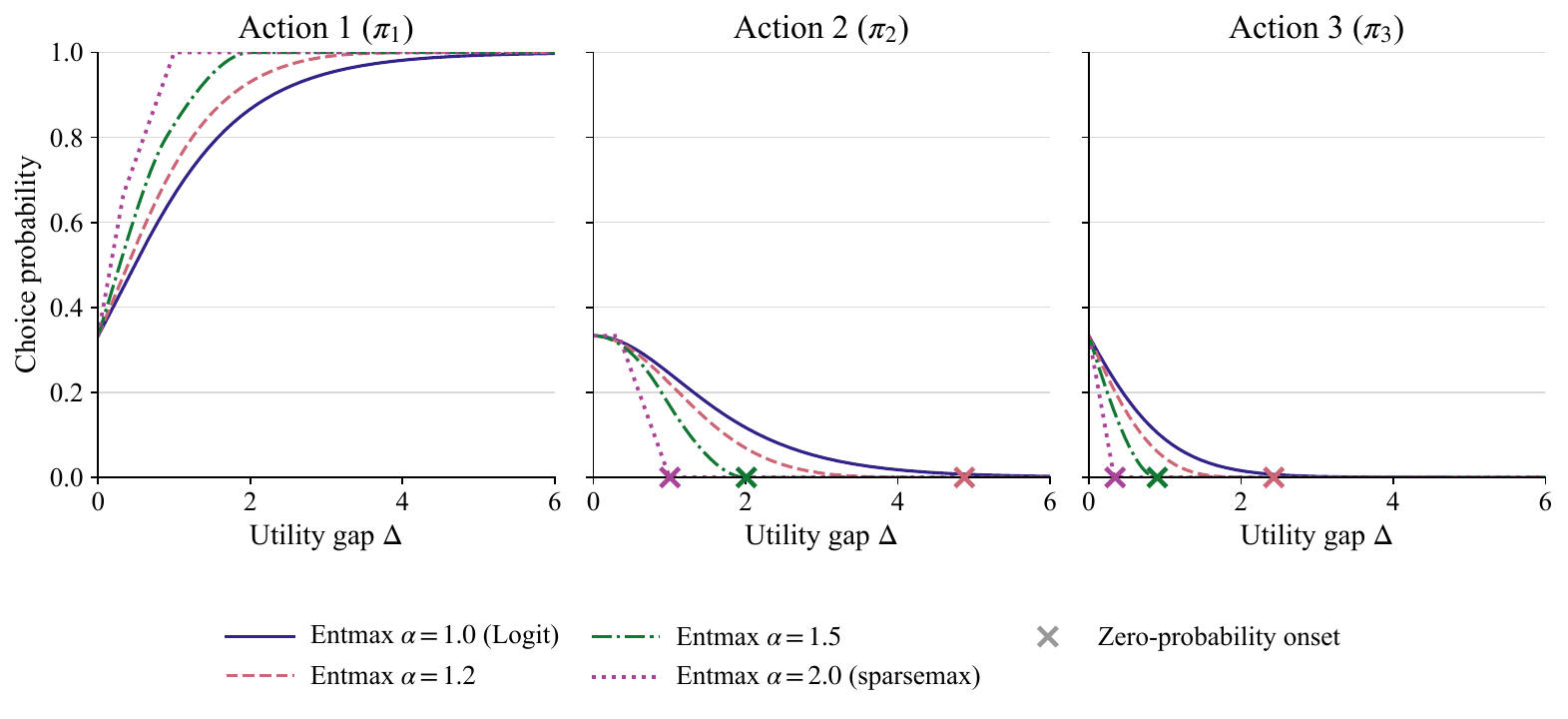}
    \caption{$\alpha$-entmax choice map against utility gap.}
    \label{fig:choice_map}
\end{figure}

To the best of our knowledge, the general $\alpha$-entmax family has not been adopted and analyzed in network equilibrium models. 
Specifically, its boundary behavior falls outside the standard Legendre-type regularity~\citep{Rockafellar+1970} assumed in the existing MTE literature, which generates only strictly positive choice probability~\citep{mcfadden1981econometric, oyama2022markovian}.
Nevertheless, the resulting PUMCM still satisfies the more general regularity conditions stated in Standing Assumption~\ref{asm:surplus}.
It is also worth noting that the corner solution generated by PUMCM provides an 
explicit interpretation of the endogenous consideration set~\citep{caplin2019rational}.
% endogenous consideration set interpretation. 
Particularly, the consideration set is defined for each state as the alternatives with positive support of $\nabla H_s$.
A similar interpretation has also been noted in \cite{tan2024endogenously} for a special case of static perturbed utility model .

\subsection{Strictly negative stage surplus}\label{sec:spec_neg_surplus}
In this section, we use a stylized network, shown in Figure~\ref{fig:small_net}, to show the necessity of the strictly negative stage surplus (Standing Assumption~\ref{asm:proper}\ref{asm:proper_nonpos}).
We solve the PUMCM for destination $d$ under the log-sum-exp surplus ($\alpha = 1$) with scale $\mu_s = 0.5$, and compare two link-cost specifications. 
Specifically, the well-posed case uses the free-flow times shown in Figure~\ref{fig:small_net} as link costs, which yields strictly negative surplus $H_s(u(s,\cdot))$ at every state (Standing Assumption~\ref{asm:proper}\ref{asm:proper_nonpos}). On the other hand, the ill-posed case reduces the cost of links $(o,u),(o,v),(u,o),(v,o)$ to zero such that the stage surplus is no longer strictly negative, thereby violating the assumption. This essentially creates a zero-cost cycle through the origin $o$.

\begin{figure}[htb]
    \centering
    \includegraphics[width=0.65\textwidth]{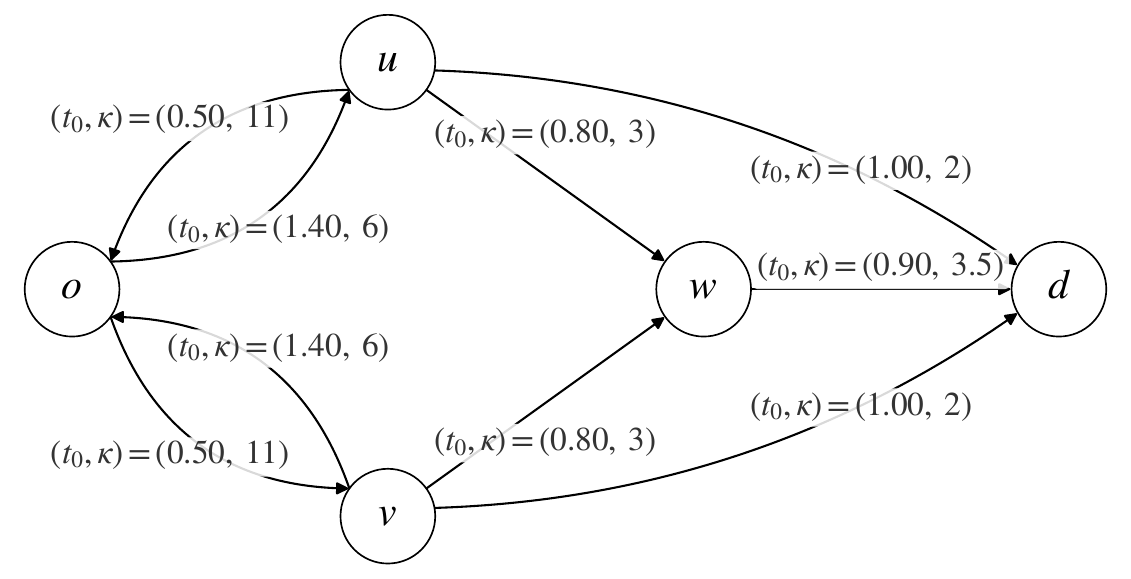}
    \caption{Stylized network.}
    \label{fig:small_net}
\end{figure}

Figure~\ref{fig:small_net_value} reports the value $V_n$ and the Bellman residual $\|T_*V_n - V_n\|_\infty$ over iterations in both tested scenarios. 
With zero-cost cycle, the value iteration diverges in the ill-posed case: the value continuously grows over iterations while the residual stalls. 
% The resulting policy concentrates on the return links, yielding a policy-induced transition matrix with spectral radius $\rho(P_\pi) = 1$, a singular evaluation matrix $(I - P_\pi)$, and zero probability of absorption, which is exactly the failure mode ruled out by Assumption~\ref{asm:proper}.
In contrast, the value iteration in the well-posed case converges linearly as per Theorem~\ref{thm:mpi_local}.

\begin{figure}[htb]
    \centering
    \includegraphics[width=0.88\textwidth]{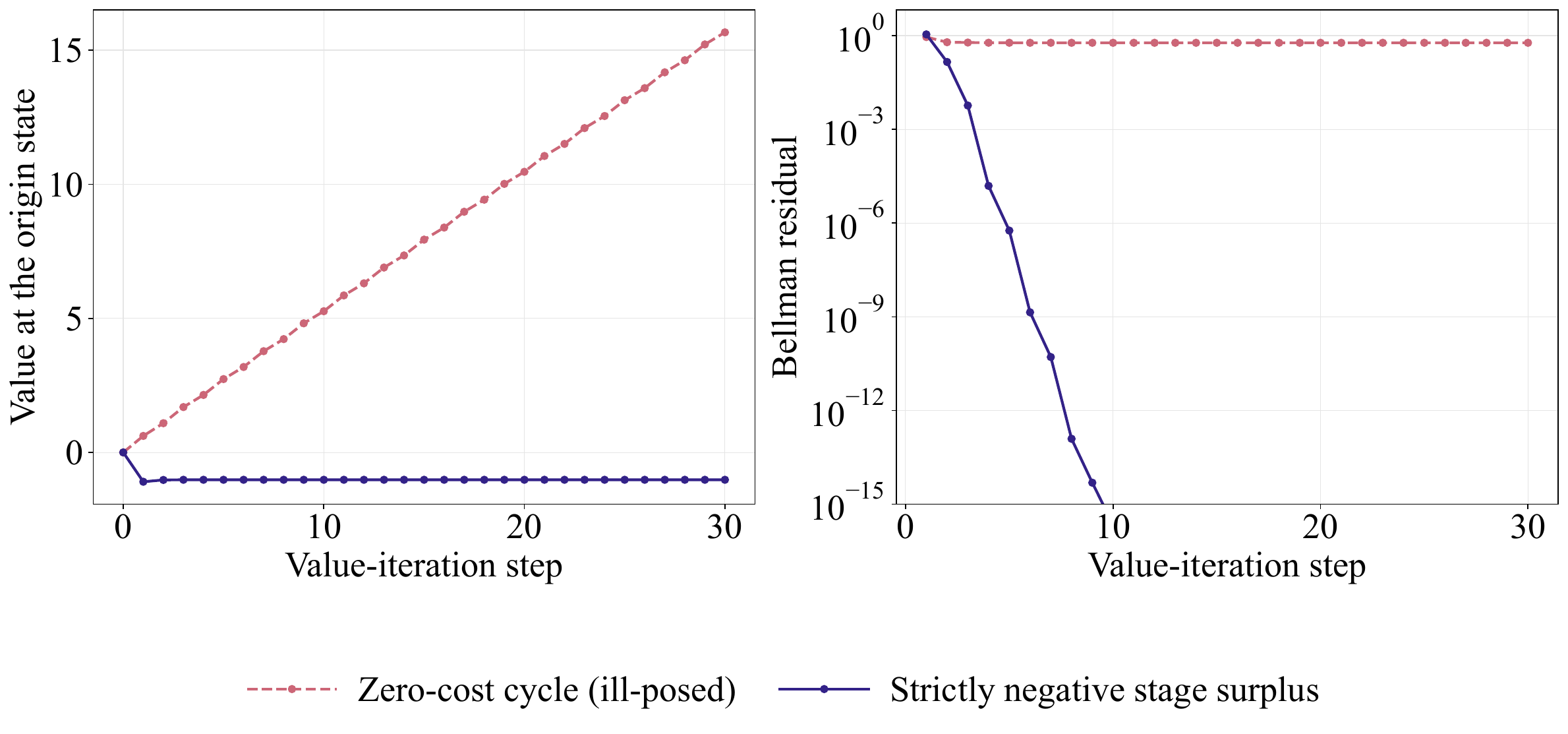}
    \caption{Value iteration on the stylized network.}
    \label{fig:small_net_value}
\end{figure}

\subsection{Interior vs corner solution of demand map}\label{sec:spec_demand}

We continue to use the same network and solve the PUMCM at free-flow travel times using the $\alpha$-entmax family with $\alpha=1$ and $\alpha=1.5$. 
In this experiment, we set a unit demand load  $q=1$ for the single OD pair $o\to d$. Figure~\ref{fig:small_net_demand} illustrates the link flows produced by the two PUMCM frameworks. When $\alpha=1$, the model reduces to the recursive logit~\citep{fosgerau2013link}. Accordingly, the induced demand map is interior and yields a positive flow on every link, even the backward links $(u,o)$ and $(v,o)$. On the other hand, PUMCM with $\alpha =1.5$ produces a more realistic corner solution with exactly zero flow on these links. This example gives strong evidence of the importance of ensuring corner solutions and thus demonstrates a wider applicability of PUMCM compared to existing logit-based models.

\begin{figure}[h]
    \centering
    \includegraphics[width=1.0\textwidth]{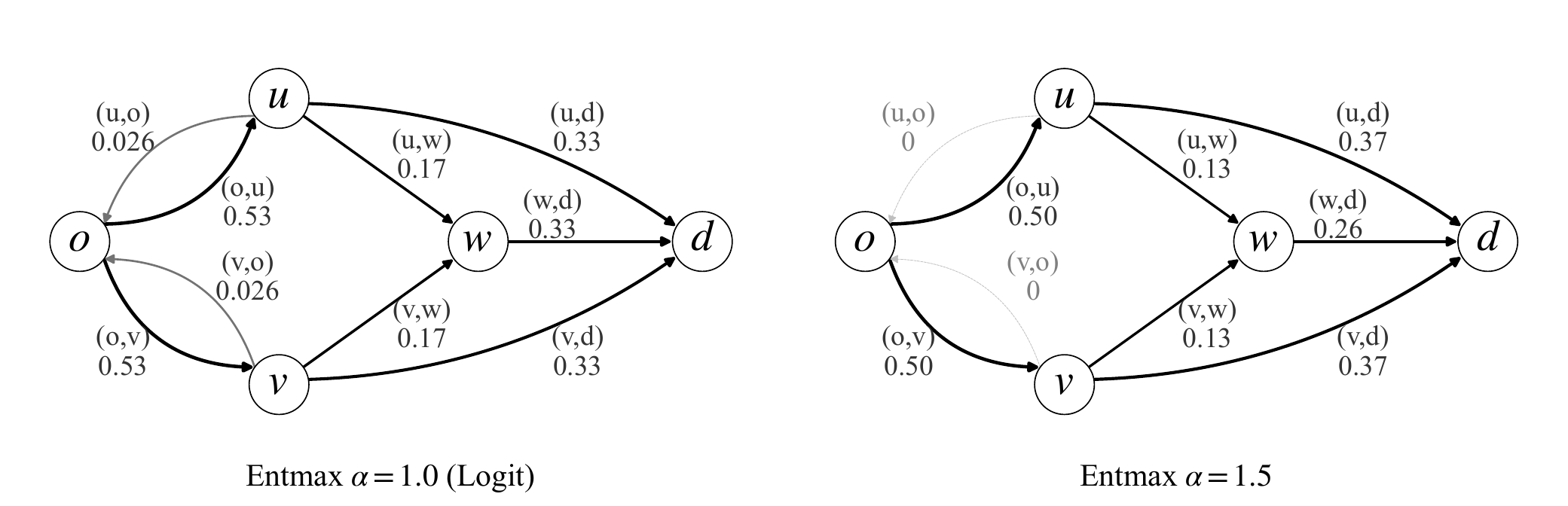}
    \caption{Link flows on the stylized network under recursive logit ($\alpha = 1$, interior solution) and $\alpha$-entmax ($\alpha = 1.5$, corner solution).}
    \label{fig:small_net_demand}
\end{figure}

\subsection{Potential vs non-potential PUME}\label{sec:spec_supply}

With the same setting of surplus functions, we proceed to solve PUME over the same network with uniform demand load $q=4$ on two OD pairs $o\to w$ and $o\to d$. The link travel time follows the standard BPR function
\begin{equation}\label{eq:BPR}
    t(x) = t_{0} \left[1 + 0.15 \left(\frac{x}{\kappa}\right)^4\right],
\end{equation}
where $t_0$ and $\kappa$ denote the free-flow time and capacity, respectively. The link-specific parameters are labeled in Figure~\ref{fig:small_net}. 

In this experiment, we define the vector-valued supply function as
\begin{align}\label{eq:nonpotential_supply}
    z(c) = \left[\mathbb{I} + \iota\,W\right]t^{-1}(c),
\end{align}
where $t^{-1}$ refers to the inverse of BPR function, $W$ is a sparse, row-stochastic asymmetric matrix encoding the interdependence among link travel times (e.g., congestion spillover), and parameter $\iota\in[0,1)$ indicates the coupling intensity. 
% In this example, the detour links $(u,w),(v,w)$ congest the direct links $(u,d),(v,d)$ at the diverge nodes, so a larger $\varepsilon$ makes the direct paths more congested. 
Accordingly, when $\iota=0$, there are no link interactions and the corresponding PUME admits a potential function~\citep{baillon2008markovian}. 
% A direct calculation shows that the symmetric part of $A$ is positive definite whenever $\varepsilon < 1/(1-\lambda_{\min})$, where $\lambda_{\min}$ is the smallest eigenvalue of $(W+W\tran)/2$, so that the monotonicity condition required by the VI formulation is preserved.

Figure~\ref{fig:small_net_supply} plots the fraction of detour flow, i.e., the sum of flows on $(u,w),(v,w)$ divided by the total flow, and the average system travel time (ASTT) against $\varepsilon$. Under both surplus functions, the detour flow increases monotonically with $\varepsilon$ as the coupling makes the direct paths for OD $o\to d$ more congested. Overall, the recursive logit model leads to higher detour flow, whereas the detour probability grows faster in the $\alpha$-entmax model because the coupling destroys the sparsity in route choice. Nevertheless, the average travel time of recursive logit remains higher than that under $\alpha$-entmax due to cycling flows (see Figure~\ref{fig:small_net_demand}).

\begin{figure}[h]
    \centering
    \includegraphics[width=0.8\textwidth]{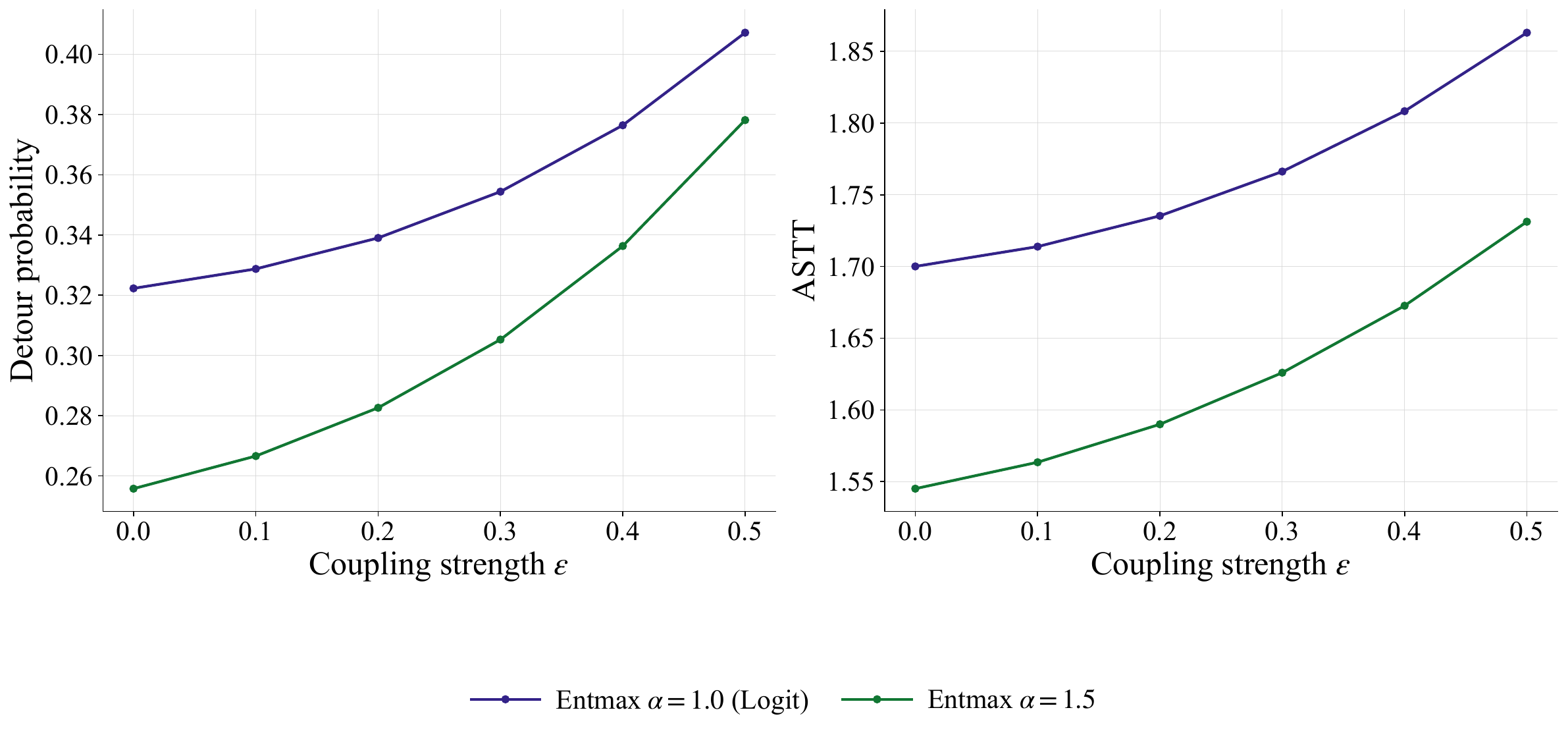}
    \caption{PUME subject to asymmetric link interactions.}
    \label{fig:small_net_supply}
\end{figure}

\section{Numerical experiments}\label{sec:experiments}
% \kz{(CONTINUE FROM HERE)}

The numerical experiments presented in this section are organized in three parts. We first demonstrate the effectiveness of the proposed meta-algorithm using a small-sized benchmark network (Section~\ref{sec:exp_solver}), then test the robustness of PUME across different behavioral models and cost structures (Section~\ref{sec:exp_robustness}), and finally perform sensitivity analyses against the network size, demand, and algorithm parameters (Section~\ref{sec:exp_sensitivity}). 
% We now turn to the computational performance of the framework developed in Section~\ref{sec:algorithms}.
% The experiments are organized in three parts.
% We first identify a practical solver configuration for the framework.
% We then examine whether the same computational scheme remains effective across the model classes covered by the paper.
% Finally, we study how computational performance varies with problem size, demand intensity, and MPI depth.

All test instances use the inverse of the BPR function Eq.~\eqref{eq:BPR} as the supply function and are solved with the proposed meta-algorithm.
Unless otherwise noted, we use the relative natural residual (line 14 in Algorithm~\ref{alg:meta}) as the stopping criterion with tolerance $\varepsilon = 10^{-5}$, a maximum of 2$\times 10^4$ outer iterations, and default MPI evaluation depth $m=10$.
% \begin{equation}
%     \frac{\|r_{\mathrm{nat}}(c)\|_\infty}{\max(1, \|c\|_\infty)} \le \varepsilon,
% \end{equation}
% with a maximum of $20000$ outer iterations, and the default MPI evaluation depth is $m=10$.
The detailed setting of base solver, acceleration oracle, and network construction are presented in Appendix~\ref{app:solver_params}. 
% Additional solver settings, oracle parameters, and synthetic-network construction are collected in Appendix~\ref{app:solver_params}. 
All experiments are conducted on a laptop with an Apple M4 Pro chip and 48 GB of memory. The implementation is in PyTorch.
% , \kz{with selected routines accelerated through C bindings.}

\subsection{Effectiveness of the meta-algorithm framework}\label{sec:exp_solver}

We test the meta-algorithm with and without acceleration on the Sioux Falls network (76 links, 24 nodes) with nested recursive logit (NRL) demand~\citep{mai2015nested}, a special type of recursive logit model with state-specific scale parameter $\mu_s$ (see Eq.~\eqref{eq:softmax})
% We begin with the well-known SiouxFalls network (76 links, 24 zones) under nested recursive logit (NRL) demand~\eqref{eq:logit_choice} and inverse BPR supply~\eqref{eq:BPR}.
% Table~\ref{tab:solver_comparison} reports six meta-algorithm configurations, namely aGRAAL with and without metric preconditioning, two accelerated variants that pair variable-metric aGRAAL with AA1 and NGMRES, and the Solodov-Tseng modified-projection method with and without AA1 acceleration.

\begin{table}[h]
\centering
\caption{Solver comparison on Sioux Falls\\(NRL with inverse BPR supply).}
\small
\begin{tabular}{lrrrc}
\toprule
Meta-algorithm config. & Iter & Eval & Time\,(s) & Rel.\ $r_{\mathrm{nat}}$ \\
\midrule
% aGRAAL            & 20000$^\dagger$ & 20002 & 30.4   & $2.8\!\times\!10^{-1}$ \\
aGRAAL          & 20000$^\dagger$ & 20002 & 30.4   & $2.7\!\times\!10^{-2}$ \\
aGRAAL + AA1    & \textbf{672}      & \textbf{784}     & \textbf{1.2} & \textbf{${9.9\!\times\!10^{-6}}$ }\\
aGRAAL + NGMRES & \textbf{1140}           & \textbf{1179}  & \textbf{1.8}   & \textbf{${4.9\!\times\!10^{-6}}$ }\\
ST                & 20000$^\dagger$ & 40011 & 58.9   & $1.7\!\times\!10^{-2}$ \\
ST + AA1          & 20000$^\dagger$ & 55105 & 80.5  & $2.1\!\times\!10^{-4}$ \\
ST + NGMRES  & 20000$^\dagger$ & 40346 & 59.4  & $5.0\!\times\!10^{-2}$\\
\bottomrule
\end{tabular}
\par\medskip
\footnotesize $^\dagger:$ terminate at maximum iteration without satisfying the stopping criterion. 
% terminal iteration reached without satisfying the stopping criterion. 
% \\Eval. is the total number of operator evaluations of $E(\cdot)$. 
\label{tab:solver_comparison}
\end{table}

Table~\ref{tab:solver_comparison} reports the main outcomes of six configurations. Specifically, the column ``Iter'' reports the number of outer iterations, and the column ``Eval'' refers to the number of evaluations of operator $E$. 
Although both base solvers guarantee global convergence, their overall efficiency is not satisfactory as neither of them converges within the maximum iteration number and the final gap is far from the commonly applied threshold in traffic assignment algorithms. 

When equipped with acceleration, most solvers manage to reach a much smaller gap (except for ST + NGMRES) whereas the convergence performance highly depends on the base solver and the acceleration oracle. Overall, aGRAAL outperforms ST and AA1 outperforms NGMRES, while the advantage of the base solver is dominant. The combination aGRAAL + AA1 achieves the best outcome in the tested scenarios, converging at 672 iterations and 1.2 seconds. 
The second place aGRAAL + NGMRES also converges fast at 1140 iterations and 1.8 seconds. In contrast, ST is not able to converge even with acceleration and requires more operator evaluations that largely extend the computation time. The comparison is better illustrated in Figure~\ref{fig:solver_traces} in Appendix~\ref{app:solver_params}.
These results thus demonstrate the practical importance of the meta-algorithm framework while the choice of acceleration oracle could be problem-specific.

% Table~\ref{tab:solver_comparison} shows three patterns.
% First, metric preconditioning improves the convergence performance of the aGRAAL: aGRAAL-M reduces the relative natural residual by roughly one order of magnitude with the same number of iterations.
% Second, this improvement alone is not sufficient to make the base methods practically usable on this benchmark within the prescribed iteration budget: among the six configurations, only the accelerated variants reach the stopping tolerance.
% Third, once the same variable-metric base method is embedded in the safeguarded meta-algorithm, both AA1 and NGMRES become effective, with AA1 reaching the tolerance in 672 outer iterations and 1.2 seconds, versus 1140 iterations and 1.8 seconds for NGMRES. The difference between the two oracles is therefore secondary to the larger change induced by the meta-algorithm itself. We show in Appendix Figure~\ref{fig:solver_traces} the iteration trajectories of all six configurations.

% Taken together, these results show that the practically important ingredient is the meta-algorithm framework, while the choice of acceleration oracle is benchmark-dependent. Accordingly, we carry forward the two convergent meta-algorithm configurations in the remaining experiments, using aGRAAL-M + AA1 as the default specification.

\subsection{Robustness across model classes}\label{sec:exp_robustness}

We proceed to examine whether the same solution framework remains effective across model classes covered by PUME.
To do so, we keep the algorithm parameters fixed and vary modeling components.
On the demand side, we consider $\alpha$-entmax with $\alpha=1.5$, logit, and NRL. Specifically, logit is a simplified version of recursive logit with common scale parameter $\mu_s=1, \forall s$ in Eq.~\eqref{eq:softmax}~\citet{fosgerau2013link}.
On the supply side, we consider both the potential and non-potential specifications, where the latter is generated by the asymmetric coupling function Eq.~\eqref{eq:nonpotential_supply} with $\iota = 0.1$.
We adopt aGRAAL as the base solver accelerated with AA1 or NGMRES and test on Sioux Falls, Anaheim, and Chicago~Sketch, the widely used benchmark networks at different sizes.
% The benchmark networks are SiouxFalls, Anaheim, and Chicago~Sketch.
% This gives 18 instances, and we solve each one with both AA1 and NGMRES, for a total of 36 runs.

\begin{table}[h]
\centering
\caption{Runtime performance across model classes (seconds).}
\small
\begin{tabular}{ll cc cc cc}
\toprule
& & \multicolumn{2}{c}{SiouxFalls} & \multicolumn{2}{c}{Anaheim} & \multicolumn{2}{c}{Chicago Sketch} \\
& & \multicolumn{2}{c}{\footnotesize (76 links, 552 ODs)} & \multicolumn{2}{c}{\footnotesize (914 links, 1406 ODs)} & \multicolumn{2}{c}{\footnotesize (2950 links, 149382 ODs)} \\
\cmidrule(lr){3-4} \cmidrule(lr){5-6} \cmidrule(lr){7-8}
Supply & Demand & AA1 & NGMRES & AA1 & NGMRES & AA1 & NGMRES \\
\midrule
\multirow{3}{*}{Potential}
& $\alpha$-entmax     & 1.5  & 1.6  & 4.1  & 6.4  & 101  & 62  \\
& Logit               & 1.1  & 1.9  & 3.7  & 6.2  & 92  & 62  \\
& NRL                 & 1.2  & 1.8  & 4.2  & 6.5  & 89  & 68  \\
\midrule
\multirow{3}{*}{Non-potential}
& $\alpha$-entmax     & 1.6  & 1.5  & 5.8  & 7.0  & 125  & 75  \\
& Logit               & 1.1  & 1.7  & 5.4  & 6.6  & 93  & 68  \\
& NRL                 & 1.2  & 1.8  & 4.9  & 6.6  & 97  & 76  \\
\bottomrule
\end{tabular}
\label{tab:robustness}
\end{table}

All tested scenarios converge to the prescribed gap tolerance and their runtime are reported in Table~\ref{tab:robustness}. 
% Table~\ref{tab:robustness} reports the results.
% We first observe that all 36 runs converge to the prescribed tolerance under the same set of solver parameters.
In general, the $\alpha$-entmax demand model requires more computational effort compared to the others possibly due to the existence of corner solutions, whereas non-potential supply does not systematically affect the solution efficiency. 
Interestingly, while AA1 performs better in Sioux Falls and Anaheim, NGMRES leads to consistently shorter runtime in Chicago Sketch. A closer look into the iterations reveals this phenomenon largely results from the stability of the two oracles around the equilibrium (see Figures~\ref{fig:robustness_siouxfalls}--\ref{fig:robustness_chicago} in Appendix~\ref{app:solver_params}).
% Across the tested instances, the corner-solution model $\alpha$-entmax requires effort comparable to the two interior models (logit and NRL), and moving from the potential to the non-potential specification does not systematically affect the performance.
% The relative behavior of the two acceleration oracles also changes smoothly with scale: AA1 is faster on SiouxFalls and Anaheim, whereas NGMRES is more efficient on Chicago~Sketch.

To sum up, the results above show that the same solution framework can be easily adopted to compute PUME under various demand and supply settings without the need for parameter tuning. It thus further emphasizes its practical significance in a wide range of applications. 
% These results are consistent with Sections~\ref{sec:equilibrium}--\ref{sec:algorithms}: the same computational scheme remains effective across both potential and non-potential supply settings, and across both interior and corner-solution demand models, without the need for instance-specific parameter tuning.

% Figure~\ref{fig:robustness_anaheim} shows the corresponding iteration trajectories on Anaheim under both supply types.
% The trajectories for the three demand models remain close, which is consistent with the performance patterns in Table~\ref{tab:robustness}.
% The corresponding plots for SiouxFalls and Chicago~Sketch are reported in Appendix Figures~\ref{fig:robustness_siouxfalls} and~\ref{fig:robustness_chicago}, and show similar patterns. 

\subsection{Sensitivity analyses}\label{sec:exp_sensitivity}

We finally conduct a series of sensitivity analyses using a synthetic grid network. The demand uniformly is distributed in the grid network as shown in Figure~\ref{fig:grid_illustration}, and the network size is expressed by the demand node density $k$. 
The network construction is detailed in Appendix~\ref{app:solver_params}. Throughout this section, we apply aGRAAL + AA1 as the solver and consider $\alpha$-entmax demand with $\alpha=1.5$. 
% We finally turn to the sensitivity of the computation to problem size, demand, and MPI depth.
% To this end, we use synthetic grid networks, which provide a controlled family of increasing size.
% Each grid with parameter $k$ yields a unit of $(4k\!+\!1)\times(4k\!+\!1)$ grid and zone-based gravity demand.

\begin{figure}[h]
\centering
\includegraphics[width=0.5\textwidth]{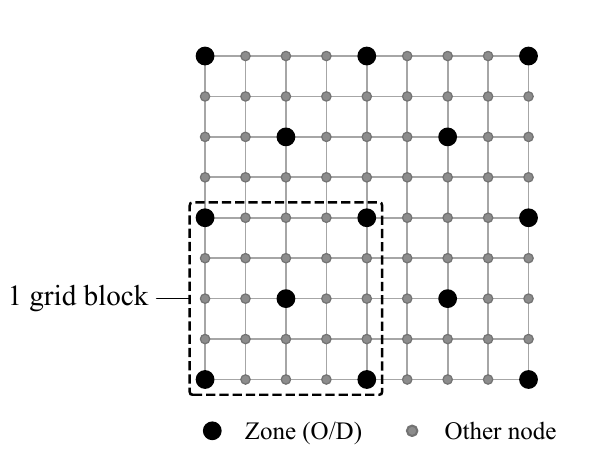}
\caption{Example of synthetic grid network ($k=2$).}
\label{fig:grid_illustration}
\end{figure}
% The detailed construction is provided in Appendix~\ref{app:solver_params}.
% Throughout this subsection, the solver is aGRAAL-M + AA1 with $\alpha$-entmax ($\alpha=1.5$) demand and inverse BPR supply.

\subsubsection{Sensitivity to network size}

We begin with a test on the scalability of the solution framework. Specifically, we increase the network size parameter $k$ while fixing the per-origin demand at $q=10^4$. As shown in Figure~\ref{fig:scalability}, the number of outer iterations until convergence is well bounded as the network size increases. Meanwhile, the total and per-iteration runtime increases linearly with problem size on the log-log scale with a slope approximately equal to 1, meaning that they share an approximate linear relationship in the normal scale. 
In other words, the results demonstrate a reasonable scalability of the solution framework

% We begin with network scalability on the grid family with $k \in \{2,4,6,\ldots,20\}$.
% This range spans a roughly 90-fold increase in problem size, while the per-origin demand is fixed at $q=10\,000$ throughout.
% The comparison is therefore intended to show how the computational burden changes as the dimension of the equilibrium problem increases dramatically.

\begin{figure}[h]
\centering
\includegraphics[width=\textwidth]{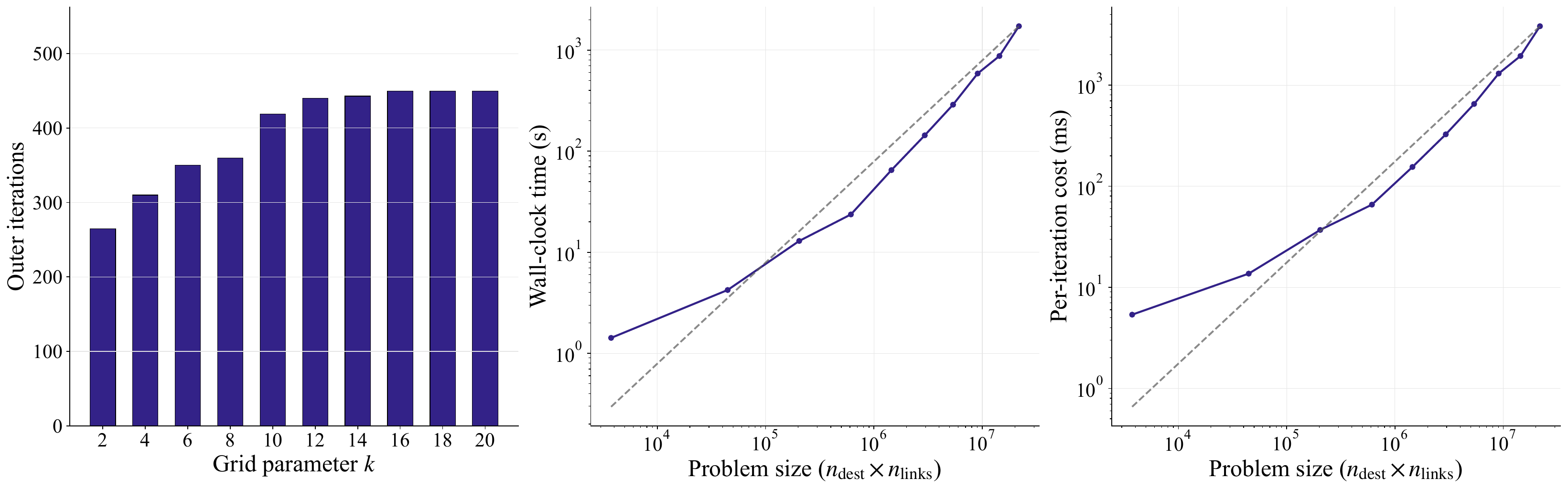}
\caption{Sensitivity results on grid networks with varying sizes.}
\label{fig:scalability}
\end{figure}

% Figure~\ref{fig:scalability} reports three patterns.
% First, the outer iteration count remains bounded over the tested network size range, varying from 265 to 450 despite the large increase in problem size.
% Second, total runtime grows approximately linearly with problem size on the log--log scale.
% Third, the per-iteration cost follows the same pattern, so the increase in total runtime comes from the higher cost of applying the equilibrium operator on larger networks.
% In other words, enlarging the network mainly raises the cost of each iteration, and the number of outer iterations remains relatively stable.
% On the other hand, the approximately linear growth of runtime with respect to problem size (number of links times number of destinations) suggests that larger instances should continue to benefit from faster linear-algebra routines and parallel hardware.

\subsubsection{Sensitivity to demand level}

We next consider three network sizes with $k \in \{5,10,15\}$ and vary the per-origin demand $q \in \{5,7.5,10,12.5,15\}\times 10^3$. To better compare the results across networks, we plot the exact number of outer iterations and the relative runtime ratio (normalized by that at demand $q=5 \times 10^3$). As shown in Figure~\ref{fig:congestion}, neither outer iterations nor runtime is sensitive to the network size, though they both increase with the demand level. 
This result is expected because the higher demand leads to a more congested network, whose equilibrium is typically much harder to solve. Nevertheless, the solution algorithm manages to converge within a reasonable number of iterations (e.g., around 600 at the highest demand level), which showcases its efficiency and scalability.

% For a fixed network, increasing demand provides a practical proxy for stronger congestion effects.

\begin{figure}[h]
\centering
\includegraphics[width=0.8\textwidth]{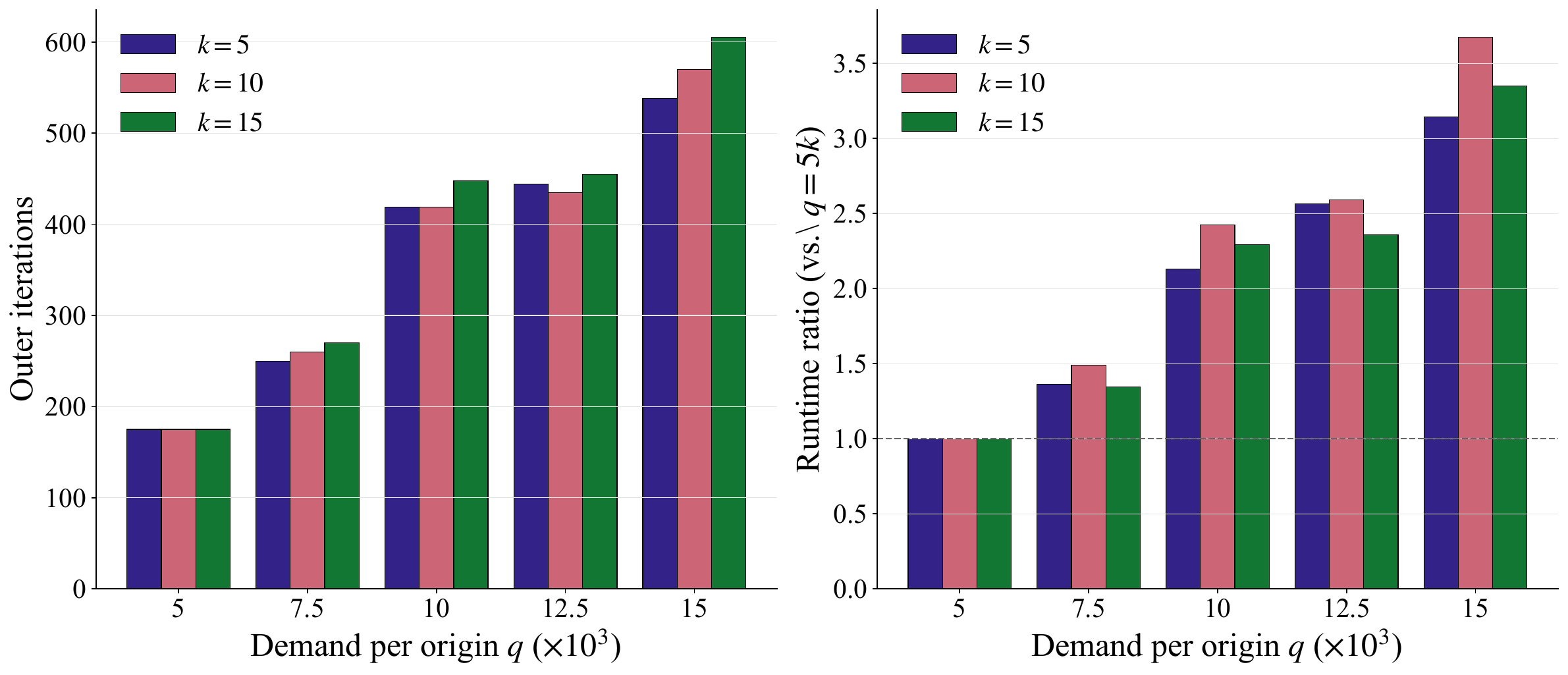}
\caption{Sensitivity results on grid networks at $k=5$, $10$, $15$ with varying demand levels.}
\label{fig:congestion}
\end{figure}

% Figure~\ref{fig:congestion} shows a clear and regular pattern across all three grid sizes.
% For a given demand level, both the outer iteration count and the runtime remain similar across different network sizes of $k=5$, $10$, and $15$.
% Across demand levels, both quantities increase as demand rises.
% When demand triples, the iteration count rises from 175 to 538 at $k=5$, from 175 to 570 at $k=10$, and from 175 to 605 at $k=15$.
% The runtime curves follow the same pattern.
% Overall, higher demand mildly increases the computational effort of the equilibrium solve, while changes in network size within this range have limited impact. This is in line with the general expectation that more congested equilibrium problems are simply harder to solve.

\subsubsection{Sensitivity to MPI evaluation depth}\label{sec:exp_sensitivity_mpi}
We finally examine the choice of MPI evaluation depth $m$. 
Theorem~\ref{thm:mpi_local} implies that a larger $m$ improves the local convergence factor $\nu$ when solving each inner problem of PUMCM. However, it also yields more computation time and may not be beneficial at the very beginning when the outer solution is far from equilibrium. 
To study this trade-off, we test on a grid of $k=5$ with $\alpha$-entmax demand with $\alpha=1.5$ and per-origin demand $q=10^4$ and vary $m \in \{1,2,5,10,20,\infty\}$. Specifically, $m=\infty$ means the policy evaluation (lines 4-6 in Algorithm~\ref{alg:mpi}) converges to the true value.

% the inner PUMCM solve from $\|P_{\pi_*}\|$ to $\|P_{\pi_*}\|^m$, but each evaluation becomes more expensive.
% To study this trade-off, we sweep $m \in \{1,2,5,10,20,\infty\}$ on a grid with $k=5$ and $\alpha$-entmax demand.

\begin{figure}[h]
\centering
\includegraphics[width=\textwidth]{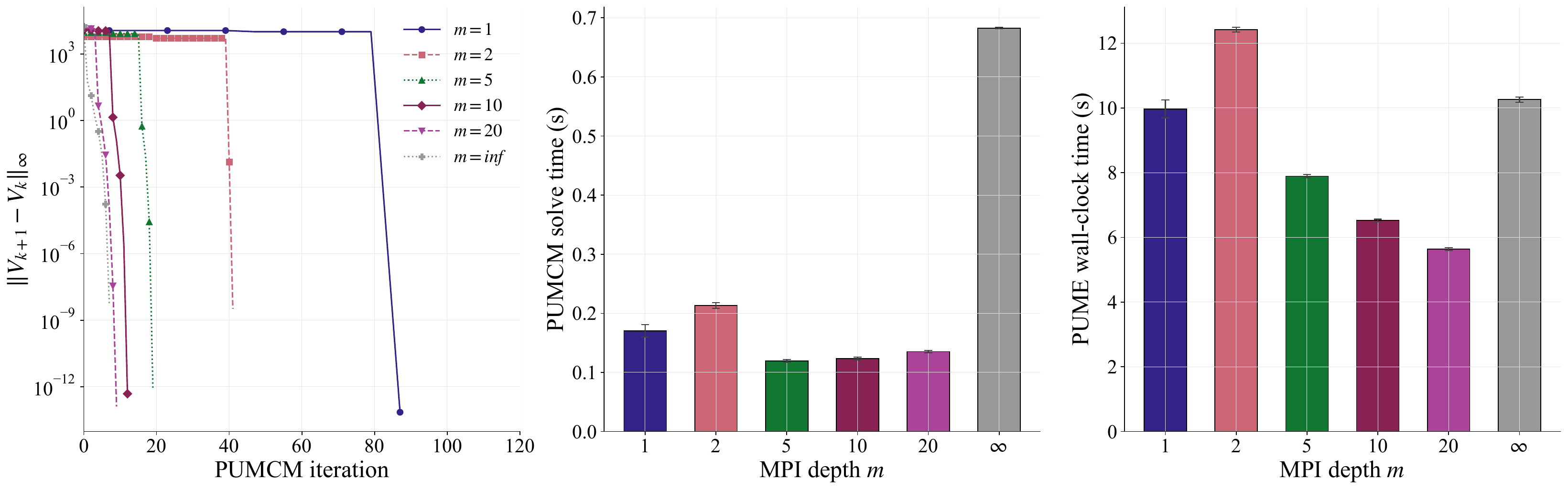}
\caption{Sensitivity results with varying MPI evaluation steps.}
% \caption{MPI evaluation depth trade-off on grid $k=5$ with $\alpha$-entmax demand.
% (a)~Inner PUMCM convergence traces.
% (b)~Per-solve wall-clock time.
% (c)~Total PUME runtime.}
\label{fig:mpi_depth}
\end{figure}

The left panel of Figure~\ref{fig:mpi_depth} plots the value gap over inner iterations. Consistent with Theorem~\ref{thm:mpi_local}, the slowest convergence rate is observed at $m=1$. Specifically, it takes 87 iterations to converge, against 19 at $m=5$ and 7 at $m = \infty$. 
% Yet, the benefit quickly diminishes as $m$ increases, and essentially flat out with $m \in [5, 20]$.
% Meanwhile, $m=\infty$ leads to a significantly longer runtime to solve each inner problem. 
Yet, $m=\infty$ leads to a significantly longer runtime for the first inner problem, as shown in the middle panel. Nevertheless, as the solution proceeds towards equilibrium, the benefit of fewer inner iterations with larger value of $m$ (faster local convergence rate) becomes dominant. Consequently, the total runtime is the shortest at $m=20$, while $m=\infty$ performs comparably well with $m=1$.
% On the other hand, as shown in the right panels, the total runtime to convergence does not show a monotone trend over $m$, even though the number of outer iterations is the same (420 iterations). 
% % Figure~\ref{fig:mpi_depth} reports three panels.
% % Panel~(a) shows how $m$ affects the inner PUMCM solve: at $m=1$ (value iteration) convergence proceeds at the slowest rate, consistent with Theorem~\ref{thm:mpi_local}, while $m=\infty$ (exact policy evaluation via LU factorization) converges in a single step per policy update.
% % The outer PUME iteration count, however, is essentially invariant to $m$ (${\approx}\,911$ iterations across all tested depths), so the runtime differences visible in panel~(c) stem entirely from the per-evaluation cost shown in panel~(b).
% % At small depths the inner solve requires many sweeps, making each ev aluation expensive in aggregate; at $m=\infty$ the LU factorization cost dominates.
% This means the overall PUME runtime is dominated by inner solve. In our example, the sweet spot lies at moderate finite depths $m \in [5,20]$, where larger $m$ cuts the average inner iterations (from 54 at $m=1$ to 12 at $m=5$, to 8 at $m=10$ and to 4.4 at $m=20$) faster than it adds per iteration computation works.

% where the convergence rate is fast enough to keep the inner solve inexpensive without the overhead of a direct factorization.  

%%%%%%%%%%%%%%%%%%%%%%%%%%%%%%%%%%%%%%%%%%%%%%%%%%%%%%%%%%%%%%%%%%%%%%
%%           SECTION 8: CONCLUSION
%%%%%%%%%%%%%%%%%%%%%%%%%%%%%%%%%%%%%%%%%%%%%%%%%%%%%%%%%%%%%%%%%%%%%%
\section{Conclusion}\label{sec:conclusion}

This paper develops the perturbed utility Markovian equilibrium (PUME) framework that consists of three key components: i) a perturbed utility Markovian choice model (PUMCM) that generalizes ARUM-based dynamic route choice models, ii) a variational inequality (VI) formulation defined on the dual cost space that accommodates asymmetric supply, and iii) globally convergent algorithms for both network loading and equilibrium computation.

On the demand side, PUMCM replaces the distributional assumption underlying ARUM-based Markovian models with a convex perturbation function that admits both interior and corner solutions.
The central theoretical basis is the well-posedness of the underlying MDP: under a mild condition of strictly negative stage surplus, any policy consistent with the Bellman equation is guaranteed to reach the destination. 
% This result ensures that the undiscounted MDP is well posed without discounting or topology-dependent restrictions.
This result leads to ideal properties of the optimal value function, the induced demand map, and the convergence guarantees of classic modified policy iteration, creating a unified theoretical and computational framework for the entire model class.

On the equilibrium side, the monotonicity inherited from PUMCM allows the equilibrium to be formulated as a cost-space VI whose existence and uniqueness require only standard conditions on the supply mapping.
The VI formulation naturally handles asymmetric congestion interactions that fall outside the potential-based optimization framework used in previous Markovian equilibrium models.
For computation, a safeguarded meta-algorithm is developed that combines monotone VI solvers with acceleration oracles, preserving global convergence while substantially improving practical speed.

The computational experiments confirm both the scalability and robustness of the proposed framework: the computation time scales linearly with the problem size and demand level, while the same solver configuration accommodates across different demand models, under both potential and non-potential supply. The analysis of the MPI depth further provides practical guidance on the trade-off between convergence speed and per-iteration cost.
% , identifying moderate evaluation depths as the sweet spot between convergence speed and per-iteration cost.

There are several directions remain open.
On the modeling side, estimating the rewards and preferences from observed choices is a natural next step. Specifically, the smoothness-sparsity trade-off characterized in this paper may inform the choice of perturbation family in estimation.
On the computational side, the meta-algorithm framework is flexible and modular, so future work can explore additional acceleration oracles.
Finally, the VI formulation opens the door to multi-class extensions and network design problems where asymmetric interactions occur.

%% ---- Bibliography ----
\Urlmuskip=0mu plus 1mu\relax
\bibliographystyle{abbrvnat}
% Robust bib path whether compiling from `tmp/` or repo root.
\IfFileExists{pume_paper.bib}{%
    \bibliography{pume_paper}%
}{%
    \bibliography{mybib}%
}

%%%%%%%%%%%%%%%%%%%%%%%%%%%%%%%%%%%%%%%%%%%%%%%%%%%%%%%%%%%%%%%%%%%%%%
%%           APPENDICES
%%%%%%%%%%%%%%%%%%%%%%%%%%%%%%%%%%%%%%%%%%%%%%%%%%%%%%%%%%%%%%%%%%%%%%
\newpage
\appendix

\section{Deferred proofs from Section~\ref{sec:PUMCM}}\label{app:deferred-sec3}

\subsection{Proof of Lemma~\ref{lem:choice_map}}\label{app:proof_lem_choice_map}

The weak duality is directly proved by rearranging Eq.~\eqref{eq:conjugate}.
Since $H_s$ is convex and continuously differentiable, for any $Q'\in \R^{|\mathcal{A}_s|}$, 
\begin{align}
    H_s(Q') \ge H_s(Q) + \nabla H_s(Q)\tran(Q' - Q),
\end{align}
which can be rearranged as 
\begin{align}
    \nabla H_s(Q)\tran Q' - H_s(Q') \le \nabla H_s(Q)\tran Q - H_s(Q).
\end{align}
The inequality can be further expanded as
\begin{align}
    \nabla H_s(Q)\tran Q' - H_s(Q') \le \sup_{Q' \in \R^{|\mathcal{A}_s|}} \bigl\{\nabla H_s(Q)\tran Q' - H_s(Q')\bigr\} \le \nabla H_s(Q)\tran Q - H_s(Q). 
\end{align}
and, when plugging in Eq.~\eqref{eq:conjugate}, it further reduces to 
\begin{align}
    H_s^*(\nabla H_s(Q)) \le \nabla H_s(Q)\tran Q - H_s(Q).
\end{align}
Rearranging the inequality and combining it with the weak duality yields, 
\begin{align}
    \max_{\pi\in\Delta_s} \left\{\pi\tran Q - H_s^*(\pi)\right\} \le H_s(Q) \le \nabla H_s(Q)\tran Q - H_s^*(\nabla H_s(Q)), 
\end{align}
which implies the results of strong duality. This completes the proof.

% \subsection{Existence and uniqueness of policy-induced value}\label{app:proof_lem_neumann}

% Let $S_N \coloneqq \sum_{k=0}^{N} P_\pi^k$ .
% Since $\rho(P_\pi) < 1$, we have $P_\pi^{N+1}\to 0$ as $N\to\infty$, and therefore the series $\sum_{k=0}^{\infty} P_\pi^k$ converges to some matrix $S$.
% Moreover,
% \[
% (\mathbb{I}-P_\pi)S_N \;=\; \mathbb{I}-P_\pi^{N+1} \;\longrightarrow\; \mathbb{I},
% \]
% so $(\mathbb{I}-P_\pi)S=\mathbb{I}$, i.e., $(\mathbb{I}-P_\pi)^{-1}=S=\sum_{k=0}^\infty P_\pi^k$.
% Since $P_\pi\ge 0$ entrywise, we have $P_\pi^k\ge 0$ for all $k$, and hence $S\ge 0$ entrywise.
% Finally, $V = U + P_\pi V$ is equivalent to $(\mathbb{I}-P_\pi)V=U$; multiplying by $(\mathbb{I}-P_\pi)^{-1}$ gives the unique solution $V=(\mathbb{I}-P_\pi)^{-1}U$.

% This appendix collects the deferred arguments from Section~\ref{sec:PUMCM}.
% We first gather the proofs behind the Bellman fixed-point and optimality results, and then prove the regularity and monotonicity statement for optimal values and action demand.
\subsection{Proof of Proposition~\ref{prop:well-posedness}}\label{app:proof_well-posedness}

The proof of Proposition~\ref{prop:well-posedness} requires an intermediate result about the existence and uniqueness of policy-induced value, which stated in the following lemma. 

\begin{lemma}[Existence and uniqueness of policy-induced value]\label{lem:neumann}
If policy $\pi$ is proper such that $\rho(P_\pi) < 1$, then the fundamental matrix $(\mathbb{I} - P_\pi)^{-1}$ exists and is nonnegative element-wise.
Consequently, for every $U \in \R^{|\mathcal{S}|}$, the linear system $V = U + P_\pi V$ has a unique finite solution $V = (\mathbb{I} - P_\pi)^{-1} U$. 
\end{lemma}

\begin{proof}
    Consider the geometric series $S_N \coloneqq \sum_{k=0}^{N} P_\pi^k$. Since $\rho(P_\pi) < 1$, we have $P_\pi^{N+1}\to 0$ as $N\to\infty$ and thus $(\mathbb{I}-P_\pi)S_N = \mathbb{I}-P_\pi^{N+1} \to \mathbb{I}.$ 
Besides, $\rho(P_\pi) < 1$ implies $S_N$ converges to some bounded matrix $S$ as $N\to\infty$~\citep[][Corollary~5.6.16]{horn2012matrix}. Accordingly, we have $(\mathbb{I}-P_\pi)S_N \to (\mathbb{I}-P_\pi)S=\mathbb{I}$. It thus implies that $(\mathbb{I} - P_\pi)^{-1}$ exists and equals $S$, which yields the unique solution to the linear system $V = U + P_\pi V$ as $V = (\mathbb{I} - P_\pi)^{-1} U$.  
Moreover, as $P_\pi\ge 0$ element-wise, $(\mathbb{I} - P_\pi)^{-1}$ is nonnegative element-wise as well. This completes the proof. 
\end{proof}

We prove each result in Proposition~\ref{prop:well-posedness} as follows:
\begin{enumerate}[label=\textup{(\roman*)}]
\item By Lemma~\ref{lem:choice_map}\ref{FY:max}, $H_s(u(s,\cdot)) = \max_{\pi \in \Delta_s} \{\pi\tran u(s,\cdot) - H_s^*(\pi)\}$.
Hence, for any policy with $\pi(\cdot|s) \in \Delta_s, \forall s$, we have 
% Since $\pi(\cdot|s) \in \Delta_s$ is feasible, 
$U_\pi(s) = \pi(\cdot|s)\tran u(s,\cdot) - H_s^*(\pi(\cdot|s)) \le H_s(u(s,\cdot)) < 0$.
\item We prove the sufficiency and necessity separately below:

\textit{''Sufficiency''} By Lemma~\ref{lem:neumann}, if $\pi$ is proper and admissible, the linear system $V_\pi = T_\pi V_\pi = U + P_\pi V$ has a unique finite solution $V_\pi = (\mathbb{I} - P_\pi)^{-1} U_\pi$.

\textit{''Necessity''} Suppose there exists an improper policy $\pi$ that satisfies $V = T_\pi V$ for some $V\in\mathbb{R}^{|S|}$. 
Then $P_\pi$ is associated with a nonempty set of non-terminal states, denoted by $C$, that are recurrent. Accordingly, the submatrix $P_{\pi,C}$ is stochastic and associated with a stationary distribution $\nu>0$ such that $ \nu^\top  P_{\pi,C} = \nu^\top $ by the Perron--Frobenius theorem~\citep[][Theorem~8.4.4]{horn2012matrix}.
The corresponding Bellman equation is given by $V_C = U_{\pi,C} + P_{\pi,C} V_C$. Left-multiplying $\nu$ on both sides yields 
\begin{align}
    \nu\tran V_C = \nu\tran U_{\pi,C} + \nu\tran P_{\pi,C} V_C = \nu\tran U_{\pi,C} + \nu\tran  V_C \quad \Rightarrow \quad \nu\tran U_{\pi,C} = 0,
\end{align}
which contradicts Property \ref{prop:universal-strict} proved above that states $U_{\pi}<0$, hence $U_{\pi,C}<0$, element-wise. Therefore, $\pi$ must be proper. 

\item The optimality of policy $\pi_*$ implies $V_* = T_* V_* = T_{\pi_*} V_*$. Then, the properness of $\pi_*$ is directly induced by Property \ref{prop:auto-proper} proved above. 
\end{enumerate}

% \begin{proof}[Proof of Theorem~\ref{thm:auto-proper}]
% \kz{(copied from main text) 
% Suppose $\pi$ is improper; since $P_\pi$ is substochastic, this means $\rho(P_\pi)=1$.
% Then $P_\pi$ has a recurrent class $C$, such that the Markov chain, once in $C$, never leaves.
% The restriction $P_{\pi,C}$ is then stochastic and admits a stationary distribution $\nu>0$ by the Perron--Frobenius theorem.
% Restricting similarly the Bellman equation to $C$ gives $V_C = U_{\pi,C} + P_{\pi,C} V_C$ and left-multiplying by $\nu\tran$ yields $0=\nu\tran U_{\pi,C}$.
% This contradicts Lemma~\ref{lem:universal-strict}, which gives $U_{\pi,C}<0$ componentwise. Hence, such $\pi$ must be proper.}
% Suppose for contradiction that $\pi$ is improper, i.e., $\rho(P_\pi) = 1$ (substochasticity gives $\rho(P_\pi) \le 1$).
% Then $P_\pi$ has a recurrent class $C \subseteq \mathcal{S}$: a minimal nonempty subset from which the chain never leaves.
% The restricted matrix $P_{\pi,C}$ is (row-)stochastic ($P_{\pi,C}\,\mathbf{1} = \mathbf{1}$), so by the Perron--Frobenius theorem it admits a stationary distribution $\nu > 0$ with $\nu\tran P_{\pi,C} = \nu\tran$.
% Restricting $V = T_\pi V = U_\pi + P_\pi V$ to $C$ and left-multiplying by $\nu\tran$ yields
% \[
%   \nu\tran V_C = \nu\tran U_{\pi,C} + \nu\tran P_{\pi,C}\, V_C = \nu\tran U_{\pi,C} + \nu\tran V_C,
% \]
% hence $0 = \nu\tran U_{\pi,C}$.
% But Lemma~\ref{lem:universal-strict} gives $U_{\pi,C} < 0$ componentwise, and $\nu > 0$, so $\nu\tran U_{\pi,C} < 0$---a contradiction.
% \end{proof}
\subsection{Proof of Lemma~\ref{lem:Tstar_properties}}\label{app:proof_Tstar_properties}
We prove each result as follows:
\begin{enumerate}[label=\textup{(\roman*)}]
    \item $T_\pi V_1 - T_\pi V_2 = P_\pi(V_1 - V_2) \ge 0$ as $P_\pi \ge 0$.

    \item For any state $s$, let $\pi_2 \coloneqq \nabla H_s(Q_s(V_2))$, then 
    \begin{align}
        T_* V_1(s) = H_s(Q_s(V_1)) &\ge \pi_2\tran Q_s(V_1) - H_s^*(\pi_2)\\
        &=\pi_2\tran Q_s(V_2) - H_s^*(\pi_2) + \pi_2\tran (Q_s(V_1) - Q_s(V_2)) \nonumber\\
        &= H_s(Q_s(V_2)) + \pi_2\tran P(\cdot|s,\cdot)(V_1 - V_2) \ge T_* V_2(s)\nonumber
    \end{align}
    where the first inequality is due to Lemma~\ref{lem:choice_map}\ref{FY:ineq} the following equality is due to Lemma~\ref{lem:choice_map}\ref{FY:max}, and the last inequality holds when plugging the condition $V_1 \ge V_2$ into Eq.~\ref{eq:Q_def}.
    
%     Fix $s$ and set $\pi_2 \coloneqq \nabla H_s(Q_s(V_2))$.
% Then $T_* V_1(s) = H_s(Q_s(V_1)) \ge \pi_2\tran Q_s(V_1) - H_s^*(\pi_2)$ (weak duality) $= H_s(Q_s(V_2)) + \pi_2\tran [Q_s(V_1) - Q_s(V_2)] \ge T_* V_2(s)$, where the last step uses $\pi_2 \ge 0$ and $Q_s(V_1) - Q_s(V_2) = P(\cdot|s,\cdot)(V_1 - V_2) \ge 0$.

    \item For any state $s$, it holds that 
    \begin{align}
        T_* V_\pi(s) &= \max_{\pi'\in \Delta_s} \{{\pi'}\tran Q_s(V_{\pi}) - H_s^*(\pi')\} \\
        &\ge \pi(\cdot|s)\tran Q_s(V_\pi) - H_s^*(\pi(\cdot|s)) = T_\pi V_\pi(s) = V_\pi(s).\nonumber
    \end{align}
\end{enumerate}

\subsection{Proof of Proposition~\ref{prop:Vstar}}\label{app:proof_Vstar}

% We first establish the properties of Bellman operators $T_\pi$ and $T_*$ that will be used in the proof of Proposition~\ref{prop:Vstar}, summarized in the following lemma. 
% Next, we establish properties of the Bellman operators $T_\pi$ and $T_*$ that are used in the proof of Proposition~\ref{prop:Vstar}.

We prove each result as follows:
\begin{enumerate}[label=\textup{(\roman*)}]
    \item Existence: We prove the existence of $V_*$ by Brouwer's fixed point theorem,
    % ~\kz{(ref)}, 
    which requires two conditions: i) $T_*$ is continuous, and ii) $T_*$ maps from a compact, convex set to itself. 

    As per Eq.~\eqref{eq:T_star}, $T_*$ is continuous if $H_s$ and $Q_s$ are both continuous. The former holds from conditions in Assumption~\ref{asm:surplus} and the latter is due to Eq.~\eqref{eq:Q_def}. 
    To construct the feasible set of $V$ that yields a self-mapping, we consider a proper policy $\pi_0$ ensured by Assumption~\ref{asm:proper}. Then, by Lemma~\ref{lem:neumann}, we can compute the value $V_{\pi_0} = (\mathbb{I}-P_{\pi_0})^{-1} U_{\pi_0}$. Since $(\mathbb{I}-P_{\pi_0})^{-1}\ge 0$ element-wise (Lemma~\ref{lem:neumann}) and $U_{\pi_0} < \mathbf{0}$ (Proposition~\ref{prop:well-posedness}\ref{prop:universal-strict}), we have $V_{\pi_0}\leq \mathbf{0}$ and it will serve as the lower bound of feasible values given that $T_*V_{\pi_0}\geq V_{\pi_0}$ as per Lemma~\ref{lem:Tstar_properties}(iii).
    As for the upper bound, we simply consider the case $V=\mathbf{0}$. Then, we have $T_*\mathbf{0}(s) = H_s(u(s,\cdot)) < 0$ by Assumption~\ref{asm:proper}. 
    Hence, $T_*$ maps from the compact, convex set $\Omega_V=\{V\in\mathbb{R}^{|S|}|V_{\pi_0}\le V\le \mathbf{0}\}$ to itself. 

    \item Uniqueness: We prove uniqueness of $V_*$ by contradiction. Suppose $V_*$ and $\check V$ are both fixed points of $T_*$. By Lemma~\ref{lem:choice_map}\ref{FY:max}, $\check\pi = \nabla H(Q(\check V))$ is the optimal policy at $\check V$ that yields $\check V = T_{\check \pi} \check V = U_{\check\pi} + P_{\check\pi}\check V$. As per Proposition~\ref{prop:well-posedness}\ref{prop:optimal-proper}, $\check\pi$ is proper, which implies $P_{\check\pi}^N\to 0$ as $N\to\infty$ and $\check V = (\mathbb{I}-P_{\check\pi})^{-1}U_{\check\pi}$ (Lemma~\ref{lem:neumann}). 

    By the definition of $T_*$, we have $V_*=T_*V_*\ge T_{\check\pi}V_*$. Then, Lemma~\ref{lem:Tstar_properties}(i) yields
    \begin{align}
        V_* \ge T_{\check\pi}V_* \ge T^2_{\check\pi}V_*\ge \cdots \ge T^N_{\check\pi}V_* = \sum_{k=0}^{N-1} P^k_{\check\pi} U_{\check\pi} + P_{\check\pi}^N V_*. 
    \end{align}
    As $N\to \infty$, the first term converges to $(\mathbb{I}-P_{\check\pi})^{-1}U_{\check\pi}$ (see the proof of Lemma~\ref{lem:neumann}) and the second term converges to zero given the bounded $V_*\le \mathbf{0}$. Accordingly, we have $V_*\ge (\mathbb{I}-P_{\check\pi})^{-1}U_{\check\pi} = \check V$.

    Following the same reasoning, we may derive the result $V_{\check\pi}\ge V_*$ using operator $T_{\pi_*}$. Therefore, $V_* = \check V$ must hold, and thus the optimal value is unique. 

    \item Dominance:
    The existence argument used a particular proper policy $\pi_0$, but the same construction applies for any proper policy $\pi$. Accordingly, the feasible set $\Omega_V$ can be constructed using the lower envelope of $V_\pi$ over all proper policies. The uniqueness of $V_*$ carries through, which then yields $V_*\ge V_\pi$ for any proper policy $\pi$.

    Finally, since $V_*$ is unique and $\nabla H$ is single-valued (Assumption~\ref{asm:surplus}), the optimal policy $\pi_* = \nabla H(Q(V_*))$ is unique (Lemma~\ref{lem:choice_map}). Its properness follows from Proposition~\ref{prop:well-posedness}\ref{prop:optimal-proper}.
\end{enumerate}

\subsection{Proof of Proposition~\ref{prop:demand}}\label{app:proof_demand}
We prove each result as follows:
\begin{enumerate}[label=\textup{(\roman*)}]
    \item Convexity of $V_*(u)$: Consider a proper policy $\pi$ such that 
    \begin{align}
        & V_\pi(u) = T_\pi V_\pi(u) = \pi\tran u + P_\pi V_{\pi}(u) - H^*(\pi)\\
        \Rightarrow \quad & V_\pi = (\mathbb{I} -  P_\pi)^{-1}(\pi\tran u - H^*(\pi))
    \end{align}
    Hence, $V_\pi(u)$ is affine function of $u$. 
    The convexity of $V_*(u)$ then follows from the supremum of affine functions, using the result $V_*(u)=\sup_{\pi\in\Omega_\pi} V_\pi(u)$ from Proposition~\ref{prop:Vstar}, where $\Omega_\pi$ denotes the set of all proper policies. 

    \item Smoothness of $V_*(u)$, $\phi(u)$ and expression of $x_*(u)$: As per Lemma~\ref{lem:choice_map}\ref{FY:max}, the fixed point of $V_*(u)$ can be rewritten as $V_*(u) = T_* V_*(u) = H(Q(u,V_*(u)))$. 
    Consider $G(V_*, u) = V_* - H(Q(u, V_*))$, which is differentiable in both arguments by construction. We have its partial Jacobians:
    \begin{align}
        \frac{\partial G}{\partial V_*} &= \mathbb{I} - \nabla H(Q(u,V_*(u))) P = \mathbb{I}-P_{\pi_*},\label{eq:proof_demand_partial_1}\\
        \frac{\partial G}{\partial u} &= - \nabla H(Q(u,V_*(u))) = -\pi_*(u)
    \end{align}
    Here $\pi_*(u)=\nabla H(Q(u,V_*(u)))$ is identified via Proposition~\ref{prop:well-posedness}\ref{prop:optimal-proper}, with a minor abuse of dimensionality: $\pi_*(u)$ is embedded in $\R^{|\mathcal{S}|\times |\mathcal{S}||\mathcal{A}|}$ with each row nonzero only on the corresponding state-action pairs, and the dependence of $P_{\pi_*}$ on $u$ is suppressed for brevity.

    Since $\pi_*$ is proper (Proposition~\ref{prop:well-posedness}\ref{prop:optimal-proper}), $(\mathbb{I}-P_{\pi_*})$ is invertible (Lemma~\ref{lem:neumann}). Then, by the implicit function theorem, $V_*(u)$ is $C^1$ with Jacobian $\nabla V_*(u) = (\mathbb{I}-P_{\pi_*})^{-1}\pi_*(u)$, yielding $x_*(u) = [(\mathbb{I}-P_{\pi_*})^{-1}\pi_*(u)]\tran q$ as in Eq.~\eqref{eq:link_demand}.

    With the definition of $\phi(u)$, we have $x_*(u) = \nabla \phi(u)$. 
    The potential $\phi$ inherits convexity and $C^1$ regularity from $V_*(u)$, so $x_*(u)$ is continuous and monotone.

    \item Additional properties under Assumption~\ref{asm:surplus_extra}: Finally, we prove the properties of $\phi(u)$ and $\nabla x_*(u)$ under the additional Assumption~\ref{asm:surplus_extra}. Since $H_s$ is $C^2$, $G(u, V_*)$ is also $C^2$. Then, by the implicit function theorem, $V_*(u)$ is $C^2$ and thus $\phi(u)$ is $C^2$. With the convexity of $\phi$, we have $\nabla x_*(u) = \nabla^2 \phi(u)$ exists and is positive semi-definite. 
\end{enumerate}

\section{Deferred proofs from Section~\ref{sec:equilibrium}}

\subsection{Proof of Proposition~\ref{prop:monotone_demand}}\label{app:proof_monotone_demand}

Applying Proposition~\ref{prop:demand}, we first derive the class-specific potential in the cost space as 
\begin{align}
    \phi_k(c) := q_k\tran V_{k,*}(u_k(c)). 
\end{align}
Under Assumption~\ref{asm:utility}, $u_k(c)$ is a linear map. Hence, $\phi_k$ is convex and $C^1$ in $c$ as it is convex and $C^1$ in $u$. Further, its gradient is evaluated as 
\begin{align}
    \nabla \phi_k(c) = - B_k\tran x_{k,*}(u_k(c)). 
\end{align}

Now define the global potential $\Phi(c):= \sum_{k\in\mathcal{K}} \phi_k(c)$. Then, it is easy to show $\Phi$ is convex and $C^1$ in $c$ with
\begin{align}
    \nabla \Phi(c) = - \sum_{k\in\mathcal{K}} B_k\tran x_{k,*}(u_k(c)) = -x(c). 
\end{align}
The convexity of $\Phi$ also implies 
\begin{align}
    \langle \nabla \Phi(c) - \nabla \Phi(c'), c - c'\rangle \ge 0, \quad \forall c,c'\in \Omega,
\end{align}
which yields Eq.~\eqref{eq:cost_monotone}. 

Finally, under Assumption~\ref{asm:surplus_extra}, $\phi_k(u)$ is $C^2$ in $u$ as per Proposition~\ref{prop:demand}. Following the same reasoning, we have $\Phi$ is convex and $C^2$ in $c$, which further yields
$\nabla x(c) = - \nabla^2 \Phi(c) \preceq 0$.

\subsection{Proof of Theorem~\ref{thm:existence}}\label{app:proof_exist}

Consider the anchor point $\hat{c}$ introduced in Assumption~~\ref{asm:supply}\ref{asm:supply_coercive}. As per Proposition~\ref{prop:monotone_demand}, the monotone link demand satisfies
\begin{align}
    \langle x(c) - x(\hat{c}),\, c - \hat{c} \rangle \le 0,\quad \Rightarrow \quad \langle x(c),\, c - \hat{c} \rangle \le \langle x(\hat{c}),\, c - \hat{c} \rangle, \quad \forall c\in \Omega. 
\end{align}
Accordingly, for $c\neq \hat{c}$, we have
\begin{align}
    \frac{\langle E(c),\, c - \hat{c} \rangle}{\|c - \hat{c}\|}
    &= \frac{\langle z(c),\, c - \hat{c} \rangle}{\|c - \hat{c}\|}
    - \frac{\langle x(c),\, c - \hat{c} \rangle}{\|c - \hat{c}\|}\\
    &\ge \frac{\langle z(c),\, c - \hat{c} \rangle}{\|c - \hat{c}\|}
    - \frac{\langle x(\hat{c}),\, c - \hat{c} \rangle}{\|c - \hat{c}\|}
    \ge  \frac{\langle z(c),\, c - \hat{c} \rangle}{\|c - \hat{c}\|} - \|x(\hat{c})\|,\nonumber
\end{align}
where the last inequality follows Cauchy-Schwarz inequality. Due to Eq.~\eqref{eq:coercive} in Assumption~\ref{asm:supply}\ref{asm:supply_coercive}, the first term approaches to $+\infty$ as $||c||\to \infty$ while the second term remains constant. Therefore, $E$ is coercive because $\frac{\langle E(c),\, c - \hat{c} \rangle}{\|c - \hat{c}\|}\to+\infty$ as $||c||\to \infty$. The VI solution existence is then proved along with the continuity of $E$ and convex non-empty $\Omega$~\citep[see, e.g.,][Corollary~3]{Nagurney2001}. 

\subsection{Proof of Theorem~\ref{thm:uniqueness}}\label{app:proof_unique}

% \kz{
Since the supply function $z$ is strictly monotone, we have for any $c_1\neq c_2$
\begin{align}
    \langle z(c_1) - z(c_2),\, c_1 - c_2 \rangle > 0.
\end{align}
Combining this result with the monotone demand function $x$ from Proposition~\ref{prop:monotone_demand}, we have
\begin{align}
    \langle E(c_1) - E(c_2),\, c_1 - c_2 \rangle
    = \langle z(c_1) - z(c_2),\, c_1 - c_2 \rangle
    \;-\; \langle x(c_1) - x(c_2),\, c_1 - c_2 \rangle
    \;>\; 0.
\end{align}
Therefore, $E$ is strictly monotone, which implies a unique VI solution following the standard results~\citep[see, e.g.,][Theorem~12]{Nagurney2001}. 
% }
% \begin{proof}
% For any $c_1 \neq c_2$,
% \[
%     \langle E(c_1) - E(c_2),\, c_1 - c_2 \rangle
%     = \langle z(c_1) - z(c_2),\, c_1 - c_2 \rangle
%     \;-\; \langle x(c_1) - x(c_2),\, c_1 - c_2 \rangle
%     \;>\; 0.
% \]
% Hence $E$ is strictly monotone, and the VI solution is unique \citep[see, e.g.,][Theorem~12]{Nagurney2001}.
% \end{proof}
\subsection{Proof of Lemma~\ref{lem:PURC}}\label{app:proof_PURC}
For given $c\in\Omega$, we denote $x_0 = x(c)= -\nabla\Phi(c)\in\mathcal{X}$ as per Proposition~\ref{prop:monotone_demand}. The convexity of $\Phi$ implies that
\begin{align}
    & \Phi(c')\geq \Phi(c) + \langle \nabla\Phi(c), c'-c\rangle = \Phi(c) + \langle -x_0,\, c' - c \rangle,\quad \forall c'\in \Omega \\
    \Rightarrow \quad & -c\tran x_0 - \Phi(c) \geq -{c'}\tran x_0 - \Phi(c'), \quad \forall c'\in \Omega.
\end{align}
Therefore, we have $R(x_0) = \sup_{c'\in\Omega} \left\{-{c'}\tran x_0 - \Phi(c') \right\} =  -c\tran x_0 - \Phi(c)$ and 
\begin{align}
    R(x) - R(x_0)  &= \sup_{c'\in\Omega} \left\{-{c'}\tran x - \Phi(c') \right\} - \left(-c\tran x_0 - \Phi(c)\right) \\
    &\ge - c\tran x - \Phi(c) + c\tran x_0 + \Phi(c) = \langle -c,\, x - x_0 \rangle,\quad \forall x\in\mathcal{X}. \nonumber
\end{align}
This condition is the subgradient inequality for $-c \in \partial R(x_0) \Leftrightarrow 0 \in c + \partial R(x_0)$. Together with the condition that $x(c) \in \mathcal{X}$ is always feasible, it implies that $x_0$ satisfies the first-order optimality condition of perturbed best response problem \eqref{eq:PURC}.
% Hence, $x(c) \in \mathcal{X}$ solves
% perturbed best response problem \eqref{eq:PURC}. 
% Since $x_0 = x(c)\in\mathcal{X}$ (Eq.~\eqref{eq:feasible_set}), $x_0$ also minimizes it over $\mathcal{X}$, i.e., $x_0$ solves
% perturbed best response problem defined in \eqref{eq:PURC}. 

% This inequality implies that $-c$ is a subgradient of $R(x)$ at $x_0$, i.e., $-c\in\partial R(x_0)$ or equivalently, $0\in c+\partial R(x_0)$. 

% \begin{proof}
% Write $x_0 = x(c)$.
% Since $\Phi$ is convex with $\nabla\Phi(c) = -x_0$ (Theorem~\ref{prop:monotone_demand}), the first-order characterization of convexity gives $\Phi(c') \ge \Phi(c) + \langle -x_0,\, c' - c \rangle$ for all $c' \in \Omega$, that is, $-{c'}\tran x_0 - \Phi(c') \le -c\tran x_0 - \Phi(c), \forall c' \in \Omega$.
% Hence, the supremum in~\eqref{eq:Fenchel_R} is attained at $c' = c$ (not necessarily unique) and we have $R(x_0) = -c\tran x_0 - \Phi(c)$ by definition of $R$.
% For any $x$, evaluating the RHS of Eq.~\eqref{eq:Fenchel_R} at $c' = c$ gives $R(x) \ge -c\tran x - \Phi(c)$, so
% \[
%     R(x) - R(x_0) \;\ge\; \langle -c,\, x - x_0 \rangle, \qquad \forall\, x,
% \]
% which is the subgradient inequality and $-c$ is a subgradient, i.e., $-c \in \partial R(x_0)$, equivalently $0 \in c + \partial R(x_0)$.
% This is the first-order optimality condition for the unconstrained problem $\min_x\{c\tran x + R(x)\}$; since $x_0 \in \mathcal{X}$, it also minimizes over $\mathcal{X}$, i.e., $x_0 \in \arg \min_{x \in \mathcal{X}}\{c\tran x + R(x)\}$.
% \end{proof}

\subsection{Proof of Theorem~\ref{thm:VNE}}\label{app:proof_VNE}

Consider an interior PUME $c^*\in \interior(\Omega)$. The dual VI condition reduces to the market clearance $z(c^*) = x(c^*)$, which yields $c^* = z^{-1}(x(c^*)) = z^{-1}(x^*)$ as per Assumption~\ref{asm:supply_inverse}. 
By Lemma~\ref{lem:PURC}, we have $0 \in c^* + \partial R(x(c^*))  =  z^{-1}(x^*) + \partial R(x^*)$. In other words, there exists $\rho^* \in \partial R(x^*)$ such that $z^{-1}(x^*) + \rho^* = 0$. Combining with $x^* \in \mathcal{X}$ as per Eq.~\ref{eq:feasible_set}, $x^*$ is a solution to VI problem \eqref{eq:VNE}.

% \begin{proof}
% Under the boundary condition~\eqref{eq:additional_interior_cond}, Theorem~\ref{thm:existence} guarantees that the PUME $c^*$ lies in the interior of $\Omega$, so the VI reduces to market clearance: $z(c^*) = x(c^*)$. Hence, $c^* = t(x^*)$ by invertibility of $z$ (Assumption~\ref{asm:supply_inverse}).
% By Lemma~\ref{lem:PURC}, $0 \in c^* + \partial R(x(c^*)) = t(x^*) + \partial R(x^*)$.
% Hence there exists $\rho^* \in \partial R(x^*)$ with $t(x^*) + \rho^* = 0$, and~\eqref{eq:VNE} holds for all $x \in \mathcal{X}$.
% \end{proof}

\section{Deferred proofs from Section~\ref{sec:algorithms}}\label{app:deferred-sec5}

% This appendix collects the deferred arguments from Section~\ref{sec:algorithms}.
% We first record the auxiliary MPI lemmas and proofs used in the inner solver analysis, then prove the convergence theorem for the safeguarded meta-algorithm, and finally verify condition \textup{(B1)} for the implemented base methods.

% \subsection{MPI auxiliary lemmas and proofs}
\subsection{Proof of Theorem~\ref{thm:mpi_global}}\label{app:proof_mpi_global}

\begin{lemma}[Iterative improvement]\label{lem:subsolution}
Suppose the value generated by MPI with $m\ge 1$ at iteration $n$ satisfies $V_n\leq T_*V_n$, then $V_n\le T_* V_n \leq V_{n+1}\le T_*V_{n+1}$. 
% let $m \ge 1$ and $\{V_n\}$ be generated by~\eqref{eq:mpi_update}.
% If $V_n \le T_* V_n$, then $V_n \le V_{n+1}$ and $V_{n+1} \le T_* V_{n+1}$.
\end{lemma}
\begin{proof}
Let $\pi_n = \nabla H(Q(V_n))$ be the greedy policy at iteration $n$. Then, by Lemma~\ref{lem:choice_map}, we have $T_{\pi_n} V_n = T_* V_n$. The condition $V_n \le T_* V_n$ then gives $V_n \le T_{\pi_n} V_n$.
Lemma~\ref{lem:Tstar_properties}(ii) implies $T_{\pi_n} V_n \leq T_{\pi_n}(T_*V_n) = T^2_{\pi_n} V_n$ and inductively leads to $T_{\pi_n} V_n\le T^m_{\pi_n}V_n$. Altogether, we have $V_n \le T_* V_n = T_{\pi_n} V_n\le T^m_{\pi_n} V_n = V_{n+1} \le T_{\pi_n} V_{n+1} \leq T_*V_{n+1}$, where the last inequality is due to Lemma~\ref{lem:choice_map}(i).
% The monotonicity of $T_{\pi_n}$ (Lemma~\ref{lem:Tstar_properties}) further implies $T_* V_n = T_{\pi_n} V_n\le T^m_{\pi_n} V_n = V_{n+1}$. It thus yields $V_n\le T_* V_n \leq  V_{n+1}$. The last inequality is due to the fact that $V_{n+1}\leq T_{\pi_n} V_{n+1}\leq T_*V_{n+1}$. 

% Let $\pi_n = \nabla H(Q(V_n))$ so that $T_* V_n = T_{\pi_n} V_n$ by the optimal identity (Lemma~\ref{lem:fenchel-young}).
% The condition $V_n \le T_* V_n$ then gives $V_n \le T_{\pi_n} V_n$.
% Hence, $V_n = W^{(0)} \le T_{\pi_n} V_n = W^{(1)}$. Together with monotonicity of $T_{\pi_n}$ (Lemma~\ref{lem:Tstar_properties}), a standard induction then implies that the sequence $W^{(h)} = T_{\pi_n}^h V_n$ is nondecreasing in $h$.
% In particular, $V_{n+1} = T_{\pi_n}^m V_n = W^{(m)} \ge W^{(0)} = V_n$.
% Moreover, $T_{\pi_n} V_{n+1} = T_{\pi_n}^{m+1} V_n = W^{(m+1)} \ge W^{(m)} = V_{n+1}$, hence
% \[
%     T_* V_{n+1} \ge T_{\pi_n} V_{n+1} \ge V_{n+1},
% \]
% where the first inequality holds since $T_*$ maximizes over all policies. 
\end{proof}

We prove each result in Theorem~\ref{thm:mpi_global} as follows:
\begin{enumerate}[label=\textup{(\roman*)}]
    \item We show this by induction. 
    By Lemma~\ref{lem:neumann} and Proposition~\ref{prop:well-posedness}\ref{prop:universal-strict}, the initial value of at some proper policy $\pi_0$ satisfies $V_0 \le 0$, .
    
    % Due to Lemma~\ref{lem:neumann} and Proposition~\ref{prop:well-posedness}\ref{prop:universal-strict}, the initial value $V_0 \le 0$.     
    Suppose $V_n\leq 0$, the facts that $U_{\pi_n}< 0$ (Proposition~\ref{prop:well-posedness}\ref{prop:universal-strict}) and $P_{\pi_n} \geq 0$ yields
    \begin{align}
        T_{\pi_n}V_{n} = U_{\pi_n} + P_{\pi_n}V_n \leq 0,
    \end{align}
    which implies $V_{n+1} = T^m_{\pi_n}V_{n}\leq 0$. By induction, we prove $V_n\leq 0$ for all iterates $n$. 

    As per Lemma~\ref{lem:Tstar_properties}(iii), we have $V_0 \le T_*V_0$. Combining this result with Lemma~\ref{lem:subsolution} yields the monotone value iterates $V_0\leq \dots \leq V_n\leq V_{n+1}\leq 0$.

    % Moreover, $V_0 = V_{\pi_0}$ for a proper policy $\pi_0$, Lemma~\ref{lem:Tstar_properties}(iii) gives $V_0 \le T_*V_0$ that seeds the requirement of Lemma~\ref{lem:subsolution}. Combined this result with Lemma~\ref{lem:subsolution}, we have $V_0\leq \dots \leq V_n\leq V_{n+1}\leq 0$. 
    
    The upper bound $V_*$ is also proved by induction. By construction,  we have $V_0 \le V_*$. Suppose $V_n\leq V_*$, then by Lemma~\ref{lem:Tstar_properties}(i), we have $T_{\pi_n}V_n \leq T_{\pi_n}V_*$, which further yields $V_{n+1}= T^m_{\pi_n}V_n \leq T^m_{\pi_n}V_*$. 
    Since $T_{\pi_n}V_*\leq T_*V_*=V_*$ (the former by the optimality of $T_*$ and the latter by Proposition~\ref{prop:Vstar}), Lemma~\ref{lem:Tstar_properties}(i) implies $T^2_{\pi_n}V_*\leq T_{\pi_n}V_* \leq V_*$. Inductively, we derive $V_{n+1}\leq T^m_{\pi_n}V_*\leq V_*$.
    
    % Since $T^m_{\pi_n}V_* \leq T_*V_* = V_*$ (the first inequality due to the optimality and monotonicity of $T_*$ and the second equality due to Proposition~\ref{prop:Vstar}), we have $V_{n+1}\leq V_*$.

    \item The monotone and bounded value iterates derived in (i) ensure convergence element-wise. Let $\bar{V} = \lim_{n \to \infty} V_n$, then the remaining task is to prove $\bar{V} = V_*$. 
    Since $V_0 \leq T_* V_0$,  Lemma~\ref{lem:subsolution} implies $V_n \le T_*V_n \le V_{n+1}$ for any $n$. Since $T_*$ is continuous and $V_n \to \bar V$, we have $\bar V \le T_*\bar V \le \bar V$ as $n\to\infty$ and thus $\bar{V} = T_* \bar{V}$.
    % letting $n\to\infty$ yields $\bar V \le T_*\bar V \le \bar V$, i.e., $\bar{V} = T_* \bar{V}$.
    % Since the mapping $V \mapsto \pi(V) \coloneq \nabla H(Q(V))$ is continuous, and the operator $T_\pi^m$ is continuous in $(\pi, V)$, we have 
    % \begin{align}
    %     \bar{V} = \lim_{n \to \infty} V_n = \lim_{n \to \infty} V_{n+1} = \lim_{n \to \infty} T_{\pi_n}^m V_n = T^m_{\pi(\bar{V})} \bar{V} =  T_{\pi(\bar{V})} \bar{V},
    % \end{align}
    % where the last equality is due to Lemma~\ref{lem:fixedpoints_Tpim}. Since $\pi(\bar{V})$ is the optimal policy at value $\bar{V}$, it satisfies that $T_{\pi(\bar{V})} = T_* \bar{V}$ (Lemma~\ref{lem:choice_map}). Therefore, the above equation reduces to $\bar{V} = T_* \bar{V}$, i.e., $\bar{V}$ is the fixed point of $T_*$. 
    Then, by the uniqueness of the optimal value (Proposition~\ref{prop:Vstar}), it concludes that $\bar{V}=V_*$. 
\end{enumerate}

\subsection{Proof of Theorem~\ref{thm:mpi_local}}\label{app:proof_mpi_local}

\begin{lemma}[Differentiable MPI map]\label{lem:mpi_jacobian}
The MPI map $O_m$ is differentiable at $V_*$ and $\nabla O_m(V_*) = P_{\pi_*}^m$.
\end{lemma}
\begin{proof}
Let $\pi$ be the greedy policy constructed at $V$. 
Since $T_\pi V = U_\pi + P_\pi V$, we have $T_\pi^m V = (\sum_{j=0}^{m-1}P_\pi^j)U_\pi + P_\pi^m V$ by unrolling the expression.
The MPI map is expanded as
\begin{align}
    O_m(V) = T^m_\pi V &= \left(\sum_{j=0}^{m-1} P_\pi^j\right) U_\pi +  P_\pi^m V = \left(\sum_{j=0}^{m-1} P_\pi^j\right) (T_\pi V - P_\pi V) + P_\pi^m V\\
    % &= \left(\sum_{j=0}^{m-1} P_\pi^j\right) (T_\pi V - V) + \left(\sum_{j=0}^{m-1} P_\pi^j\right) (V - P_\pi V) + P_\pi^m V\nonumber\\
    &= \left(\sum_{j=0}^{m-1} P_\pi^j\right) (T_\pi V - V) + \sum_{j=0}^{m} P_\pi^j V - \sum_{j=1}^{m} P_\pi^j V \nonumber\\
    &= V - S_m(V) (V - T_\pi V), \nonumber
\end{align}
where $S_m(V) = \sum_{j=0}^{m-1} P_{\pi}^j$.

Define the gap function as $G(V) = V - T_* V = V - H(Q(V))$. Since $\pi$ is the greedy policy associated with $V$, we have $G(V) = V - T_\pi V$. Besides, the fixed-point condition of optimal value implies $G(V_*) = 0$ (Proposition~\ref{prop:Vstar}). Due to Standing Assumption~\ref{asm:surplus}, the gap function $G$ is $C^1$ and its first-order Taylor expansion at $V_*$ is written as
\begin{align}\label{eq:proof_lem_mpi_jacobian_gap}
    G(V) &= G(V_*) + \nabla G(V_*)(V-V_*) + o(||V-V_*||)= (\mathbb{I} - P_{\pi_*}) (V-V_*) + o(||V-V_*||),
\end{align}
where $\nabla G$ follows the derivation in Eq.~\ref{eq:proof_demand_partial_1}. 

Plugging the above into the convergence gap yields
\begin{align}\label{eq:proof_lem_mpi_jacobian_conv_gap}
    O_m(V) - V_* &= (V-V_*) - S_m(V)G(V) \\
    &= (V-V_*) - S_m(V_*)G(V) - (S_m(V)- S_m(V_*))G(V)\nonumber\\
    &= (V-V_*) - S_m(V_*)\left[(\mathbb{I} - P_{\pi_*})(V-V_*) + o(||V-V_*||)\right]  - (S_m(V)- S_m(V_*))G(V)\nonumber\\
    &= \left[\mathbb{I} - S_m(V_*)(\mathbb{I} - P_{\pi_*})\right](V-V_*) -  (S_m(V)- S_m(V_*))G(V) - S_m(V_*) o(||V-V_*||).\nonumber
\end{align}

% Each term in Eq.~\eqref{eq:proof_lem_mpi_jacobian_conv_gap} is further simplified as follows:

By expanding $S_m(V_*) = \sum_{j=0}^{m-1}P^j_{\pi_*}$, the first factor in Eq.~\eqref{eq:proof_lem_mpi_jacobian_conv_gap} is reduced to 
\begin{align}
    \mathbb{I} - S_m(V_*)(\mathbb{I} - P_{\pi_*}) = \mathbb{I} - (\mathbb{I} - P^m_{\pi_*}) = P^m_{\pi_*}.
\end{align}

The second term is bounded by $\|(S_m(V)- S_m(V_*))G(V)\|\leq \|S_m(V)- S_m(V_*)\|\|G(V)\|$.
% \begin{align}\label{eq:proof_lem_mpi_jacobian_conv_gap_2}
%     \frac{\|(S_m(V)- S_m(V_*))G(V)\|}{\|V-V_*\|}
%     &\leq \|S_m(V)- S_m(V_*)\|\frac{\|G(V)\|}{\|V-V_*\|}.
% \end{align}
Since $P_{\pi}$ is continuous in $V$, $S_m(V)$ is continuous at $V_*$ and therefore
$\|S_m(V)-S_m(V_*)\|\to0$ as $V\to V_*$. Since $G$ is differentiable at $V_*$ and $G(V_*)=0$,
$\|G(V)\|/\|V-V_*\|$ remains bounded in a neighborhood of $V_*$. Hence,
$(S_m(V)-S_m(V_*))G(V)=o(\|V-V_*\|)$.
    
    % We plug Eq.~\eqref{eq:proof_lem_mpi_jacobian_gap} into the second term and apply the triangle inequality:
    % \begin{align}\label{eq:proof_lem_mpi_jacobian_conv_gap_2}
    %     ||(S_m(V)- S_m(V_*))G(V)||&= ||(S_m(V)- S_m(V_*))[(\mathbb{I} - P_{\pi_*}) (V-V_*) + o(||V-V_*||)]|| \nonumber\\
    %     &\leq ||(S_m(V)- S_m(V_*))(\mathbb{I} - P_{\pi_*})|||| (V-V_*)|| \nonumber \\
    %     &\quad + ||S_m(V)- S_m(V_*)|| o(||V-V_*||).
    % \end{align}
    % The second term in Eq.~\eqref{eq:proof_lem_mpi_jacobian_conv_gap_2} directly reduces to $o(||V-V_*||)$, since $||S_m(V)- S_m(V_*)||$ is bounded in a neighborhood of $V_*$ by continuity of $S_m$ at $V_*$.
    
    % The factor in the first term is further expanded as
    % \begin{align}
    %     &S_m(V)(\mathbb{I} - P_{\pi}) - \mathbb{I} + \mathbb{I} -  S_m(V_*)(\mathbb{I} - P_{\pi_*}) + S_m(V)(P_\pi - P_{\pi_*})\\
    %     =& (P_\pi - P_{\pi_*})(S_m(V) - \mathbb{I}) = (P_\pi - P_{\pi_*})\sum_{j=1}^{m-1} P^j_\pi\nonumber
    % \end{align}
    % Then, to prove the first term is also $o(||V-V_*||)$, it suffices to show that 
    % \begin{align}
    %     \lim_{V\to V_*} (P_\pi - P_{\pi_*})\sum_{j=1}^{m-1} P^j_\pi = 0,
    % \end{align}
    % which is indeed the case because $P_\pi$ is bounded and continuous in $V$ due to the continuous mapping $\pi(V)$ and Eq.~\eqref{eq:policy_transition}.

Lastly, the fact that $S_m(V_*)$ is constant yields $S_m(V_*) o(||V-V_*||) = o(||V-V_*||)$.
Altogether, 
\begin{align}
    &O_m(V) - V_* = P^m_{\pi_*}(V-V_*) + o(||V-V_*||)\\
    \Rightarrow \quad & \nabla O_m(V_*) = \lim_{V\to V_*} \frac{O_m(V) - O_m(V_*)}{V-V_*} = \lim_{V\to V_*} \frac{O_m(V) - V_*}{V-V_*} = P^m_{\pi_*}.\nonumber
\end{align}
\end{proof}

Below, we prove Theorem~\ref{thm:mpi_local}. Since the optimal policy $\pi_*$ is proper (Proposition~\ref{prop:well-posedness}~\ref{prop:optimal-proper}), its spectral radius satisfies $\rho(P_{\pi_*}) < 1$. Hence, there exists a matrix norm with $||P_{\pi_*}||<1$\citep[Lemma~5.6.10]{horn2012matrix}, and further $||P^m_{\pi_*}||\le ||P_{\pi_*}||^m < 1$.

Plugging the result of Lemma~\ref{lem:mpi_jacobian}, the convergence gap can be rewritten as
\begin{align}
    O_m(V) - V_* = O_m(V) - O_m(V_*) &= \nabla O_m(V_*)(V-V_*) + o(||V-V_*||)\\
    &= P^m_{\pi_*}(V-V_*) + o(||V-V_*||)\nonumber
\end{align}

Taking norms on both sides and applying the triangle inequality yields
\begin{align}\label{eq:proof_mpi_local_gap_norm}
    || O_m(V) - V_*|| &\leq ||P^m_{\pi_*}(V-V_*)|| + o(||V-V_*||)\\
    &\leq (||P^m_{\pi_*}||+\varepsilon)||V-V_*||.\nonumber
\end{align}
For any $\varepsilon \in (0, 1 - ||P^m_{\pi_*}||)$, there exists neighborhood $\mathcal{N}(V_*)$ of radius $r(\varepsilon)$ such that the second inequality holds for all $V\in \mathcal{N}(V_*)$, i.e., $||V-V_*||<r(\varepsilon)$. This concludes the proof of (a) with contraction factor $\nu = ||P^m_{\pi_*}||+\varepsilon < 1$. 

By Theorem~\ref{thm:mpi_global}, the iterates $V_n\to V_*$. Then, there exists $N$ such that $V_n\in \mathcal{N}(V_*)$ for all $n\geq N$. 
Replacing $V$ and $O_m(V)$ with $V_n$ and $V_{n+1}$, respectively, 
Eq.~\ref{eq:proof_mpi_local_gap_norm} yields
\begin{align}
    ||V_{n+1} - V_*|| \le \nu ||V_n - V_*||, \quad \forall n\geq N. 
\end{align}
Since the result holds for any $\varepsilon > 0$, it implies that 
\begin{align}
    \limsup_{n\to\infty}\frac{||V_{n+1} - V_*||}{||V_n - V_*||} \leq ||P^m_{\pi_*}||. 
\end{align}

\subsection{Proof of Theorem~\ref{thm:meta}}\label{app:meta_convergence}
% \subsection{Meta-algorithm auxiliary lemmas and proofs}\label{app:meta_convergence}

We first outline the key properties of the merit function defined in~\eqref{eq:merit_def} in the following lemma, which will be applied in the proof of meta-algorithm convergence afterwards. 

% We prove that the merit function $\zeta(c) = \langle E(c),\, r_{\mathrm{nat}}(c)\rangle$ defined in~\eqref{eq:merit_def} satisfies
% three properties used throughout the convergence analysis.

% \kz{add result $\zeta(c) \ge \|r(c)\|^2 \ge 0$}

\begin{lemma}[Merit function properties]\label{lem:merit}
% Under the conditions of Theorems~\ref{thm:uniqueness}, define $r_{\mathrm{nat}}(c) = c - \mathrm{Proj}_\Omega(c - E(c))$ and $\zeta(c) = \langle E(c),\, r_{\mathrm{nat}}(c)\rangle$. Then:
Under the assumptions of Theorem~\ref{thm:uniqueness}, the merit function \eqref{eq:merit_def} satisfies the following:
\begin{enumerate}[nosep,label=\textup{(M\arabic*)}]
    \item $\zeta(c) \geq \|r(c)\|^2 \geq 0$.
    \item $\zeta(c) = 0 \;\Leftrightarrow\; c = c^*$, where $c^*$ is the unique PUME.
    \item $\zeta$ is continuous on $\Omega$.
    \item The sublevel set $\mathcal{S}_\alpha =\{c \in \Omega : \zeta(c) \le \alpha\}$ are bounded for any finite $\alpha \ge 0$.
\end{enumerate}
\end{lemma}

\begin{proof}
For notation simplicity, we denote $p(c) \coloneqq \mathrm{Proj}_\Omega(c - E(c))$ and thus $r(c)= c-p(c)$. In what follows, we prove the properties sequentially:
\begin{enumerate}[label=\textup{(M\arabic*)}]
    \item Since $\Omega$ is nonempty, closed, and convex, by the variational characterization of convex projection, we have 
    \begin{equation}\label{eq:proj_var}
        \langle (c - E(c)) - p(c),\; c' - p(c) \rangle \le 0 \qquad \forall\, c' \in \Omega.
    \end{equation}
    Eq.~\eqref{eq:proj_var} holds for $c'=c$, which yields
    \begin{align}
        \langle (c - E(c)) - p(c),\; c - p(c) \rangle \le 0\;\Rightarrow\;&
        \langle r(c) - E(c),\; r(c) \rangle \le 0\\
        \Rightarrow\;&
        0\leq \|r(c)\|^2 \le \langle E(c),\, r(c) \rangle = \zeta(c).\nonumber
    \end{align}

    \item ($\Rightarrow$): If $\zeta(c) = \langle E(c), r(c)\rangle = 0$, then by (M1), $\|r(c)\|^2\leq \zeta(c)= 0$ and thus $r(c)=0$. Accordingly, we have $c = p(c) = \mathrm{Proj}_\Omega(c - E(c))$ that aligns with the VI condition~\eqref{eq:VI_PUME} By the uniqueness of VI solution (Theorem~\ref{thm:uniqueness}), we have $c = c^*$. 

    ($\Leftarrow$): Since the PUME $c^*$ is the solution to the VI problem~\eqref{eq:VI_PUME}, it satisfies $c^* = \mathrm{Proj}_\Omega(c^* - E(c^*))=p(c^*)$. Hence, we have $r(c^*)=0$ that yields $\zeta(c^*)=0$. 

    \item Since the supply $z(c)$ and demand $x(c)$ are both continuous in $c$ (Assumption~\ref{asm:supply} and Proposition~\ref{prop:monotone_demand}), the excess supply $E(c) = z(c) - x(c)$ is also continuous. So as the $c-E(c)$. Since the projection on $\Omega$ is non-expansive, it is continuous. Therefore, $r$ is continuous and the inner product $\zeta$ of two continuous mappings is also continuous. 

    \item We prove this property by contradiction. Suppose there exists some $\alpha \geq 0$ such that $\mathcal{S}_\alpha$ is not bounded, i.e., there exist some $c^\infty$ such that $||c^\infty||\to\infty$ while $\zeta(c^\infty)\leq \alpha$. By (M1), we have $||r(c^\infty)||^2\leq \zeta(c^\infty)\leq \alpha$ and thus $||r(c^\infty)||\leq \sqrt{\alpha}$. 

    Let $\hat{c}\in \Omega$ be the anchor point of the coercivity condition~\eqref{eq:coercive}. 
    Since $\|c^\infty\|\to\infty$, the distance $\|c^\infty-\hat{c}\|\to\infty$.
    % Meanwhile, one can always find a bounded cost $\hat{c}\in \Omega$ such that the distance $||c^\infty-\hat{c}||\to\infty$. Let $\hat{c}$ be the anchor point in Eq.~\eqref{eq:coercive} (coercive supply function) and rewrite it with respect to $c^\infty$:
    % \begin{align}
    %     \frac{\langle z(c^\infty), c^\infty- \hat{c}\rangle}{||c^\infty - \hat{c}||} \to \infty
    % \end{align}
    Plugging $c=c^\infty$ and $c'=\hat{c}$ into Eq.~\eqref{eq:proj_var}, we have 
    \begin{align}\label{eq:proj_var_M4}
        &\langle (c^\infty - E(c^\infty)) - p(c^\infty),\; \hat{c} - p(c^\infty) \rangle \le 0\\
        \Rightarrow\; & \langle r(c^\infty) - E(c^\infty), \hat{c} + r(c^\infty) - c^\infty\rangle \leq 0 \nonumber\\
        \Rightarrow\; & ||r(c^\infty)||^2 - \langle r(c^\infty), c^\infty - \hat{c}\rangle + \langle E(c^\infty), c^\infty - \hat{c}\rangle - \langle E(c^\infty), r(c^\infty) \rangle \leq 0\nonumber\\
        \Rightarrow\; & \langle E(c^\infty), c^\infty - \hat{c}\rangle \leq \zeta(c^\infty) + \langle r(c^\infty), c^\infty - \hat{c}\rangle - ||r(c^\infty)||^2.\nonumber
    \end{align}

    By dropping the non-negative term $||r(c^\infty)||^2$ and applying the Cauchy-Schwarz inequality, Eq.~\eqref{eq:proj_var_M4} is reduced to 
    \begin{align}
        &\langle E(c^\infty), c^\infty - \hat{c}\rangle \leq \zeta(c^\infty) + ||r(c^\infty)|| ||c^\infty - \hat{c}||\\
        \Rightarrow\; & \frac{\langle E(c^\infty), c^\infty - \hat{c}\rangle}{||c^\infty - \hat{c}||}\leq \frac{\zeta(c^\infty)}{||c^\infty - \hat{c}||} + ||r(c^\infty)|| \leq \frac{\alpha}{||c^\infty - \hat{c}||} + \sqrt{\alpha} \to \sqrt{\alpha}. \nonumber
    \end{align}
    By the coercivity of the excess supply $E$ established in the proof of Theorem~\ref{thm:existence} (Appendix~\ref{app:proof_exist}), 
    \begin{align}
        \frac{\langle E(c^\infty),\, c^\infty - \hat{c}\rangle}{\|c^\infty - \hat{c}\|} \to \infty
    \end{align}
 as $\|c^\infty\|\to\infty$. This contradicts the upper bound $\sqrt{\alpha}$ derived above, so no such unbounded $c^\infty$ exists and every sublevel set $\mathcal{S}_\alpha$ is bounded.
\end{enumerate}
\end{proof}

Following (M4) in Lemma~\ref{lem:merit}, we define the sublevel radius as follow:
% It follows immediately from the properties of the merit function, we can define the sublevel radius as follow:
\begin{definition}[Sublevel radius]
For $\alpha \ge 0$, the sublevel radius of a merit function $\zeta$ is defined as 
\begin{align}
    \mathcal{R}(\alpha) \coloneqq \sup\{\|c - c^*\|_M: c \in \Omega,\; \zeta(c) \le \alpha\}.
\end{align}
% with the convention $\mathcal{R}(0) = 0$ (since $\zeta(c) = 0 \Leftrightarrow c = c^*$).
\end{definition}
As per Lemma~\ref{lem:merit}(M4), the sublevel set $\mathcal{S}_\alpha$ is bounded for any $\alpha\geq 0$, thus $\mathcal{R}(\alpha)<\infty$. Besides, $\mathcal{R}$ is non-decreasing in $\alpha$ by construction and $\mathcal{R}(\alpha)\to 0$ as $\alpha\to0^+$ due to the unique limiting point of $c^*$ (Lemma~\ref{lem:merit}(M2)).

Below, we prove the convergence results presented in Theorem~\ref{thm:meta} in order:
\begin{enumerate}[label=\textup{(\roman*)}]
    \item Merit convergence: Let $\mathcal{N}_{ACC} = \{n: c_{n+1} = c_{n+1}^{ACC}\}$ denote the set of iterations at which an accelerated solution is accepted. Due to the restart mechanism with period $R$, the complementary set $\mathcal{N}_B = \{n: c_{n+1} = c_{n+1}^B\}$ is infinite. Then, exactly one of the two following scenarios occurs:
    \begin{enumerate}
        \item $|\mathcal{N}_{ACC}|<\infty$: In this case, there exists a finite number $N$, in any iteration $n \geq N$, the accelerated candidate is rejected. Thus, the algorithm falls back to the base method $c_{n+1}=B(c_n)$ and Condition (B2) ensures the convergence to the VI solution, which implies $\zeta(c_n)\to 0$ as $n\to \infty$. 

        \item $|\mathcal{N}_{ACC}|=\infty$: Enumerate the iteration indices in $\mathcal{N}_{ACC}$ in increasing order as $n_0<n_1<\dots$. We denote the iteration interval $[n_j+1, n_{j+1}]$ as the $(j+1)$-th epoch, 
        whose first iterate $c_{n_j+1}$ is the accepted accelerated solution and the remaining iterates $c_{n_j+2},\dots,c_{n_{j+1}}$ are base solutions.
        % within which all iterates are base solutions except for the last one. 

        At each accelerated step $n_j$, the safeguard checking ensures that $\zeta(c_{n_j+1})\leq \eta\rho_{n_j}< \tau \rho_{n_j}$ as $\eta <\tau$. Accordingly, the reference is updated as 
        \begin{align}
            \rho_{n_j+1} = \min\{\rho_{n_j}, \max\{\tau\rho_{n_j}, \zeta(c_{n_j+1})\}\} = \min\{\rho_{n_j},\tau\rho_{n_j}\} = \tau\rho_{n_j}. 
        \end{align}
        Since the reference update is non-increasing, we have $\rho_{n_{j+1}}\leq \rho_{n_j+1} = \tau \rho_{n_j}$. It is then easy to prove by induction that $\rho_{n_j}\leq \tau^j\rho_0$. Consequently, the merit at each accelerated step satisfies
        \begin{align}
            \zeta(c_{n_j+1})\leq \eta \rho_{n_j} \leq \eta \tau^j \rho_0 < \tau^{j+1} \rho_0 \to 0 \text{ as } j\to \infty. 
        \end{align}

        It remains to show the full sequence of merits converges, i.e., $\zeta(c_n)\to 0$ as $n\to\infty$. 
        Since $\zeta(c_{n_j+1})\le\tau^{j+1}\rho_0$, $c_{n_j+1}$ is inside the sublevel set $S_\alpha$ with $\alpha=\tau^{j+1}\rho_0$, and equivalently, $\normM{c_{n_j+1}-c^*} \leq \mathcal{R}(\tau^{j+1}\rho_0)$. Therefore, Condition (B1) with $c_0 = c_{n_j+1}$ implies that, for any iterate in the $(j+1)$-th epoch $n=n_j+1, \dots, n_{j+1}$, 
        $ \normM{c_n-c^*}\le C\,\normM{c_{n_j+1}-c^*} \leq C\mathcal{R}(\tau^{j+1}\rho_0)$.
        % Consider any iterate in the $(j+1)$-th epoch $n=n_j+1, \dots, n_{j+1}$, Condition (B1) with $c_0 = c_{n_j+1}$ implies that for $C \geq 1$
        % \begin{align}
        %     \zeta(c_n) = \zeta(c^B_n) \leq \sigma(\zeta(c_{n_j+1})) \leq \sigma(\tau^{j+1}\rho). 
        % \end{align}
        % \begin{align}
        %     \normM{c_n-c^*}\le C\,\normM{c_{n_j+1}-c^*} \leq C\mathcal{R}(\tau^{j+1}\rho_0).
        % \end{align}
        % The last inequality holds because the results $\zeta(c_{n_j+1})\le\tau^{j+1}\rho_0$ places $c_{n_j+1}$ in the sublevel set $\mathcal{S}_\alpha$ with $\alpha=\tau^{j+1}\rho_0$. Hence, $\normM{c_{n_j+1}-c^*}\le\mathcal{R}(\tau^{j+1}\rho_0)$. 
        % Since $\sigma(\alpha)\to 0$ as $\alpha\to 0^+$, the merit of base solutions also converges to zero as $j\to \infty$.
        Since there are infinitely many acceptances and $\mathcal{R}(\alpha) \rightarrow 0$ as $\alpha \rightarrow 0^+$, 
        $\mathcal{R}(\tau^{j+1}\rho_0)\to0$ as $j \rightarrow \infty$ and thus 
        $\normM{c_n-c^*}\to0$. Finally, 
        by the continuity of $\zeta$ and $\zeta(c^*)=0$ (Lemma~\ref{lem:merit}(M2) and (M3)), $\zeta(c_n)\to 0$ for the full sequence.
    \end{enumerate}

    \item Solution boundedness: Given the merit convergence, there must exist some $N$ such that $\zeta(c_n)\leq \zeta(c_N)$ for all $n$. Then, the boundedness of $\{c_n\}$ is induced from Property (M4) with $\alpha = \zeta(c_N)$

    \item Equilibrium convergence: Since the sequence $\{c_n\}$ is bounded with limiting point $\bar{c}$, $\zeta$ is continuous (Lemma~\ref{lem:merit}(M3)), and
    $\zeta(c_n)\to 0$ as $n\to\infty$ (merit convergence),
    we have $\zeta(\bar c) = 0\Leftrightarrow \bar c = c^*$ is the unique PUME due to Property (M2) in Lemma~\ref{lem:merit}
    % , the limiting point of the bounded sequence $\{c_n\}$ is the unique PUME $c^*$
    . 
\end{enumerate}

\subsection{Verification of meta-algorithm convergence conditions}\label{app:B1_proof}
% \subsection{Verification of Condition \textup{(B1)}}\label{app:B1_proof}

Since the global convergence (B2) holds for both ST and aGRAAL under the conditions of Theorem~\ref{thm:uniqueness} \citep{solodov1996modified,malitsky2020forward}, we only need to verify (B1) in Theorem~\ref{thm:meta}. 

\begin{corollary}[Verification of ST method]
    The ST method with update rules Eqs.~\eqref{eq:ST_fst}--\eqref{eq:ST_lst} satisfies (B1) in Theorem~\ref{thm:meta}. 
\end{corollary}
\begin{proof}
    With continuous monotone operator $E$ and closed convex set $\Omega$, the ST method implemented in Eqs.~\eqref{eq:ST_fst}--\eqref{eq:ST_lst} is a special variant of Algorithm~3.2 of \citet{solodov1996modified} with preconditioning matrix $\mathbb{I}$. The proof of the ST method \citep[Theorem~3.2, Eq.~(3.11)]{solodov1996modified} gives
    \begin{align}
        \|c_{n+1}^{B} - c^*\|^2
        \le
        \|c_n^{B} - c^*\|^2
        -
        \theta(2-\theta)(1 - \delta)^2
        \frac{\|c_n-\hat c_n\|^4}
             {\|(c_n-\hat c_n)-\beta_n E(c_n)+\beta_n E(\hat c_n)\|^2}.
    \end{align}
    Since the second term is always non-negative, we have $\|c_{n+1}^{B} - c^*\|^2\leq \|c_n^{B} - c^*\|^2,\forall n$. It then yields $\|c_{n+1}^{B} - c^*\|\leq \|c_0 - c^*\|$, i.e., Condition (B1) holds with $C=1$.
\end{proof}

% \begin{corollary}[Verification of aGRAAL with fixed preconditioning matrix]
%     The aGRAAL method with update rules Eqs.~\eqref{eq:aGRAAL_fst}--\eqref{eq:aGRAAL_lst} using a fixed matrix $M$ over acceleration steps satisfies (B1) in Theorem~\ref{thm:meta}. 
% \end{corollary}
% \begin{proof}
%     This corollary is proved via an intermediate result. Specifically, Lemma~\ref{lem:distance_bound} establishes a uniform distance bound, with which Proposition~\ref{prop:B1} can be moderately adjusted to show aGRAAL satisfies (B1). For ease of reading, we only present the lemma below while deferring its proofs to Section~\ref{sec:distance_bound_proof}. 
% \end{proof}
\begin{corollary}[Verification of aGRAAL with fixed preconditioning matrix]\label{lem:distance_bound}
    % For any initial solution $c_0\in\Omega$, let $c_1=c_0$ and $\varphi$ be the golden ratio. Then, the aGRAAL method generates iterates satisfying
    % \begin{align}
    %     ||c^B_n - c^*||\leq C_\varphi \sqrt{\text{cond}(M)}||c_0-c^*||, 
    % \end{align}
    % where $\text{cond}(M)=M_{\max}/M_{\min}$ is the condition number ($M_{\max},M_{\min}$ are the maximum and minimum value in $M$, respectively), and $C_\varphi=(\varphi+1)/(\varphi-1)$. 
    The aGRAAL method with update rules Eqs.~\eqref{eq:aGRAAL_fst}--\eqref{eq:aGRAAL_lst} using a fixed matrix $M$ over acceleration steps satisfies (B1) in Theorem~\ref{thm:meta}.
\end{corollary}

\begin{proof}
     Since $E$ is monotone,
     % and \kz{let feasible set enters through indicator functions}, 
    aGRAAL~\eqref{eq:aGRAAL_fst}--\eqref{eq:aGRAAL_lst} coincides with the metric golden-ratio method of \citet[Algorithm~2]{malitsky2020golden}. Its convergence analysis \citep[Eq.~(70)]{malitsky2020golden} shows that, for any solution $c^*$, there is a non-decreasing energy $W_n$ over iterations:
    \begin{align}
        W_n \coloneqq \frac{\varphi}{\varphi-1}\normM{\hat c_n - c^*}^2 + \frac{\varphi\lambda_{n-1}}{2\lambda_{n-2}}\normM{c_n - c_{n-1}}^2
    \end{align}
    Initializing $c_1 = c_0$ and $\hat c_0 = c_0$ , Eq.~\eqref{eq:aGRAAL_lst} gives $\hat c_1 = c_0$ and $\normM{c_1-c_0}=0$. Accordingly, $W_1 = \frac{\varphi}{\varphi-1}\normM{c_0 - c^*}^2$ and
    \begin{align}
        \frac{\varphi}{\varphi-1}\normM{\hat c_n - c^*}^2 \le W_n \le W_1 = \frac{\varphi}{\varphi-1}\normM{c_0 - c^*}^2,
        \quad \Rightarrow \quad 
        \normM{\hat c_n - c^*}\le\normM{c_0 - c^*}.
    \end{align}
    Besides, rearranging Eq.~\eqref{eq:aGRAAL_lst} yields $c_n = \frac{\varphi}{\varphi-1}\hat c_n - \frac{1}{\varphi-1}\hat c_{n-1}$. Then, by the triangle inequality,
    \begin{align}
        \normM{c_n - c^*}
        \le \frac{\varphi}{\varphi-1}\normM{\hat c_n - c^*} + \frac{1}{\varphi-1}\normM{\hat c_{n-1} - c^*}\leq C_\varphi\,\normM{c_0 - c^*},
    \end{align}
    where $C_\varphi = \frac{\varphi+1}{\varphi-1}$ serves as the constant in Condition~(B1).
\end{proof}

    % Lemma~\ref{lem:distance_bound} directly implies $||c^B_n - c^*||\leq C_\varphi \sqrt{\text{cond}(M)}||c_0-c^*||\leq C_\varphi \sqrt{\text{cond}(M)}\mathcal{R}(\zeta(c_0))$. 
    % Hence, by replacing $\mathcal{R}$ in Proposition~\ref{prop:B1} with a scaled sublevel radius $C_\varphi \sqrt{\text{cond}(M)}\mathcal{R}$, we may follow the same proof to show that aGRAAL satisfies (B1). 

We note that the same result holds when the preconditioning matrix $M$ varies over acceleration steps, subject to additional regularity conditions. Since this is far away from the main focus of this study, we do not detail the proof for the general aGRAAL method.

\section{Surplus family}\label{app:surplus-families}

In Section~\ref{sec:surplus}, we first define the surplus $H_s$ and then construct the perturbation $H_s^*$ as its convex conjugate. The converse also holds, and the conditions for an automatic construction of admissible surplus from a perturbation are presented in the following proposition. 

% All surplus functions used in this paper share a common structure: $H_s$ can be expressed as the convex conjugate of a perturbation $F_s$ on the simplex, effectively the converse of Section~\ref{sec:FY}.
% The following result shows that this construction automatically delivers an admissible surplus function from any strictly convex perturbation.

\begin{proposition}[Construction of surplus function]\label{prop:conjugate-construction-body}
Let $H^*_s:\Delta_s \to \R$ be a finite, continuous, and strictly convex perturbation function. Then, its convex conjugate given by
\begin{align}\label{eq:perturb-conjugate}
    H_s(Q) \coloneqq \max_{\pi \in \Delta_s} \bigl\{ \pi\tran Q - H^*_s(\pi) \bigr\}, \qquad Q \in \R^{|\mathcal{A}_s|}
\end{align}
satisfies Standing Assumption~\ref{asm:surplus}. Besides, the maximizer of \eqref{eq:perturb-conjugate} is unique for every $Q$ and coincides with $\nabla H_s(Q)$, that is,
\begin{align}
    \nabla H_s(Q) = \argmax_{\pi \in \Delta_s} \bigl\{ \pi\tran Q - H^*_s(\pi) \bigr\} \in \Delta_s.
\end{align}

% Define
% \[
%   H_s(Q) \coloneqq \max_{\pi \in \Delta_s} \bigl\{ \pi\tran Q - F_s(\pi) \bigr\}, \qquad Q \in \R^{|\mathcal{A}_s|}.
% \]
% Then $H_s$ satisfies \ref{A1}--\ref{A3}, the maximizer is unique for every $Q$, and
% \[
%   \nabla H_s(Q) = \argmax_{\pi \in \Delta_s} \bigl\{ \pi\tran Q - F_s(\pi) \bigr\} \in \Delta_s.
% \]
\end{proposition}

\begin{proof}
We prove that $H_s$ constructed from Eq.~\eqref{eq:perturb-conjugate} satisfies each condition in Standing Assumption~\ref{asm:surplus}, along with the uniqueness of maximizer and its correspondence to $\nabla H_s$, as follows: 
\begin{enumerate}[label=\textup{(\roman*)}]
    \item Convexity \ref{asm:surplus_convex}: As a pointwise supremum of affine function of $Q$, $H_s$ is convex as per~\citet[Theorem~5.5]{Rockafellar+1970}. 
    
    \item Translation equivalence \ref{asm:surplus_equiv}: For each $\pi \in \Delta_s$ such that $\pi\tran\mathbf{1} = 1$,
    \begin{align}
        H_s(Q + \alpha\mathbf{1})
        = \max_{\pi \in \Delta_s} \bigl\{\pi\tran Q + \alpha - H^*_s(\pi)\bigr\}
        = \alpha + H_s(Q), \qquad \alpha \in \R.
    \end{align}
    
    \item Unique maximizer: Since $H^*_s$ is strictly convex, the objective $\pi\tran Q - H^*_s(\pi)$ is strictly concave. Hence, the maximizer of \eqref{eq:perturb-conjugate} is unique. 
    
    \item Differentiability, simplex gradient \ref{asm:surplus_grad}, and maximizer correspondence: By \citet[Theorem~23.5]{Rockafellar+1970}, $\pi \in \partial H_s(Q)$ if and only if $\pi$ attains the supremum defining $H_s(Q)$. Since the maximizer $\pi_s(Q)$ is unique, $\partial H_s(Q) = \{\pi_s(Q)\}$ is a singleton. As per \citet[Theorem~25.1]{Rockafellar+1970}, a convex function whose subdifferential is a singleton everywhere is differentiable. Hence, $\nabla H_s(Q) = \pi_s(Q) \in \Delta_s$.

    \item Smoothness (\ref{asm:surplus_convex}): 
    Since $H_s(Q)$ is the maximum value of a finite, continuous function~\eqref{eq:perturb-conjugate} over the compact simplex, it is finite.
    The smoothness is then directly implied from the differentiability due to \citet[Corollary~25.5.1]{Rockafellar+1970}. 
    % A finite convex function that is differentiable on all of $\R^{|\mathcal{A}_s|}$ is smooth due to~\citet[Corollary~25.5.1]{Rockafellar+1970}.
\end{enumerate}
\end{proof}

Below, we discuss the properties of each pair of surplus and choice map presented in Table~\ref{tab:surplus}. The properties stated in Standing Assumption~\ref{asm:surplus} are omitted because it is easy to verify that the negative Tsallis entropy in Eq.~\eqref{eq:tsallis} satisfies the conditions stated in Proposition~\ref{prop:conjugate-construction-body}.

\subsection{Logit/Softmax}
The logit/softmax choice map Eq.~\eqref{eq:softmax} can is easily shown to be interior and smooth. The Hessian of surplus is derived as 
\begin{align}
    \nabla^2 H_s(Q) = \frac{1}{\mu_s}\Big(\text{diag}(\nabla H_s(Q)) - \nabla H_s(Q)\nabla H_s(Q)\tran\Big),
\end{align}
which is smooth in $Q$. Thus, $H_s$ defined in Eq.~\eqref{eq:log-sum-exp} satisfies Assumption~\ref{asm:surplus_extra} and yields
a Lipschitz choice map.

\subsection{Sparsemax}
The sparsemax choice map Eq.~\eqref{eq:sparsemax} is piecewise affine and produces sparse/corner solution corresponding to the active set. Hence, it is not smooth and Assumption~\ref{asm:surplus_extra} does not hold either~\citep{martins2016softmax}. It nonetheless remains Lipschitz, since projection onto the simplex is a nonexpansive map and Lipschitz of constant 1.

\subsection{$\alpha$-entmax}
For $\alpha\in(1,2)$, the KKT condition of the maximization problem in Eq.~\eqref{eq:perturb-conjugate} with the negative Tsallis perturbation Eq.~\eqref{eq:tsallis} yields a thresholded choice map with each element
\begin{align}\label{eq:entmax_map}
    \pi_a =\left[(\alpha-1)\left(\frac{Q_a}{\mu_s}-\tau\right)\right]_+^{\frac{1}{\alpha-1}},
\end{align}
where $[\,\cdot\,]_+^p=(\max\{0,\cdot\})^p$, and the threshold $\tau$ is the Lagrangian multiplier associated with constraint $\sum_a\pi_a=1$ (with minor abuse of notation). The constraint is then rewritten as
\begin{align}\label{eq:entmax_normalization}
    F(Q, \tau) = \sum_{a\in\mathcal{A}_s}\left[(\alpha-1)\left(\frac{Q_a}{\mu_s}-\tau\right)\right]_+^{\frac{1}{\alpha-1}}-1 = 0.
\end{align}
Since the power $1/(\alpha-1)>1$, the optimal choice probability $\pi^*_a$ is $C^1$ in the difference $Q_a/\mu_s-\tau$ and strictly decreases in $\tau$ at active actions ($\pi_a>0$). Accordingly, $\partial F(Q,\tau)/\partial\tau \neq 0$ and as per the implicit function theorem, $\tau(Q)$ is $C^1$. Plugging $\tau(Q)$ back to Eq.~\eqref{eq:entmax_map} leads to an optimal choice map that is $C^1$ in $Q$. Given its equivalence to $\nabla H_s$, $H_s$ is $C^2$ and satisfies Assumption~\ref{asm:surplus_extra} and generates a Lipschitz choice map.

% For $\alpha\in(1,2)$, set $p=1/(\alpha-1)>1$.
% Writing the scaled Q-value $z_a = Q_a/\mu_s$, the KKT condition of the conjugate problem~\eqref{eq:perturb-conjugate} with the negative Tsallis perturbation~\eqref{eq:tsallis} yields a thresholded choice map
% \begin{align}\label{eq:entmax_map}
%     [\nabla H_s(Q)]_a \;=\; \bigl[(\alpha-1)\,(z_a-\tau)\bigr]_+^{\,p},
%     \qquad [\,\cdot\,]_+=\max\{0,\cdot\},
% \end{align}
% in which the scalar $\tau$ (the multiplier of the constraint $\sum_a\pi_a=1$) is chosen such that the normalization
% \begin{align}\label{eq:entmax_normalization}
%     g(\tau)\;\coloneqq\;\sum_{a\in\mathcal{A}_s}\bigl[(\alpha-1)\,(z_a-\tau)\bigr]_+^{\,p}-1\;=\;0 .
% \end{align}

% Since $p> 1$, the map $x\mapsto[(\alpha-1)x]_+^{p}$ is $C^1$ on $\R$, so $g$ is $C^1$ in $(Q,\tau)$. 
% In addition, since $g$ is strictly decreasing in $\tau$ ($\partial_\tau g<0$) for active actions, the implicit function theorem delivers a $C^1$ threshold $\tau(Q)$. Substituting it into~\eqref{eq:entmax_map} shows $\nabla H_s$ is $C^1$, hence $H_s$ is $C^2$. Assumption~\ref{asm:surplus_extra} thus holds for every $\alpha\in(1,2)$.

\subsubsection{Solution method of $\alpha$-entmax choice map}
Solving the choice map Eq.~\eqref{eq:entmax_map} is equivalent to a root finding problem of Eq.~\eqref{eq:entmax_normalization}. While depending on the value of $\alpha$, it may or may not yield a closed-form solution. Below, we detail the two cases $\alpha=1.2$ and $\alpha=1.5$ used in the experiments. 

\begin{itemize}
    \item \textbf{$\alpha=1.5$}: In this case, Eq.~\eqref{eq:entmax_map} reduces to 
    $\pi_a =\left[\left(Q_a/\mu_s-\tau\right)/2\right]_+^2$. Let $\mathcal{A}_s^+$ be the set of active actions and $z_a = Q_a/\mu_s$, then Eq.~\eqref{eq:entmax_normalization} has closed-form roots
    \begin{align}\label{eq:entmax15_tau}
        \tau = \mu_\tau - \sqrt{\frac{4}{|\mathcal{A}_s^+|} - \sigma_\tau},
    \end{align}
    where $\mu_\tau = \left(\sum_{a\in \mathcal{A}_s^+} z_a\right)/|\mathcal{A}_s^+|$ and $\sigma_\tau = \left(\sum_{a\in \mathcal{A}_s^+} z_a^2\right)/|\mathcal{A}_s^+| - \mu_\tau^2$. 

    The remaining task is to find the active action set $\mathcal{A}_s^+$, which can be efficiently done via Algorithm~\ref{alg:entmax15}~\citep{duchi2008efficient, peters2019sparse}

    \item \textbf{$\alpha=1.2$}: In this case, Eq.~\eqref{eq:entmax_map} has no closed-form solution but the monotone property of $F(Q,\tau)$ makes the root finding easily done via bisection search~\citep{blondel2020learning}. 
    Let $z_{\max} = \max_a z_a$ with $z_a=Q_a/\mu_s$. The lower bound can be easily found as $F(Q, z_{\max}) = -1 <0$. While the upper bound is constructed as 
    $F(Q, z_{\max}-\frac{1}{\alpha-1}) = \sum_a [(\alpha-1)(z_a - z_{\max}) + 1]_+^{1/(\alpha-1)} - 1 = \sum_{a: z_a < z_{\max}} [(\alpha-1)(z_a - z_{\max}) + 1]_+^{1/(\alpha-1)} + \sum_{a: z_a = z_{\max}}1 - 1 \geq 0$. The searching method is summarized in Algorithm~\ref{alg:entmax12}, and it is also the general solution approach to the $\alpha$-entmax choice map. 
\end{itemize}

% \paragraph{Computation for $\alpha$-entmax.}
% Computing the choice map~\eqref{eq:entmax_map} reduces to a single root finding of problem~\eqref{eq:entmax_normalization}. 
% We give the two cases ($\alpha=1.2, 1.5$) used in our experiments.

% \noindent \textbf{\textit{Closed form for $\alpha=1.5$.}}
% Here, Eq.~\eqref{eq:entmax_map} reads $\pi_a=\bigl[(z_a-\tau)/2\bigr]_+^2$. 
% For the set of active actions $\mathcal{A}_{s+}$, the normalization~\eqref{eq:entmax_normalization} becomes the quadratic $\sum_{a\in\mathcal{A}_{s+}}(z_a-\tau)^2=4$, whose root can be found by
% \begin{align}\label{eq:entmax15_tau}
%     \tau \;=\; \frac{1}{|\mathcal{A}_{s+}|}\!\!\sum_{a\in\mathcal{A}_{s+}}\!\!z_a \;-\; \sqrt{\frac{1}{|\mathcal{A}_{s+}|}-\mathrm{Var}},
%     \qquad
%     \bar z=\frac{1}{|\mathcal{A}_{s+}|}\!\!\sum_{a\in\mathcal{A}_{s+}}\!\!z_a,\quad
%     \mathrm{Var}=\frac{1}{|\mathcal{A}_{s+}|}\!\!\sum_{a\in\mathcal{A}_{s+}}\!\!z_a^2-\bar z^{2}.
% \end{align}
% The key is then to find the set of active actions $\mathcal{A}_{s+}$. 
% Thanks to the quadratic structure for $\alpha=1.5$, this can be done efficiently by a single descending sort followed by a scan~\citep{duchi2008efficient, peters2019sparse}, as summarized in Algorithm~\ref{alg:entmax15}.

\begin{algorithm}[h]
\caption{Exact $\alpha$-entmax choice map with $\alpha=1.5$}\label{alg:entmax15}
\begin{algorithmic}[1]
\REQUIRE $Q\in\R^{n}$, scale $\mu_s>0$
\STATE set $z \leftarrow Q/\mu_s$, $n=|\mathcal{A}_s|$
\STATE sort $z$ in descending order $z_{(1)}\ge\cdots\ge z_{(n)}$; set $S_1\leftarrow 0$, $S_2\leftarrow 0$
\FOR{$k = 1, \ldots, n$}
    \STATE $S_1\leftarrow S_1+z_{(k)}$, \quad $S_2\leftarrow S_2+z_{(k)}^2$ \hfill (cumulative sums over the top $k$ scores)
    \STATE $\tau \leftarrow \bigl(S_1-\sqrt{S_1^{2}-k\,(S_2-4)}\bigr)/k$ \hfill (root in Eq.~\eqref{eq:entmax15_tau} for the candidate active set)
    \IF{$k=n$ or $z_{(k+1)}\le \tau$}
        \STATE \textbf{break} \hfill (threshold $z_{(k)}>\tau\ge z_{(k+1)}$)
    \ENDIF
\ENDFOR
\RETURN $\pi$ with $\pi_a=\bigl[(z_a-\tau)/2\bigr]_+^{2}$ for all $a\in\mathcal{A}_s$
\end{algorithmic}
\end{algorithm}

% \noindent \textbf{\textit{Bisection for $\alpha=1.2$.}}
% For $\alpha=1.2$ (so $p=5$) no closed form is available, but the monotone structure of $g$ makes the root easy to find. 
% Because $g$ is strictly decreasing with $g(z_{\max})=-1<0$ and $g\bigl(z_{\max}-\tfrac{1}{\alpha-1}\bigr)\ge 0$, the threshold satisfies $\tau\in\bigl[z_{\max}-\tfrac{1}{\alpha-1},\,z_{\max}\bigr]$, where $z_{\max}=\max_a z_a$. Bisection on this range converges to any tolerance (Algorithm~\ref{alg:entmax12}), which is the general routine for evaluating the choice map at any $\alpha\in(1,2)$~\citep{blondel2020learning}.

\begin{algorithm}[h]
\caption{Bisection search for $\alpha$-entmax choice map}\label{alg:entmax12}
\begin{algorithmic}[1]
\REQUIRE $Q\in\R^{n}$, scale $\mu_s>0$, $\alpha\in(1,2)$, tolerance $\epsilon>0$
\STATE $z \leftarrow Q/\mu_s$, \quad $p \leftarrow 1/(\alpha-1)$, \quad $z_{\max}\leftarrow \max_a z_a$
\STATE $\underline{\tau}\leftarrow z_{\max}-1/(\alpha-1)$, \quad $\overline{\tau}\leftarrow z_{\max}$ 
\WHILE{$\overline{\tau}-\underline{\tau}>\epsilon$}
    \STATE $\tau\leftarrow(\underline{\tau}+\overline{\tau})/2$
    \STATE $g\leftarrow\sum_{a}\bigl[(\alpha-1)(z_a-\tau)\bigr]_+^{\,p}-1$
    \IF{$g>0$}
        \STATE $\underline{\tau}\leftarrow\tau$ 
    \ELSE
        \STATE $\overline{\tau}\leftarrow\tau$
    \ENDIF
\ENDWHILE
\RETURN $\pi$ with $\pi_a = \bigl[(\alpha-1)(z_a-\tau)\bigr]_+^{\,p}$ for all $a\in\mathcal{A}_s$
\end{algorithmic}
\end{algorithm}

\section{Experiment setup and supplementary results}\label{app:solver_params}

This appendix details the parameters and other setups used in the numerical experiments in Section~\ref{sec:experiments}. It follows with additional results of meta-algorithm convergence in the experiments on benchmark networks. 

% This appendix records the computational settings shared across the experiments in Section~\ref{sec:experiments} and describes the synthetic grid networks used in the sensitivity analysis.

\subsection{Default parameters}
Table~\ref{tab:param} reports the default values used in the solution framework. 
When NRL is applied, the scale parameter $\mu_s$ is independently and uniformly sampled from [0.5,2] for each state $s$, while for logit model, a uniform $\mu_s=1$ is applied.

\begin{table}[htb]
\centering
\caption{Default values of algorithm parameters.}
\small
\begin{threeparttable}
\begin{tabular}{lcc}
\toprule
\textbf{Description} & \textbf{Notation} & \textbf{Value}  \\
\midrule
MPI evaluation depth & $m$ & 10 \\
Inner gap threshold$^\dagger$ & $\varepsilon_\text{in}$ & $10^{-7}$ \\
Outer gap threshold & $\varepsilon_\text{out}$ & $10^{-5}$ \\
Maximum iteration & & $2\times 10^4$ \\
Safeguard factor & $\eta$ & 0.9 \\
Decay rate & $\tau$ & 0.999 \\
Restart period & $R$ & 20 \\
\midrule
\textbf{Base solver: ST} & & \\
Line search parameter & $\delta$ & 0.5\\
Correction scale parameter & $\theta$ & 1.5 \\
\midrule
\textbf{Base solver: aGRAAL} & & \\
Preconditioning matrix & $M$  & $\text{diag}((\nabla_\ell z_\ell(c_\ell)_{\forall \ell})$ \\
Initial step size & $\lambda_0$ & 0.05 \\
Maximum step size & $\bar{\lambda}$ & 1 \\
Golden ratio & $\varphi$ & 1.618 \\
\midrule
\textbf{Acceleration oracle} & & \\
Memory depth & $m_{ACC}$ & 10 \\
Regularization parameter$^*$ & $\eta_\text{reg}$ & $10^{-6}/10^{-4}$\\
\bottomrule
\end{tabular}
\begin{tablenotes}
\footnotesize 
\item[$\dagger$:] A smaller gap $\varepsilon_\text{in} = 10^{-12}$ is used in Section~\ref{sec:exp_sensitivity_mpi}.
\item[$*$:] $\eta_\text{reg}=10^{-6}$ is used for AA1 while $10^{-4}$ is used for NGMRES. 
\end{tablenotes}
\end{threeparttable}
\label{tab:param}
\end{table}

\subsection{Synthetic grid network generation}
As illustrated in Figure~\ref{fig:grid_illustration}, each grid network of demand node density $k$ has $4k+1$ nodes along each side. Each node is connected to its cardinal neighbors via bidirectional links with common BPR function parameters $t_0=1$ and $\kappa=2500$.

The demand between each pair of demand nodes is generated from a gravity model:
\begin{align}
    q_{od} \;=\; q \; \frac{\exp(-0.25 \mathrm{SPT}(o,d))}
    {\sum_{d' \neq o} \exp(-0.25 \mathrm{SPT}(o,d'))},
\end{align}
where $q$ is the demand per origin node, and $\mathrm{SPT}(o,d)$ denotes the free-flow shortest-path cost. 

\subsection{Convergence on benchmark networks}

Figure~\ref{fig:solver_traces} plots the outer gap over iterations for the six solver configurations reported in Table~\ref{tab:solver_comparison}. Figures~\ref{fig:robustness_siouxfalls}--\ref{fig:robustness_chicago} plot the gap trajectories in the robustness experiments on the demand and supply models across the three benchmark networks discussed in Section~\ref{sec:exp_robustness}.

\begin{figure}[h]
\centering
\includegraphics[width=0.6\textwidth]{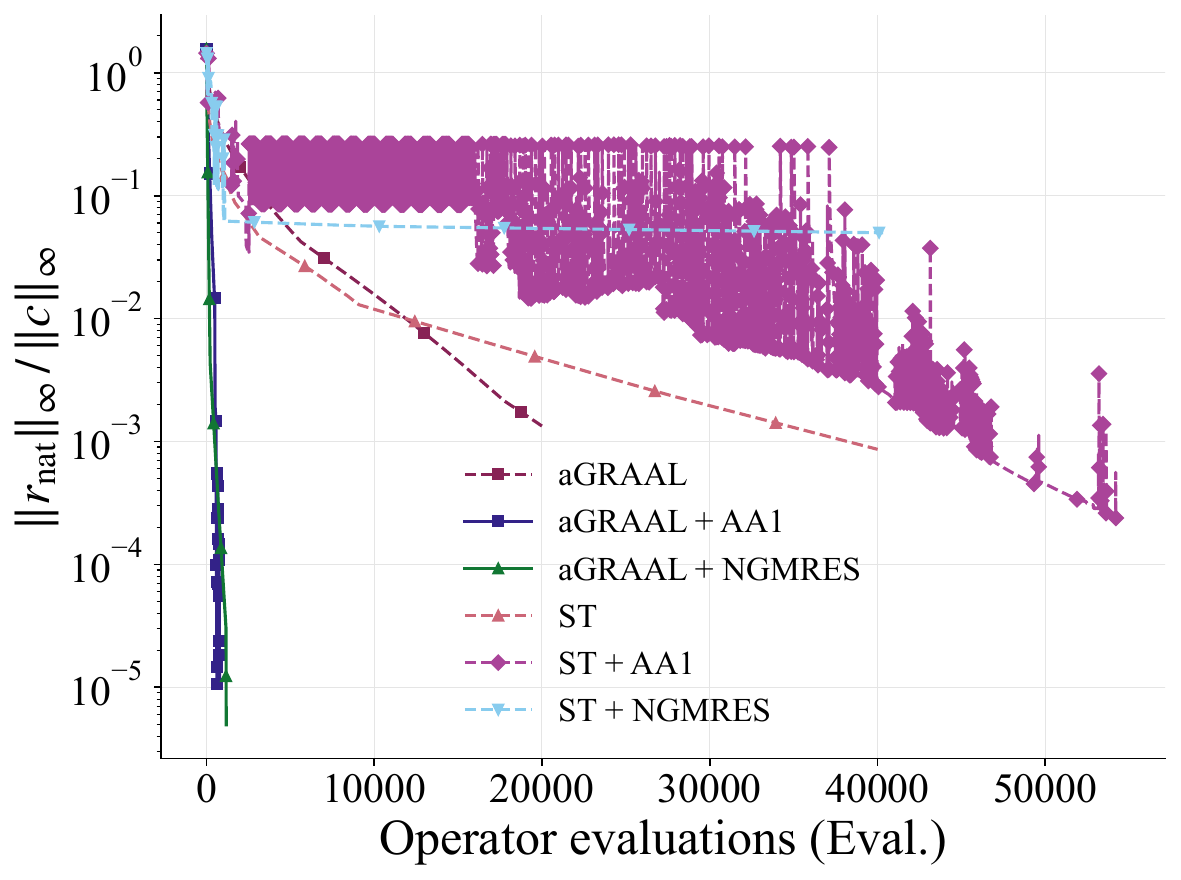}
\caption{Outer gap trajectories for different solver configurations.}
\label{fig:solver_traces}
\end{figure}

\begin{figure}[h]
\centering
\includegraphics[width=1.0\textwidth]{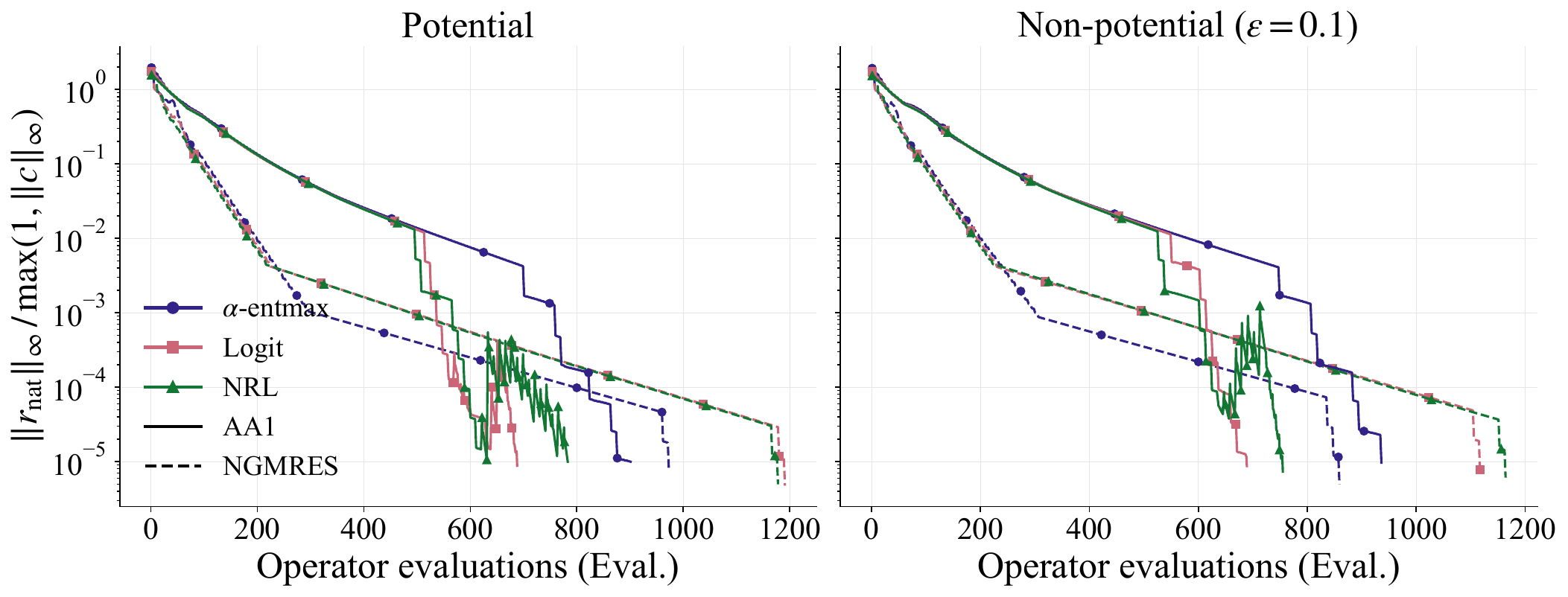}
\caption{Iteration trajectories on Sioux Falls.}
\label{fig:robustness_siouxfalls}
\end{figure}

\begin{figure}[h]
\centering
\includegraphics[width=\textwidth]{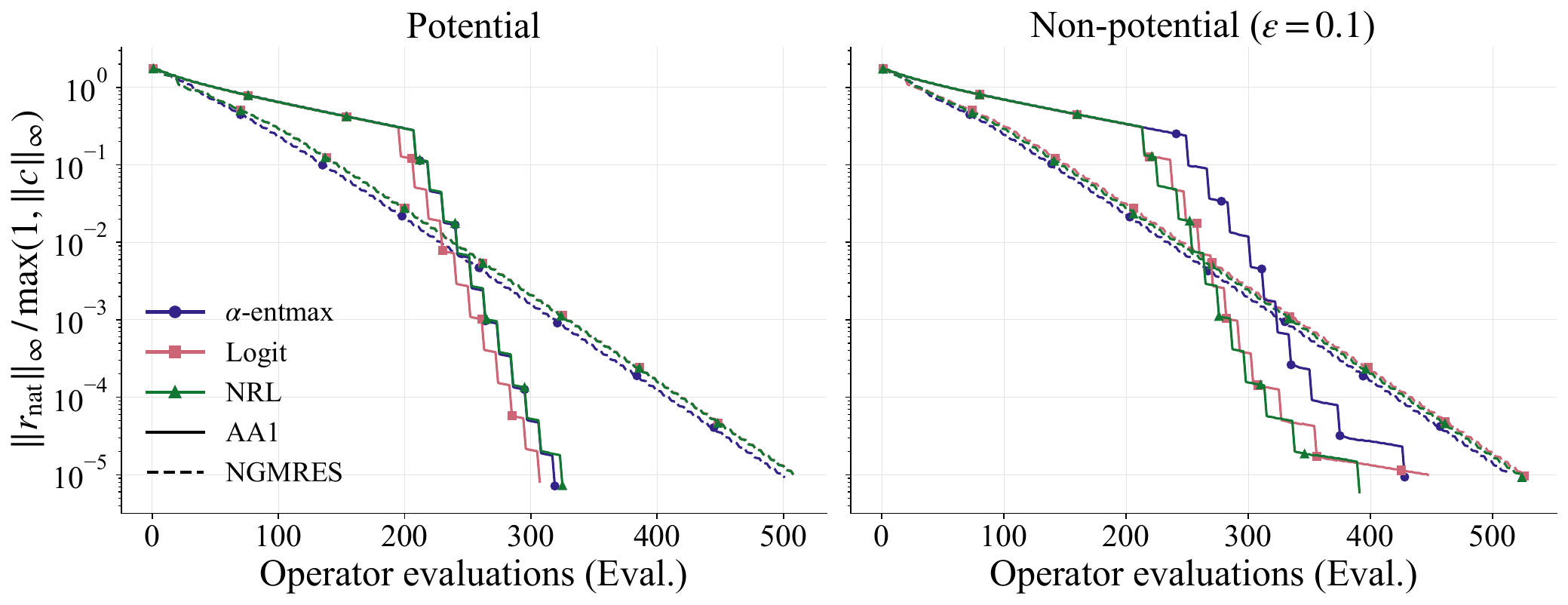}
\caption{Convergence trajectories on Anaheim.}
\label{fig:robustness_anaheim}
\end{figure}

\begin{figure}[h]
\centering
\includegraphics[width=\textwidth]{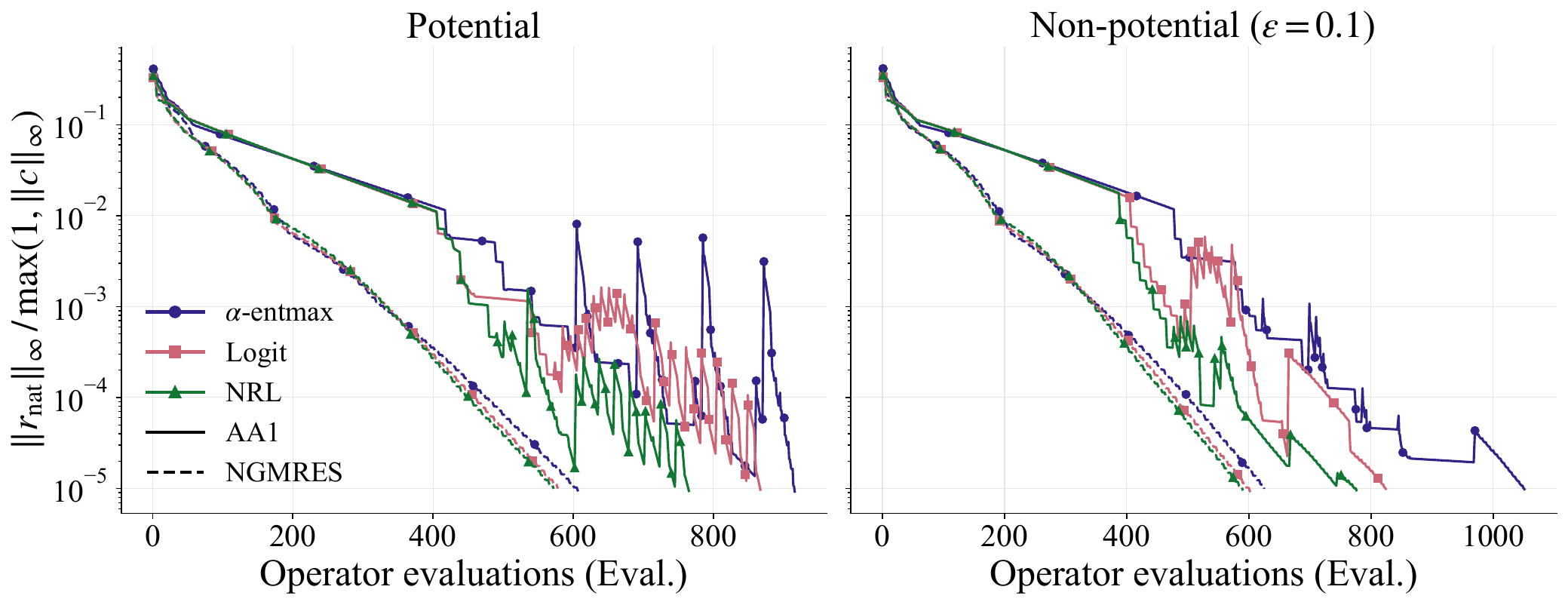}
\caption{Iteration trajectories on Chicago~Sketch.}
\label{fig:robustness_chicago}
\end{figure}

\end{document}